\documentclass[a4paper,11pt,reqno]{amsart}

\usepackage{amssymb}
\usepackage{latexsym}
\usepackage{amsmath}
\usepackage{mathrsfs}
\usepackage{euscript}
\usepackage{amsthm}
\usepackage{upgreek}
\usepackage{cite}
\usepackage{amscd}
\usepackage{xcolor}
\usepackage{tikz}
\usepackage{calc}

\usepackage{diagbox}

\usepackage{comment}

\usepackage{color}
\definecolor{ceruleanblue}{rgb}
{0.16, 0.32, 0.75}

\newcommand{\scal}[2]{\langle #1,#2\rangle}

\newcommand{\rr}[1]{\mathbf R^{#1}}
\newcommand{\zz}[1]{\mathbf Z^{#1}}

\newcommand{\nn}[1]{\mathbf N^{#1}}

\newcommand{\nm}[2]{\Vert #1\Vert _{#2}}
\newcommand{\NM}[2]{\left \Vert #1\right \Vert _{#2}}
\newcommand{\nmm}[1]{\Vert #1\Vert }

\newcommand{\op}{\operatorname{Op}}

\newcommand{\sets}[2]{\{ \, #1\, ;\, #2\, \} }

\newcommand{\ep}{\varepsilon}
\newcommand{\fy}{\varphi}

\newcommand{\cdo}{\, \cdot \, }

\newcommand{\supp}{\operatorname{supp}}

\newcommand{\esssup}{\operatorname{ess\, sup}}

\newcommand{\tp}{\operatorname{Tp}}
\newcommand{\vrum}{\vspace{0.1cm}}

\newcommand{\wpr}{{\text{\footnotesize $\#$}}}

\newcommand{\setm}{\setminus 0}

\newcommand{\GL}{\mathbf{M}}

\newcommand{\sfW}{\mathsf{W}}

\newcommand{\maclD}{\mathcal D}
\newcommand{\maclE}{\mathcal E}

\newcommand{\maclI}{\mathcal I}

\newcommand{\maclL}{\mathcal L}
\newcommand{\maclM}{\mathcal M}

\newcommand{\maclS}{\mathcal S}

\newcommand{\mascB}{\mathscr B}

\newcommand{\mascE}{\mathscr E}
\newcommand{\mascF}{\mathscr F}

\newcommand{\mascP}{\mathscr P}

\newcommand{\mascS }{\mathscr S}
\newcommand{\essup}{\operatorname{ess\, sup}}

\newcommand{\fka}{\mathfrak a}
\newcommand{\fkb}{\mathfrak b}

\numberwithin{equation}{section}

\newtheorem{thm}{Theorem}
\numberwithin{thm}{section}

\newcommand{\rubrik}{}
\newtheorem{prop}[thm]{Proposition}
\newtheorem{cor}[thm]{Corollary}
\newtheorem{lemma}[thm]{Lemma}

\theoremstyle{definition}

\newtheorem{defn}[thm]{Definition}
\newtheorem{example}[thm]{Example}

\theoremstyle{remark}

\newtheorem{rem}[thm]{Remark}              

\definecolor{darkred}{rgb}{0.8,0,0}

\author[E. Nabizadeh Morsalfard]{Elmira Nabizadeh Morsalfard}

\address{Department of Mathematics,
Linn{\ae}us University,
V{\"a}xj{\"o}, Sweden}

\email{elmira.nabizadehmorsalfard@lnu.se}

\author[C. Pfeuffer]{Christine Pfeuffer}

\address{Department of Mathematics, Martin-Luther-Universit\"at Halle-Wittenberg, Halle (Saale),  Germany}

\email{christine.pfeuffer@mathematik.uni-halle.de}

\author[N. Teofanov]{Nenad Teofanov}

\address{Department of Mathematics and Informatics,
University of Novi Sad, Novi Sad, Serbia}

\email{nenad.teofanov@dmi.uns.ac.rs}

\author[J. Toft]{Joachim Toft}

\address{Department of Mathematics,
Linn{\ae}us University,
V{\"a}xj{\"o}, Sweden}

\email{joachim.toft@lnu.se}

\title[Compactness
for operators on modulation spaces]
{Compactness
for pseudo-differential and
Toeplitz operators on modulation spaces}


\subjclass[2020]{Primary: Secondary:} 

\begin{document}

\begin{abstract}
We deduce various norm equivalences, 
and convolution estimates for
the modulation space
$M^{\sharp ,q}_{(\omega )}$
consisting of all $f\in
M^{\infty ,q}_{(\omega )}$
such that $|V_\phi f \cdot \omega |$
satisfies a mild vanishing condition
at infinity. We prove that
$M^{\sharp ,q}_{(\omega )}$ is the
completion of the Gelfand-Shilov space
$\Sigma _1$ under
the $M^{\infty ,q}_{(\omega )}$ norm.
We use these results to deduce
compactness for $\Psi$DO $\op (\fka )$,
with $\fka \in M^{\sharp ,q}_{(\omega )}$,
$0<q\le 1$, when acting on a broad family
of modulation spaces.
\end{abstract}

\maketitle

\par

\section{Introduction}\label{sec0}

\par

In the paper we analyze
subspaces of modulation spaces
of the form
$M^{\infty ,q}_{(\omega )}(\rr d)$,
where the short-time Fourier
transforms of involved functions and
distributions should obey certain mild
vanishing properties at infinity.
(See \cite{Ho1}
or Section \ref{sec1} for notations.)
For polynomial moderate weights we show that
these subspaces, denoted by
$M^{\sharp ,q}_{(\omega )}(\rr d)$,
are exactly the completions of 
the Schwartz space with respect
to the $M^{\infty ,q}_{(\omega )}$ 
norm. For
general weights, we show that the same
holds true after the Schwartz spaces are
replaced by suitable Fourier invariant
Gelfand-Shilov spaces, which are densely
embedded in the Schwartz spaces.


\par

In the second part of the paper, we
apply our results to deduce compactness
for pseudo-differential operators
with symbols in $M^{\sharp ,q}_{(\omega )}$
spaces, when $q\in (0,1]$. 
These investigations also include 
compactness investigations for Toeplitz
operators with symbols in
$M^{\sharp ,\infty}_{(\omega )}$ spaces.

\medspace


In the family of modulation spaces,
the so-called Sj{\"o}strand classes
is of peculiar interest, especially
in pseudo-differential calculus and
Gabor frame theory. For any $q\in (0,1]$
and suitable weight function $\omega$ on
the phase space $\rr {2d}$, the
corresponding Sj{\"o}strand class
$M^{\infty ,q}_{(\omega )}(\rr d)$ consists
of all functions or (ultra-)distributions
$f$ on $\rr d$ such that 
\begin{equation}
\label{Eq:DefSjosClassIntro}
\nm f{M^{\infty ,q}_{(\omega )}}
\equiv
\NM {\sup _{x\in \rr d}
|V_\phi f(x,\cdo )\omega (x,\cdo)|}
{L^q (\rr d )} <\infty .
\end{equation}
Here $V_\phi f$ is the short-time Fourier
transform of $f$, with respect to the (fixed)
window function $\phi$.

\par

In pseudo-differential calculus, 
it is often convenient to use Sjöstrand
classes and some of their subspaces as
symbol classes for the pseudo-differential
operators.
Firstly, such operators possess
convenient mapping properties, especially
when acting on (other) modulation spaces.
(See \cite{} for classical
but restrictive results, or
\cite{} for recent and general
results on this.) We remark that these mapping
properties immediately give the classical
Calderon-Vailancourt's theorem. That is, 
if $a$ is smooth
and bounded together with its
derivatives, then the
pseudo-differential operator $\op (\fka )$ is
continuous on $L^2$. (See \cite{}.)


\par

Secondly, for suitable restrictions on
the weight function $\omega$, corresponding
family of pseudo-differential operators
are inverse closed on $L^2$, which can
also be expressed that they form a so-called
Wiener algebra. For example, if $\omega =1$ and
$a\in M^{\infty ,q}(\rr {2d})$ satisfies that
$\op (\fka )$ is invertible on $L^2(\rr d)$,
then its inverse $\op (\fka )^{-1}$ is equal to
$\op (\fkb )$, for some $a\in M^{\infty ,q}(\rr {2d})$
(see e.{\,}g. \cite{GrPfTo}).

\par

Thirdly, any Gabor frame operator in
Gabor analysis is a pseudo-differential
operator with symbols in Sj{\"o}strand
classes. Some tracks on
this property can be found
already in the analysis in \cite{Stroh},
and an explicit verification of
this in unweighted case, is given in
\cite{GrPfTo}. Because of all these
roles and properties, several attentions
are paid to Sj{\"o}strand classes.

\par

There are useful modulation space which stays in the
shadow in the limelight of Sj{\"o}strand classes.
In the paper we focus on one type of them, 
where, in additional to the conditions in
\eqref{Eq:DefSjosClassIntro}, one also imposes
the condition
\begin{equation}
\label{Eq:DefSjosSubClassIntro}
\lim _{R\to \infty}
\NM {\sup _{|x|>R}
|V_\phi f(x,\cdo )\omega (x,\cdo)|}
{L^q (\rr d )} =0.
\end{equation}
For convenience we let
$M^{\sharp ,q}_{(\omega )}(\rr d)$
be the set of all $f$ which satisfy
these conditions. That is,
$$
M^{\sharp ,q}_{(\omega )}(\rr d)
\equiv
\sets {f\in M^{\infty ,q}_{(\omega )}(\rr d)}
{f\ \text{satisfies \eqref{Eq:DefSjosSubClassIntro}}}.
$$

\par

We perform detailed analysis of
$M^{\sharp ,q}_{(\omega )}$ spaces. In the
first part we show that \eqref{Eq:DefSjosSubClassIntro}
is independent of any choice of $\phi$ in
$M^r_{(v)}(\rr d)\setminus 0$, when
$r\le \min (1,q)$ and $v$ moderates $\omega$. At
the same time we show that we may replace the
$L^q$-norm in \eqref{Eq:DefSjosSubClassIntro}
with a broad class of so-called Wiener norms,
which enclose the $L^q$-norm.

\par

Thereafter we apply our results to show
compactness for pseudo-differential
and Toeplitz operators, when their
symbols belong to
$M^{\sharp ,q}_{(\omega )}$ and
$M^{\sharp ,\infty }_{(\omega )}$ spaces,
respectively. We let such operators
act on a broad class of modulation
spaces, $M(\omega _j,\mascB )$,
parameterized by suitable weight
functions $\omega _j$ and
quasi-Banach function spaces $\mascB$.
By choosing $\mascB$ as the mixed
norm space $L^{p,q}$ of Lebesgue types,
$M(\omega _j,\mascB )$ becomes
the standard modulation space
$M^{p,q}_{(\omega )}$. Hence our
results apply to standard modulation
spaces.

\par

In order to be more specific, first
suppose that these symbols belong to the 
superclasses
$M^{\infty ,q}_{(\omega )}$ and
$M^{\infty ,\infty }_{(\omega )}$
of $M^{\sharp ,q}_{(\omega )}$ and
$M^{\sharp ,\infty }_{(\omega )}$.
Then it is shown in \cite{ToPfTe}
that the mappings
\begin{alignat}{2}
\op (\fka ) &: M(\omega _1,\mascB)
\to M(\omega _2,\mascB),&
\qquad \fka &\in M^{\infty ,q}_{(\omega )}
(\rr {2d})
\intertext{and}
\tp _{\phi _1,\phi _2}(\fkb ) &:
M(\omega _1,\mascB)
\to M(\omega _2,\mascB),&
\qquad \fkb &\in M^{\infty ,\infty}
_{(\omega )}(\rr {2d})
\end{alignat}
are continuous, provided the conditions
on the involved weight functions
$\omega$, $\omega _1$, $\omega _2$,
and the window functions $\phi _1$
and $\phi _2$ for the Toeplitz operator
$\tp _{\phi _1,\phi _2}(\fkb )$,
are suitable. In Section \ref{sec5}
we show that these mappings are
compact when
$\fka in M^{\sharp ,q}_{(\omega )}$ and
$\fkb \in M^{\sharp ,\infty }_{(\omega )}$.

\par



\par

\section{Preliminaries}\label{sec1}

\par

\subsection{Weight functions}
\label{subsec1.1}

\par

For non-negative functions $g_1$ and $g_2$,
we write $g_1 \lesssim g_2$ when there is
a constant $C>0$ such that
$g_1(x ) \le C \cdot  g_2(x)$
holds uniformly for all $x$
in the intersection of the domains of
$g_1$ and $g_2$. We also
write $g_1\asymp g_2$
when $g_1\lesssim g_2 \lesssim g_1$.

\par

A \emph{weight} or \emph{weight function} on $\rr d$ is a
positive function $\omega
\in  L^\infty _{loc}(\rr d)$ such that $1/\omega \in  L^\infty _{loc}(\rr d)$.
If there is a  weight $v$ on $\rr d$ and a
constant $C\ge 1$ such that
\begin{equation}\label{moderate}
\omega (x+y) \le C\omega (x)v(y),\qquad x,y\in \rr d,
\end{equation}
then the weight $\omega$ is called
\emph{moderate}, or \emph{$v$-moderate}.
By \eqref{moderate}
we have
\begin{equation}\label{moderateconseq}
C^{-1}v(-x)^{-1}\le \omega (x)\le C v(x),\quad x\in \rr d,
\end{equation}
for some $C\ge 1$.
We let $\mascP _E(\rr d)$ be the set of all moderate weights on $\rr d$.
We say that the weight $\omega$ on $\rr d$ is polynomially moderate
if
\begin{equation}
\tag*{(\ref{moderate})$'$}
\omega (x+y) \le C\omega (x)(1+|y|)^r,\qquad x,y\in \rr d,
\end{equation}
for some constants $C,r>0$. The set of polynomially moderate weights
on $\rr d$ is denoted by $\mascP (\rr d)$. Evidently,
\begin{equation}
\label{Eq:ModerateClassesComp}
\mascP (\rr d) \subseteq \mascP _E(\rr d).
\end{equation}
We also consider various classes of exponential
type moderate weights.
In fact, suppose that $s>0$. Then the weight
$\omega$ on $\rr d$ is called $s$-moderate
if there are constants $C,r>0$ such that
\begin{equation}
\label{Eq:ExpMod}
\omega (x+y) \le C\omega (x)e^{r|y|^{\frac 1s}},
\qquad x,y\in \rr d,
\end{equation}
If for every $r>0$, there is a constant $C=C_r>0$
such that \eqref{Eq:ExpMod} is fulfilled, then
$\omega$ is called $(0,s)$-moderate. 
The sets of $s$-moderate and $(0,s)$-moderate
weights on $\rr d$ are denoted by $\mascP _s(\rr d)$
and $\mascP _{0,s}(\rr d)$, respectively.

\par

We say that a  weight $v$ is
\emph{submultiplicative} if
\begin{equation}\label{Eq:Submultiplicative}
v(x+y) \le v(x)v(y)
\quad \text{and}\quad
v(-x)=v(x),\qquad x,y\in \rr d.
\end{equation}

\par

If $v$ is positive and locally bounded and
satisfies the  inequality in 
\eqref{Eq:Submultiplicative},
then $v(x)\lesssim e^{r|x|}$ for some $r>0$,
cf. \cite{Gro2007}.

\par

We observe that given a $v$-moderate weight 
$\omega$, one can find a continuous $v$-
moderate weight $\omega_0$ such that $\omega 
\asymp \omega_0$ (see e.g.
\cite{Toft10}). In addition, a moderate 
weight $\omega$
is also moderated by a submultiplicative
weight, cf. e.{\,}g. \cite{FeiGro1,Toft10}.
Therefore, if $\omega \in \mascP _E(\rr d)$, 
then
\begin{equation}\label{Eq:weight0}
\omega (x+y) \lesssim \omega (x) e^{r|y|},
\quad x,y\in \rr d,
\end{equation}
for some $r>0$. 
In particular, \eqref{moderateconseq} shows that
for any $\omega \in \mascP_E(\rr d)$,
there is a constant $r>0$ such that
\begin{equation}\label{Eq:weight1}
e^{-r|x|}\lesssim \omega (x)
\lesssim e^{r|x|},\quad x\in \rr d.
\end{equation}
Consequently, 
\begin{equation}
\label{Eq:SubmultFuncEst}
\omega \cdot e^{-r|\cdo |}
\in L^\infty (\rr d)
\quad
\text{for some $r>0$, when
$\omega \in \mascP _E(\rr d)$.}
\end{equation}

\par

We also observe that \eqref{Eq:weight0} gives
\begin{equation}
\begin{aligned}
\mascP (\rr d)
&\subsetneq
\mascP _{0,s_2}(\rr d)
\subsetneq
\mascP _{s_2}(\rr d)
\subsetneq
\mascP _{0,1}(\rr d)
\subsetneq
\mascP _{1}(\rr d)
\\[1ex]
&=
\mascP _E(\rr d)
=
\mascP _{0,s_1}(\rr d)
=
\mascP _{s_1}(\rr d),
\quad
s_1<1<s_2.
\end{aligned}
\end{equation}

\par

In the sequel, $v$ and $v_0$, always stand for
submultiplicative weights if nothing else is stated.

\par

We say that the two weights $\omega _1$ and $\omega _2$
are \emph{equivalent} if $\omega _1\asymp \omega _2$.

\par

\begin{rem}
\label{Rem:EquivSmoothWeights}
We notice that for any $s>0$ and $\omega \in \mascP _s(\rr d)$
or $\omega \in \mascP _{0,s}(\rr d)$, there is an equivalent
smooth weight $\omega _0$ in the same class.
(See e.{\,}g. \cite[Lemma 2.8]{Toft10}.)
\end{rem}

\par 

\subsection{Gelfand-Shilov spaces}
\label{subsec1.2}

\par

We recall that the Gelfand-Shilov space
$\Sigma _1(\rr d)$ consists of all
$f\in C^\infty (\rr d)$ such that
$$
\nm f{[h]}
\equiv
\sup _{\alpha ,\beta \in \nn d}
\left (
\frac {\nm {x^\alpha D^\beta f}{L^\infty}}
{h^{|\alpha +\beta |}\alpha !\beta !}
\right )
$$
is finite for every $h>0$. It follows that
$\Sigma _1(\rr d)$ is a Fr{\'e}chet space
with projective topology defined by the
semi-norms $\nm \cdo{[h]}$.

\par

One has that $\Sigma _1(\rr d)$ is
continuously embedded and dense
in the Schwartz space
of smooth rapidly decreasing functions
$\mascS (\rr d)$, is invariant under
Fourier transforms, and is an algebra
under multiplications and under
convolutions.

\par


It can be shown that $\Sigma _1(\rr d)$ satisfies
the hypothesis of \cite[Theorem 4.1]{Ans},
wherefrom we get the following.


\par

\begin{lemma}
Suppose $f\in \Sigma _1(\rr d)$. Then
for some $g_j,h_j\in \Sigma _1(\rr d)$, $j=1,2$,
it holds
$$
f=g_1\cdot g_2 = h_1*h_2.
$$
\end{lemma}

\par 

In what follows we let $\mathscr F$ be the
Fourier transform which takes the form
$$
(\mathscr Ff)(\xi )= \widehat f(\xi ) \equiv (2\pi )^{-\frac
d2}\int _{\rr
{d}} f(x)e^{-i\scal  x\xi }\, dx
$$
when $f\in L^1(\rr d)$. Here $\scal \cdo \cdo$ denotes the usual
scalar product
on $\rr d$. The map $\mathscr F$ extends
uniquely to a homeomorphism on the space of tempered distributions
$\mascS '(\rr d)$,
to a unitary operator on $L^2(\rr d)$ and restricts
to a homeomorphism on the Schwartz space 
$\mathscr S(\rr d)$.
We observe that with our choice of the Fourier
transform, the usual
convolution identity for the Fourier transform
takes the forms
\begin{equation}\label{Eq:FourTransfConv}
\mascF (f\cdot g)
=
(2\pi )^{-\frac d2}\widehat f *\widehat g
\quad \text{and}\quad
\mascF (f* g)
=
(2\pi )^{\frac d2}\widehat f \cdot \widehat g
\end{equation}
when $f,g\in \mascS (\rr d)$.

\par

Since we are interested in general weights $\omega \in \mascP _E(\rr d)$, instead of the framework of
tempered distributions $\mascS '(\rr d)$, which is natural when dealing with weights of polynomial growth,
we consider Gelfand-Shilov spaces $\Sigma _s(\rr d)$ and $\maclS _s (\rr d)$ and their dual spaces of (ultra-)distributions $\Sigma _s '(\rr d)$ and $\maclS _s ' (\rr d)$, $ s \geq 1$.

In order to avoid technical issues related to the usual definition of the spaces of  Gelfand-Shilov type and their distribution spaces, we introduce  $\Sigma _s(\rr d)$ and $\maclS _s (\rr d)$ in terms of
decay estimates of the functions and their Fourier transforms. More precisely, if $f\in \mascS '(\rr d)$ and $s \geq 1$, then $f\in \Sigma _s (\rr d)$) ($f\in \maclS _s (\rr d)$), if and only if
\begin{alignat}{6}
f&\in \Sigma _s(\rr d) &
\quad &\Leftrightarrow &\quad
|f(x)|&\lesssim e^{-r|x|^{\frac 1s}} , &
\ 
|\widehat f(\xi )| &\lesssim
e^{-r|\xi |^{\frac 1s}}, &
\ &\text{for all} & \ r &>0,
\label{Eq:GSFtransfChar}
\intertext{and}
f&\in \maclS _s(\rr d) &
\quad &\Leftrightarrow &\quad
|f(x)|&\lesssim e^{-r|x|^{\frac 1s}}, &
\ 
|\widehat f(\xi )| &\lesssim
e^{-r|\xi |^{\frac 1s}}, &
\ &\text{for some} & \ r&>0,
\label{Eq:GSRoumFtransfChar}
\end{alignat}
cf.
\cite{ChuChuKim, Eij}.

\par

The \emph{Gelfand-Shilov distribution space}
$\Sigma _s '(\rr d)$
(of \emph{Beurling type}) can be introduced as the (strong) dual to  $\Sigma _{s}(\rr d)$, and
the \emph{Gelfand-Shilov distribution space} $ \maclS _s  '(\rr d)$
(of \emph{Roumieu type}) is the (strong) dual to  $ \maclS _s (\rr d)$, $ s\geq 1$.

\par

When $1\leq s_1\leq s_2$ we have  continuous embeddings
\begin{multline}\label{eq:EmbGelfandShilov}
\Sigma _{s_1} (\rr d) \hookrightarrow \maclS _{s_1} (\rr d) \hookrightarrow
\Sigma _{s_2} (\rr d) 
\hookrightarrow \mascS (\rr d)
\hookrightarrow \mascS ' (\rr d)
\hookrightarrow \maclS _{s_2} (\rr d) \\
\hookrightarrow
\maclS _{s_2} ' (\rr d) \hookrightarrow \Sigma _{s_2} ' (\rr d)  \hookrightarrow 
\maclS _{s_1} ' (\rr d)  \hookrightarrow \Sigma _{s_1} ' (\rr d).    
\end{multline}
Moreover, the embeddings in \eqref{eq:EmbGelfandShilov}
are dense.

\par

We let $\maclD _s(\rr d)$ and
$\maclD _{0,s}(\rr d)$ be the set of compactly supported 
elements in $\maclS _s(\rr d)$ and $\Sigma _s(\rr d)$,
respectively. We notice that $\maclD _s(\rr d)$ and
$\maclD _{0,s}(\rr d)$ are non-trivial (i.{\,}e. they
contains more elements that the zero element), if and only
if $s>1$ (see e.{\,}g. \cite [Section 8.4]{Ho1}).
As in \cite{Ho1} we let $C_0^\infty (\rr d)$ be the set of
all smooth functions on $\rr d$ with compact support.

\par

For $s>1$ we let $\maclE _s'(\rr d)$ and
$\maclE _{0,s}'(\rr d)$ be the set of all
compactly supported elements in
$\maclS _s'(\rr d)$ and $\Sigma _s'(\rr d)$,
respectively.

\par

For simplicity, from now on, we consider  $\Sigma _{s}(\rr d)$ and  $\Sigma _s '(\rr d)$, 
$ s\geq 1$.
The reader should  keep in mind that all results
and comments given there hold true when
$\Sigma _{s}$ and  $\Sigma _s '$
are replaced by $ \maclS _s$ and
$ \maclS _s  '$ respectively.

\par

\subsection{The short-time Fourier
transform and Gelfand-Shilov spaces}
\label{subsec1.3}

\par

In several situations it is convenient to
use a localized version of the Fourier
transform, called the short-time Fourier
transform, STFT for short.
The \emph{short-time Fourier transform} of
$f\in \Sigma _{s} '(\rr d)$ with respect to the fixed \emph{window function}
$\phi \in \Sigma _{s}  (\rr d)$ is defined by
\begin{equation} \label{Eq:STFTDef}
(V_\phi f)(x,\xi ) \equiv (2\pi )^{-\frac d2}
(f,\phi (\cdo -x)e^{i\scal \cdo \xi})_{L^2}.
\end{equation}
Here and in what follows, $(\cdo ,\cdo )_{L^2}$ denotes the unique continuous
extension of the inner product on $L^2(\rr d)$ restricted to
$\Sigma _{s} (\rr d)$ into a continuous map from $\Sigma _{s} '(\rr d)
\times \Sigma _{s} (\rr d)$ to $\mathbf C$.

We observe that using certain properties for tensor products
of distributions,
\begin{equation}
(V_\phi f)(x,\xi ) = \mascF (f\cdot \overline {\phi (\cdo -x)})(\xi ).
\tag*{(\ref{Eq:STFTDef})$'$}
\end{equation}
(cf. \cite{Ho1,Toft22}). If in addition
$f\in L^p(\rr d)$ for some $p\in [1,\infty ]$, then
\begin{equation}
(V_\phi f)(x,\xi ) = (2\pi )^{-\frac d2}\int _{\rr d}f(y)
\overline {\phi (y-x)}e^{-i\scal y\xi}\, dy.
\tag*{(\ref{Eq:STFTDef})$''$}
\end{equation}

\par

\begin{prop}\label{Prop:ExtSTFTSchwartz}
The map
\begin{alignat}{4}
(f,\phi ) &\mapsto V_\phi f &  &:\,  &
&\Sigma _{s} (\rr d) \times \Sigma _{s} (\rr d) & &\to \Sigma _{s} (\rr {2d})
\label{Eq:STFTSchwartz}
\intertext{is continuous, which extends uniquely to a continuous map}
(f,\phi ) &\mapsto V_\phi f & &:\,  &
&\Sigma _{s} '(\rr d) \times \Sigma _{s} '(\rr d) & &\to \Sigma _{s} '(\rr {2d}),
\label{Eq:STFTTempDist}
\intertext{which in turn restricts to an isometric map}
(f,\phi ) &\mapsto V_\phi f & &:\,  &
&L^2(\rr d) \times L^2(\rr d) & &\to L^2(\rr {2d}).
\label{Eq:STFTL2}
\end{alignat}

\end{prop}

\par

We omit the proof since it follows from results 
in \cite{Fol,GroZim2001,Teofanov2,Toft10,Toft17}.

\par

If $\phi \in \Sigma _{s} (\rr d)$ and $f\in \Sigma _{s} '(\rr d)$, then
\eqref{Eq:STFTTempDist} shows that $V_\phi f\in \Sigma _{s} '(\rr {2d})$.
On the other hand, it is easy to see that the right-hand side
of \eqref{Eq:STFTDef} defines a smooth function. Consequently
beside \eqref{Eq:STFTTempDist} and \eqref{Eq:STFTSchwartz},
we also have the continuous map
\begin{equation}\label{Eq:STFTTempDistSchwartz}
(f,\phi ) \mapsto V_\phi f   :\,
\Sigma _{s} '(\rr d) \times \Sigma _{s} (\rr d) \to \Sigma _{s} '(\rr {2d})
\cap C^\infty (\rr {2d}).
\end{equation}

\par

For the short-time Fourier transform, the Parseval identity
is replaced by the so-called Moyal identity, also known as the {\em orthogonality relation}
\begin{equation}\label{Eq:Moyal}
(V_{\phi}f,V_{\psi}g)_{L^2(\rr {2d})}
= (\psi ,\phi )_{L^2(\rr d)}(f,g)_{L^2(\rr d)},
\end{equation}
when $f,g,\phi ,\psi \in L^2(\rr d)$, cf. \cite{Gro2}.

\par
By Moyal's identity \eqref{Eq:Moyal} it follows that if
$\phi \in \Sigma _{s} (\rr d)\setminus 0$, then the identity operator
on $\Sigma _{s} '(\rr d)$ is given by
\begin{equation}\label{Eq:IdentSTFTAdj}
\operatorname{Id}
=
\left (\nm \phi{L^2}^{-2}\right ) \cdot V_\phi ^*\circ V_\phi ,
\end{equation}
provided suitable mapping properties
of the ($L^2$-)adjoint $V_\phi ^*$ of $V_\phi$ can be established.
Obviously, $V_\phi ^*$ fullfils
\begin{equation}\label{Eq:STFTAdjoint}
(V_\phi ^*F,g) _{L^2(\rr d)} = (F,V_\phi g) _{L^2(\rr {2d})}
\end{equation}
when $F\in \Sigma _{s} (\rr {2d})$ and $g\in \Sigma _{s}(\rr d)$.

\par

By expressing the scalar product and the short-time Fourier
transform in terms of integrals in \eqref{Eq:STFTAdjoint},
it follows by straight-forward manipulations that
the adjoint in \eqref{Eq:STFTAdjoint} can be written as
\begin{equation}\label{Eq:STFTAdjointFormula}
(V_\phi ^*F)(x) = (2\pi )^{-\frac d2}
\iint _{\rr {2d}} F(y,\eta )\phi (x-y)e^{i\scal x\eta}\, dy \, d\eta ,
\end{equation}
when $F\in \Sigma _{s} (\rr {2d})$ and $\Phi \in \Sigma _{s} (\rr {d})$. We may now use mapping properties
like \eqref{Eq:STFTTempDist}--\eqref{Eq:STFTL2} to extend
the definition of $V_\phi ^*F$ when $F$ and $\phi$ belong to
various classes of function and distribution spaces. For example,
by \eqref{Eq:STFTTempDist},
\eqref{Eq:STFTSchwartz} and \eqref{Eq:STFTL2},
it follows that the map
$$
(F,g)\mapsto (F,V_\phi g) _{L^2(\rr {2d})}
$$
defines a sesqui-linear form on
$\Sigma _{s} (\rr {2d})\times\Sigma _{s} '(\rr d)$,
$\Sigma _{s} '(\rr {2d})\times \Sigma _{s} (\rr d)$ and on
$L^2(\rr {2d})\times L^2(\rr d)$. This implies that
if $\phi \in \Sigma _{s} (\rr d)$, then
the mappings
\begin{equation}\label{Eq:STFTAdjCont}
\begin{alignedat}{3}
V_\phi ^* :  \Sigma _{s} (\rr {2d}) \to \Sigma _{s} (\rr d),
\qquad
&V_\phi ^* & &: &\Sigma _{s} '(\rr {2d}) &\to \Sigma _{s} '(\rr d)
\\[1ex]
\text{and}\qquad
&V_\phi ^* & &: & L^2(\rr {2d}) &\to L^2(\rr d)
\end{alignedat}
\end{equation}
are continuous.

\par

We will often use the estimates
\begin{align}
| (V _{\phi _1} \circ V_{\phi _2} ^*)F| &\leq
|F| * | V _{\phi _1}  \phi _2|
\label{Eq:adjoint-convolution}
\intertext{and}
|(\phi _2,\phi _3)_{L^2}|
\cdot
| V _{\phi _1}  f |
&\leq 
|V _{\phi _2} f| * | V _{\phi _1}
\phi _3|.
\label{Eq:ChangeWindow}
\end{align}

In some parts of our exposition it is more convenient to use the transform
\begin{align}
\label{def:TTransform}
    T _\phi f (x)
    =
    \int f(y+x) \overline{\phi(y)} e^{-i\scal y\xi} dy
    = e^{i\scal x \xi} V _\phi f(x,\xi), 
    \quad x,\xi \in \rr d
\end{align}
instead of the short-time Fourier transform. 

\par

\subsection{Modulation spaces}
\label{subsec1.4}

\par

For the definition of classical
modulation spaces, we need to discuss
mixed norm spaces of Lebesgue types.
Let $\Omega \subseteq \rr d$ be a
(Lebesgue) measurable set in $\rr d$,
and let $\maclM (\Omega )$ be the set of all 
complex-valued (Lebesgue)
measurable functions on $\Omega$.

\par

For any $p,q\in (0,\infty ]$, the mixed norm
space
$L^{p,q}_{(\omega )}(\rr {2d})$
of Lebesgue type, is given by
\begin{alignat}{3}
L^{p,q}_{(\omega )}(\rr {2d})
&=
\sets{F\in \maclM (\rr {2d})}
{\nm F{L^{p,q}_{(\omega )}}<\infty }
\label{Eq:MixLebSpace1}
\end{alignat}
for some $p,q\in (0,\infty ]$ and
$\omega \in \mascP _E(\rr {2d})$.
Here
\begin{alignat*}{1}
\nm F{L^{p,q}_{(\omega )}}
&\equiv
\nm {G_{F,\omega ,p}}{L^q(\rr {d})},
\intertext{where}
G_{F,\omega ,p}(\xi )
&\equiv
\nm {F(\cdo ,\xi )\omega (\cdo ,\xi )}
{L^p(\rr {d})},
\qquad
F\in \maclM (\rr {2d}).
\end{alignat*}
We put
$$
L^{p,q}= L^{p,q}_{(\omega )}
\quad \text{when}\quad
\omega \equiv 1.
$$
We also put $L^p_{(\omega )}=L^{p,p}_{(\omega )}$, 
and observe that $L^p=L^{p,p}$.

\par

In our investigations, (weighted versions of)
the subclass $ L^\sharp (\rr d)$ of
$L^\infty (\rr d)$ will play a fundamental role.
The space $ L^\sharp (\rr d)$ consists of
all $f \in  L^\infty (\rr d)$ such that
\begin{equation*}
\lim _{R\to \infty}
(\underset{|x|\geq R} \sup \, |f(x) | ) = 0,
\end{equation*}
and if
$\omega \in \mascP _E(\rr {d})$,
then
$ f \in  L^\sharp _{(\omega )}(\rr d)$
if $ f  \omega \in  L^\sharp (\rr d)$.
In addition, 
$L^{\sharp ,q} (\rr {2d})$, $q\in (0,\infty ]$, 
denotes the set of all
$ F\in L^{\infty ,q}(\rr {2d})$
which satisfy
\begin{equation*}
\lim _{R\to \infty}
\NM {\sup _{|x|\ge R}  \left| F (x, \cdot) \right |}{L^q}
= 0,
\end{equation*} 
and if $\omega \in \mascP _E(\rr {2d})$,
then 
$ \displaystyle
F \in  L^{\sharp ,q}_{(\omega )}(\rr {2d}) $ $\Leftrightarrow$
$  \displaystyle F \cdot \omega \in 
L^{\sharp ,q}(\rr {2d})$.


\par


\par






\par



\begin{defn}
\label{Def:SpecialModSpace}
Let $p,q\in (0,\infty ]$, $\omega \in
\mascP _E(\rr {2d})$ and
$\phi \in \Sigma _1(\rr d)\setminus 0$.
\begin{enumerate}
\item The classical modulation space
$M^{p,q}_{(\omega )}(\rr d)$ consists
of all $f\in \Sigma _1'(\rr d)$
such that
\begin{equation}
\label{Eq:ModNorm2}
\nm f{M^{p,q}_{(\omega )}}\equiv \nm {V_\phi f}
{L^{p,q}_{(\omega )}}
\end{equation}
is finite. The topology of $M^{p,q}_{(\omega )}
(\rr d)$ is defined by
the quasi-norm $\nm \cdo{M^{p,q}_{(\omega )}}$;

\begin{equation}
\nm {H_{f,\omega ,p}}{L^q}<\infty ,
\quad \text{where}\quad
H_{f,\omega ,p}(\xi )
\equiv
\nm {V_\phi f(\cdo ,\xi )
\cdot \omega (\cdo ,\xi )}{L^p}.
\end{equation}

\vrum

\item the modulation spaces
$M^{\sharp ,q}_{(\omega )}(\rr d)$ consists of
all
$f\in M^{\infty ,q}_{(\omega )}(\rr d)$
such that
\begin{equation}
\label{Eq:ModNormInfinityCond}
\lim _{R\to \infty}
\NM  { \sup _{|x|\ge R}\Big | V_\phi f(x,\cdo )
\omega (x,\cdo )\Big  |}{L^q}
= 0.
\end{equation}
The topology is defined through
the (quasi-)norm
\begin{equation}
\label{Eq:ModSharpNorm}
\nm f{M^{\sharp ,q}_{(\omega )}}
\equiv
\nm f{M^{\infty ,q}_{(\omega )}},
\qquad
f\in M^{\sharp ,q}_{(\omega )}(\rr d).
\end{equation}
\end{enumerate}
\end{defn}

\par

The (classical) modulation spaces, are
essentially
introduced in \cite{Fei1} by Feichtinger.
(See e.g. \cite{Fei6} for definition of more general 
modulation spaces.)
By \eqref{Eq:ModNormInfinityCond}
and \eqref{Eq:ModSharpNorm} it follows
that $M^{\sharp ,q}_{(\omega )}(\rr d)$
is strongly related to 
$M^{\infty ,q}_{(\omega )} (\rr d)$. 
We note that modulation spaces similar to
$M^{\sharp ,q}_{(\omega )} (\rr d)$ are studied
in \cite{CorGro2}.

\par

Evidently, by Definition
\ref{Def:SpecialModSpace} it follows
that
$f\in M^{\sharp ,q}_{(\omega )}(\rr d)$  
if and only if $V_\phi f \in 
L^{\sharp ,q}_{(\omega )}(\rr {2d})$.

\par

\begin{rem}
\label{Rem:VversusT}
Since
$|V_\phi f \cdot \omega |
=
|T_\phi f \cdot \omega |,
$
when $\phi \in \Sigma _1(\rr d)$,
$f\in \Sigma _1'(\rr d)$ and
$T_\phi$ is given by
\eqref{def:TTransform}, it follows
that we may replace $V_\phi f$
with $T_\phi f$ in Definition
\ref{Def:SpecialModSpace}.
\end{rem}

\par


\par

As in \cite{ToPfTe} we need the following 
estimate of 
the short-time Fourier transform with a 
Gaussian window. The result essentially 
follows from \cite[Lemma 2.3]{GaSa}.
For a proof of the  general situation we refer 
to \cite[Lemma 2]{Toft2022}. Here $B_r(x)$
denotes the open ball in $\rr d$ with center at
$x\in \rr d$ and radius $r>0$.

\par

\begin{lemma}\label{Lemma:EstSTFT}
Let $p \in (0,\infty]$, $R > 0$,
$X_0\in \rr {2d}$ be fixed, 
and let $\phi _0(x)
=
\pi ^{-\frac d4}e^{-\frac 12|x|^2}$, $x\in \rr d$. Then
$$
\left | V_{\phi_0} f (X_0) \right | 
\leq
C \nm{ V_{\phi_0} f }{L^p (B_R (X_0) )}, 
\qquad
f \in \Sigma_1' (\rr {d} ),
$$
where the constant $C$ is independent
of $X_0$ and $f$.
\end{lemma}

\par

In the following propositions we
list well-known properties for modulation
spaces. The first proposition deals with
invariant and topological properties
for modulation spaces.

\par

\begin{prop}
\label{Prop:InvModAndBanach}
Let $p,q\in (0,\infty ]$,
$\omega \in \mascP _E(\rr {2d})$ and
$\phi \in \Sigma _1(\rr d)\setm $.
Then the following is true:
\begin{enumerate}
\item the definition of
$M_{(\omega )}^{p,q}(\rr d)$
is independent of the choices of
$\phi \in \Sigma _1 (\rr d)\setm $,
and different choices give rise to
equivalent quasi-norm;

\vrum

\item the space $M^{p,q}_{(\omega )}(\rr d)$
is a quasi-Banach space
which increases with $p$ and $q$, and
decreases with $\omega$. If in addition
$p,q\ge 1$, then $M_{(\omega )}^{p,q}(\rr d)$
is a Banach spaces.
\end{enumerate}
\end{prop}

\par

For convenience we set
$M^p_{(\omega )}=M^{p,p}_{(\omega )}$.
If in addition $\omega \equiv 1$, then we
set $M^{p,q}=M^{p,q}_{(\omega )}$
and $M^p=M^p_{(\omega )}$.

\par

\subsection{Invariant quasi-Banach 
function spaces}\label{subsec1.5}

\par


\par

The functional
$\nm \cdo{\mascB}$, defined on the
vector space $\mascB$ is called a
\emph{quasi-norm} of order
$r _0 \in (0,1]$, or an \emph{$r_0$-norm}, if the conditions
\begin{alignat}{2}
\nm {f+g}{\mascB} ^{r _0}
&\le
\nm {f}{\mascB} ^{r _0} + \nm {g}{\mascB} 
^{r _0},& 
\quad f,g &\in \mascB .
\label{Eq:WeakTriangle1}
\\[1ex]
\nm {\alpha \cdot f}{\mascB}
&=
|\alpha| \cdot \nm f{\mascB}\ge 0,
& \quad \alpha &\in \mathbf{C},
\quad  f \in \mascB ,
\notag
\intertext{and}
\nm f{\mascB} &= 0
\quad  \Leftrightarrow \quad
f=0, & &
\notag
\end{alignat}
hold true. A vector space which contains 
a quasi-norm (of order $r_0$)
is called a quasi-norm space (of order 
$r_0$), or an $r_0$-normed space.
A complete quasi-norm space
(of order $r_0$) is called a
\emph{quasi-Banach space} (of order $r_0$),
or an \emph{$r_0$-Banach space}.

\par

The quasi-Banach
space $\mascB$ is called a
\emph{quasi-Banach function
space} (or \emph{QBF space}) on
$\rr d$, whenever $\mascB$ is
contained in
$\maclM (\rr d)$.
It is here
understood that the zero element
in $\maclM (\rr d)$ (which is
an equivalent class), is the zero
element in $\mascB$.

\par

We always assume that the QBF spaces
are \emph{solid}, which is described
in the following definition.
(See e.g. (A.3) in \cite{Fe1992}.)

\par

\begin{defn}
\label{Def:SolidBF}
Let $\mascB$ be a
\emph{QBF space} on
$\rr d$.
Then $\mascB$ is
called \emph{solid}, if the
following conditions hold:
\begin{enumerate}
\item
for any
$f,g\in \maclM (\rr d)$
which satisfy
$g\in \mascB$ and $|f|
\le |g|$, then $f\in \mascB$ and
\begin{equation}\label{Eq:Solid}
\nm f{\mascB}\le C\nm g{\mascB},
\end{equation}
for some constant $C>0$ which is
independent of $f$ and $g$ above;

\vrum

\item
for any $x\in \rr d$,
there is a constant $c>0$ and a function
$f\in \mascB$ such that $|f|>c$ on a convex
set in $\rr d$ which is not a zero set,
and which contains $x$.
\end{enumerate}
\end{defn}

\par

In the following definition
we put needed restrictions
on QBF spaces.

\par

\begin{defn}\label{Def:BFSpaces}
Let $\mascB$ be a solid QBF space of
order $r_0\in (0,1]$ on $\rr d$,
$C>0$ be the same as in 
\eqref{Eq:Solid},
and let $v_0\in \mascP _E(\rr d)$
be submultiplicative.
\begin{enumerate}
\item The space $\mascB$
is called a
\emph{solid translation invariant
quasi-Banach function space on
$\rr d$} 
(with respect to $r_0$ and $v_0$), or 
\emph{invariant QBF space on $\rr d$},
if for any $x\in \rr d$ and
$f\in \mascB$, then $f(\cdo -x)\in
\mascB$, and 
\begin{equation}\label{translmultprop1}
\nm {f(\cdo -x)}{\mascB}
\le
Cv_0(x)\nm {f}{\mascB}\text ;
\end{equation}

\vrum

\item The space $\mascB$
is called a \emph{solid translation 
invariant Banach function space on
$\rr d$} 
(with respect to $v_0$), or 
\emph{invariant BF space on $\rr d$},
if $\mascB$ is a Banach space and
an invariant QBF space on $\rr d$
with respect to $v_0$.

\vrum

\item The space $\mascB$
is called a \emph{solid translation 
invariant normal quasi-Banach function 
space on $\rr d$} (with respect to $r_0$ and $v_0$), or 
\emph{a normal QBF space on $\rr d$},
if there is an invariant BF space
$\mascB _0$ on $\rr d$
with respect to $v_0^{r_0}$ such that
\begin{equation}
\label{Eq:NormQBFNorm}
\nm f{\mascB}
\equiv
\nm {\, |f|^{r_0}\, }{\mascB _0}^{1/r_0},
\quad
f\in \mascB .
\end{equation}
\end{enumerate}
\end{defn}

\par

Evidently, $\mascB$ in
Definition \ref{Def:BFSpaces} (3),
is an invariant QBF space with respect
to $r_0$ and $v_0$.

\par



Suppose that
$\omega \in \mascP _E(\rr d)$, and 
$\mascB$ is a quasi-Banach function
space on $\rr d$ with respect
to $r_0\in (0,1]$ and
$v_0\in \mascP _E(\rr d)$. Then we
let
\begin{equation}\label{Eq:BomegaDef}
\mascB _{(\omega )}=\mascB (\omega)
\equiv
\sets {f\in \maclM (\rr d)}
{f\cdot \omega \in \mascB},
\end{equation}
and equip the space in
\eqref{Eq:BomegaDef} with the
quasi-norm
$$
\nm f{\mascB _{(\omega )}}
\equiv
\nm {f\cdot \omega}{\mascB},
\qquad
f\in \mascB _{(\omega )}.
$$
For any $f\in \maclM (\rr d)
\setminus \mascB _{(\omega )}$, we put
$\nm f{\mascB}=\infty$.
%
%
%







\par

By the definitions it follows that the
families of invariant QBF spaces
in Definition \ref{Def:BFSpaces}
are not enlarged by including
weighted versions as in
\eqref{Eq:BomegaDef}. On the other hand,
in some situations, it might be
convenient to include weights in the
definition of such spaces, see e.g. 
\cite{PfTo19}.

\par


\par

Every invariant BF space $\mascB$ is 
continuously embedded in
$\Sigma _1'(\rr d)$
and the map $(f,\fy )\mapsto f*\fy$
is well-defined and continuous from 
$\mascB \times L^1_{(v_0)}(\rr d)$
to $\mascB$. 
For the invariant BF space
$\mascB \subseteq L^1_{loc}(\rr d)$ with respect to $v_0$ we have
Minkowski's inequality, i.e.,
\begin{equation}\label{Eq:MinkIneq}
\nm {f*\fy}{\mascB}\le C \nm {f}{\mascB}\nm \fy{L^1_{(v_0)}},
\qquad f\in \mascB ,\ \fy \in L^1_{(v_0)} (\rr d),
\end{equation}
for some $C>0$ which is independent of
$f\in \mascB$ and $\fy \in L^1_{(v_0)} (\rr d)$, 
see e.g. \cite{PfTo19}.

\par

For any $\mascB$ in Definition
\ref{Def:BFSpaces}, the corresponding
sequence space,
$\ell _{\mascB} = \ell _{\mascB}(\zz d)$
is the set of all sequences
$a = \{ a(j) \} _{j\in \zz d}\subseteq
\mathbf C$ such that
\begin{equation} \label{eq:seq-space-norm}
\nm a{\ell _{\mascB}}
\equiv
\NM
{\sum _{j\in \zz d}a(j)\chi _{j+Q}}
{\mascB},
\qquad Q=[0,1]^d.    
\end{equation}
Here $\chi _E$ is the characteristic
function of the set $E\subseteq \rr d$.

\par

\begin{example}
\label{Example:LebCase}
A common choice of $\mascB$
in Definition \ref{Def:BFSpaces}
is given by $L ^{p,q}_{(\omega)} (\rr {2d})$
in
\eqref{Eq:MixLebSpace1}.

\par

Let $v\in \mascP _E(\rr {d})$ be
chosen such that $\omega$ is 
$v$-moderate,
$$
r _0= \min(1,p,q),
\quad
p_0=\frac{p}{r_0}
\quad \text{and}\quad
q_0=\frac{q}{r_0}.
$$
Then
$L ^{p,q}_{(\omega)} (\rr {d})$
and 
$L^{p,q}_{*,(\omega )} (\rr {d})$
are
normal QBF spaces on $\rr {d}$ with 
respect to 
$v$, $r _0$ and 
$\mascB _0
=
L^{p_0, q_0}_{(\omega _{0})}
(\rr {d})$ 
respectively
$\mascB _0
=
L^{p_0,q_0}_{\ast ,(\omega _{0})}
(\rr {d})$, $\omega _{0}=\omega ^{r_0}$.

\par

Let $\Omega \subseteq \rr d$ be Lebesgue measurable.
As usual we let $L^{p,q}_{(\omega )}(\Omega )$
($L^{p,q}_{*,(\omega )}(\Omega )$)
be the set of all $f\in \maclM (\Omega )$ such that
$$
f_\Omega (x)
=
\begin{cases}
f(x), & x\in \Omega ,
\\[1ex]
0, & x\in \rr d \setminus \Omega ,
\end{cases}
$$
belongs to $L^{p,q}_{(\omega )}(\rr d)$
($L^{p,q}_{*,(\omega )}(\rr d)$). The corresponding
norm is given by
$$
\nm f{L^{p,q}_{(\omega )}(\Omega )}
\equiv 
\nm {f_\Omega}{L^{p,q}_{(\omega )}(\rr d )}
\qquad
\big (
\nm f{L^{p,q}_{*,(\omega )}(\Omega )}
\equiv 
\nm {f_\Omega}{L^{p,q}_{*,(\omega )}(\rr d )}
\big ).
$$

\par

In similar ways we let
$\ell ^{p,q}_{(\omega )}(\zz {d})$
and
$\ell ^{p,q}_{*,(\omega )}(\zz {d})$, 
$d=d_1 + d_2$ with $d_1, d_2 \in \mathbb N$,
be the set of all sequences $c$ on $\zz d$
such that
$$
\nm {c}{\ell _{(\omega )}^{p,q}}<\infty 
\quad \text{and}\quad
\nm {c}{\ell _{*,(\omega )}^{p,q}}
<\infty,
$$
respectively. Here
\begin{alignat*}{3}
\nm {c}
{\ell _{(\omega )}^{p,q}(\zz {d})}
&\equiv
\nm {G_{c,\omega ,p}}{\ell ^q(\zz {d_2})}, &
\quad 
G_{c,\omega, p}(\iota)
&=
\nm {c(\cdo ,\iota )\omega (\cdo ,\iota)}
{\ell ^p(\zz {d_1})}, & \quad
\iota &\in \zz {d_2}, 
\intertext{and}
\nm {c}
{\ell ^{p,q}_{*,(\omega )}(\zz {d})}
&\equiv
\nm {H_{c,\omega ,q}}{\ell ^p(\zz {d_1})}, &
\quad 
H_{c,\omega, q}(j)
&=
\nm {c(j,\cdo )\omega (j,\cdo )}
{\ell ^q(\zz {d_2})}, & \quad
j&\in \zz {d_1},
\end{alignat*}
and $\ell _0'(\zz d)$
denotes the set of all 
formal (complex-valued) sequences 
$c=\{ c(j)\} _{j\in \zz d}$ on $\zz d$.
We abbreviate
$$
\ell ^{p,q}= \ell ^{p,q}_{(\omega )}
\quad \text{and}\quad
\ell ^{p,q}_*= \ell ^{p,q}_{*,(\omega )}
\quad \text{when}\quad
\omega \equiv 1,
$$
$\ell ^p_{(\omega )}=\ell ^{p,p}_{(\omega )}$, and
$\ell ^p=\ell ^{p,p}$.

\par

By letting $\mascB$ be equal to
$L^{p}(\rr {d})$, $L^{p,q} (\rr {d})$, 
or $L ^{p,q} _\ast (\rr {d})$
in \eqref{eq:seq-space-norm}, we obtain
$$
\ell _ {L ^p_{(\omega)} (\rr {d})}
=
\ell ^p_{(\omega)}(\zz d),
\quad
\ell _ {L ^{p,q} _{(\omega)} (\rr {d})} 
=
\ell ^{p,q} _{(\omega)}(\zz {d})
\quad
\text{and}
\quad
\ell _ {L ^{p,q} _{\ast, (\omega)}
(\rr {d})}
=
\ell ^{p,q} _{\ast, (\omega)}(\zz {d}),
$$
with equality in norms.
%
%
\end{example}

\par

Let $\mascB$ be an invariant
QBF space on $\rr d$ with
respect to $r_0\in (0,1]$
and $v_0\in \mascP _E(\rr d)$.
Then the (discrete) 
convolution $*$, formally given by
$$ 
a*b = \sum_{j \in \zz d}
a(\cdot -j) b(j), 
\qquad 
a,b \in \ell _0' (\zz d),
$$
is well-defined on
$\ell _{(v_0)}^{r_0}(\zz d)
\times
\ell _{\mascB}(\zz d)
$
and satisfies
\begin{equation}\label{Eq:DiscConvEst}
\nm {a*b}{\ell _{\mascB}}
\le
C
\nm a{\ell _{(v_0)}^{r_0}}
\nm b{\ell _{\mascB}},
\qquad
a\in \ell _{(v_0)}^{r_0}(\zz d),
\,
b\in \ell _{\mascB}(\zz d),
\end{equation}
where the constant $C$ in
\eqref{Eq:DiscConvEst} only depends
on the constant $C$ in
\eqref{translmultprop1} and the
dimension $d$.

\par

\subsection{Wiener amalgam spaces and a
broad class of modulation spaces}
\label{subsec1.6}

\par

We first define Wiener amalgam spaces.

\par

\begin{defn}
\label{Def:WienAm}
Let $r\in (0,\infty ]$,
$\mascB$ be a normal QBF space on $\rr d$
and
$\omega \in \mascP _E(\rr {d})$.
Then the quasi-norm
$\nm f{\sfW ^{r}(\omega ,\mascB )}$
of $f\in \maclM (\rr {2d})$ is given
by
\begin{equation}
\label{Eq:WienerQuasiNormsSimpleExt}
\begin{aligned}
\nm f{\sfW ^{r}(\omega ,\mascB )}
&=
\nm f{\sfW ^{r}(\omega ,\ell _{\mascB})}
\equiv
\nm {a}{\ell _{\mascB}},\quad
\text{where}
\\[1ex]
a(j)
&\equiv
\nm {f \, \omega }{L^{r} (j+Q)},
\qquad
Q =[0,1] ^{2d},\ 
j\in \zz {2d}.
\end{aligned}
\end{equation}
The Wiener amalgam space
\begin{equation}
\label{Wiener-notation-1}
\sfW ^{r}(\omega ,\mascB )
=
\sfW ^{r}(\omega ,\ell _{\mascB} )
=
\sfW ^{r}(\omega ,\ell _{\mascB}(\zz {2d}))
\end{equation}
consists of all $f\in \maclM (\rr {2d})$
such that
$\nm f{\sfW ^{r}(\omega ,\ell _{\mascB})}$ is 
finite.
\end{defn}

\par

We have
\begin{equation}
\label{Eq:EmbWienAmLeb}
\sfW ^{\infty}(\omega ,\mascB )
\hookrightarrow
\mascB _{(\omega )}
\hookrightarrow
\sfW ^{r}(\omega ,\mascB )
\end{equation}
(see e.{\,}g. \cite[Proposition 2.1]{ToPfTe}).

\par

\begin{rem}
\label{Rem:EquivWienerNorms}
The cube $Q$ in Definition \ref{Def:WienAm}
can be replaced by any suitable convex set. 
In fact, let $r\in (0,\infty ]$,
$\omega \in \mascP _E(\rr {2d})$,
$p,q\in (0,\infty ]$
$\Lambda \subseteq \rr {2d}$
be a lattice, and let $\Omega
\subseteq \rr {2d}$ be a bounded
convex set such that
$$
\bigcup _{j\in \Lambda}(j+\Omega)
=\rr {2d}.
$$
Also let $a_{0,\Omega}
=
\{ a_{0,\Omega}(j) \} _{j\in \Omega}$,
$$
a_{0,\Omega}(f,j)\equiv
\nm {f\, \omega }{L^r(j+\Omega )},
\quad
j\in \Lambda ,
$$
and consider the quasi-norm
\begin{equation}\label{Eq:NormAltern}
\nmm f = \NM {\sum _{j\in \Lambda}
a_{0,\Omega }(f,j)\chi _{j+\Omega }}
{\mascB },
\end{equation}
when $f \in \maclM (\rr d)$.
By straight-forward computations, 
and using the fact that $\omega $ is moderate
it follows that 
$\nmm f \asymp \nm {f}
{\sfW ^r(\omega ,\mascB )}$,
(see \cite[Remark 1.19]{ToPfTe} or
e.{\,}g. \cite{Fei1981,Fei1,Hei}).
\end{rem}

\par

Our most general type of modulation spaces
are given in the following.

\par

\begin{defn}\label{Def:GenModSpace}
Let $\mascB$ be a normal QBF space on $\rr {2d}$, 
$\omega \in\mascP _E(\rr {2d})$,
and 
$\phi \in \Sigma _1(\rr d)\setm $. 
Then the \emph{modulation space} $M(\omega ,\mascB )$ consists
of all $f\in \Sigma _1'(\rr d)$ such that
\begin{equation}
\label{Eq:GenModNorm}
\nm f{M(\omega ,\mascB )}
\equiv \nm {V_\phi f\cdo \omega }{\mascB} <\infty .
\end{equation}
\end{defn}

\par

\begin{example}
Let $\mascB =L^{p,q}(\rr {2d})$, 
$p,q \in (0,\infty]$, and 
$\omega \in \mascP _E(\rr {2d})$.
Then 
$
M(\omega ,\mascB )
$
in Definition \ref{Def:GenModSpace}
becomes the classical modulation
space $M^{p,q}_{(\omega )}(\rr d)$,
see Definition \ref{Def:SpecialModSpace}.
\end{example}

\par

The following proposition explains
some recent norm equivalences for modulation spaces
given in Definition \ref{Def:GenModSpace}. We refer
to \cite[Theorem 3.1]{ToPfTe} for the proof.

\par

\begin{prop}
\label{Prop:EquivNorms2}
Let $\omega ,v,v_0\in \mascP _E(\rr 
{2d})$ be such that
$\omega$ is $v$-moderate, 
$\mascB$ be a normal QBF space
on $\rr {2d}$ with respect to
$r_0\in (0,1]$ and $v_0$,
$r_1,r_2\in [r_0,\infty ]$,
and let
$\phi _1,\phi _2\in M^{r_0}_{(v_0v)}
(\rr d)\setm $.
Then $M(\omega ,\mascB )$ is a quasi-Banach
space of order $r_0$,
\begin{equation}\label{Eq:EquivNorms21}
f\in M(\omega ,\mascB )
\quad \Leftrightarrow \quad
V_{\phi _1}f\cdo \omega \in \mascB
\quad \Leftrightarrow \quad
V_{\phi _2}f\in \sfW ^{r_1,r_2}
(\omega , \mascB ),
\end{equation}
and
\begin{equation}
\label{Eq:EquivNorms22}
\nm f{M(\omega ,\mascB )}
\asymp
\nm {V_{\phi _1}f\cdo \omega }{\mascB}
\asymp
\nm {V_{\phi _2}f}
{\sfW ^{r_1,r_2} (\omega , \mascB)},
\quad f\in \Sigma _1'(\rr d).
\end{equation}
\end{prop}

\par

The next proposition explains some
basic properties for modulation spaces given
in Definition \ref{Def:GenModSpace}.
For the proof of (1)--(3), see
Proposition 4.7, Theorem 4.8,
and their proofs in \cite{ToPfTe}. The assertion
(4) follows by using norms of the form
$f\mapsto \nm {V_{\phi _2}f}
{\sfW ^{r_1,r_2} (\omega , \mascB)}$
to the right
in \eqref{Eq:EquivNorms22} for involved
modulation spaces. Here
\begin{equation}
\label{Eq:InclusionMap}
\boldsymbol {\iota} :M(\omega _1,\mascB )
\to
M(\omega _2,\mascB )
\end{equation}
denotes the usual inclusion map from
space $M(\omega _1,\mascB )$ to
$M(\omega _2,\mascB )$, provided $M(\omega _1,\mascB )
\subseteq M(\omega _2,\mascB )$.

\par

\begin{prop}
\label{Prop:ContAndCompactModSp}
Let $\omega ,\omega _j\in \mascP _E(\rr 
{2d})$, and let $\mascB$, $\mascB _1$ and
$\mascB _2$ be normal QBF spaces
on $\rr {2d}$, $j=1,2$. Then the following is true:
\begin{enumerate}
\item $\Sigma _1(\rr d)\hookrightarrow
M(\omega ,\mascB )\hookrightarrow
\Sigma _1'(\rr d)$;

\vrum

\item if $\frac {\omega _2(X)}{\omega _1(X)}\le C$,
for some constant $C>0$, then
$M(\omega _1,\mascB )\subseteq M(\omega _1,\mascB )$,
and the map \eqref{Eq:InclusionMap} is continuous;

\vrum

\item if
$\lim _{|X|\to \infty}\frac {\omega _2(X)}
{\omega _1(X)}=0$, then
the map \eqref{Eq:InclusionMap} is compact;

\vrum

\item if $\ell _{\mascB_1}(\zz {2d})\hookrightarrow
\ell _{\mascB _2}(\zz {2d})$, then
$M(\omega ,\mascB _1)\hookrightarrow
M(\omega ,\mascB _2)$.
\end{enumerate}
\end{prop}

\par

\begin{rem}
If $p_j,q_j\in (0,\infty ]$ satisfy $p_1\le p_2$
and $q_1\le q_2$, then $\ell ^{p_1,q_1}
(\zz {2d})\hookrightarrow
\ell ^{p_2,q_2}(\zz {2d})$. Hence (4) in
Proposition \ref{Prop:ContAndCompactModSp}
gives the well-ordering relationship $M^{p_1,q_1}_{(\omega )}
(\rr {d})\hookrightarrow
M^{p_2,q_2}_{(\omega )}(\rr {d})$
between classical modulation spaces.
(See \cite{Gro2} in the Banach space case and
e.{\,}g. \cite{GaSa,Toft10} in the general case.)
\end{rem}

\par


\par

\section{Invariance properties for
$M^{\sharp ,q}_{(\omega )}$}\label{sec2}

\par

In this section we show that the space 
$M^{\sharp ,q}_{(\omega )}(\rr d)$ 
possess various invariance properties
concerning the choice of the window 
$\phi \in \Sigma _1(\rr d)\setminus 0$
and some norms in
\eqref{Eq:ModNormInfinityCond}. 
Additionally, in similar ways as in
Proposition \ref{Prop:EquivNorms2},
it follows that the window
in Definition \ref{Def:SpecialModSpace}
can be chosen in larger classes of 
function spaces. 

\par

We first introduce some notation.
By $\chi _\Omega$ we denote the 
characteristic function of the set 
$\Omega \subseteq \rr d$, i.{\,}e.,
$$
\chi _\Omega (x) =
\begin{cases}
1, & x\in \Omega ,
\\[1ex]
0, & x\notin \Omega.
\end{cases}
$$
For every $R,R_1>0$, let
\begin{equation}
\begin{aligned}
\label{Eq:SetOmegaR}
B_R
=
\sets {x\in \rr d}{|x|<R},
\qquad
\Omega _R
&=
\sets {(x,\xi )\in \rr {2d}}{x\notin B_R},
\\[1ex]
\maclI _R &= \Omega _R\cap \zz {2d},
\end{aligned}
\end{equation}
and (by a slight abuse of notation) let 
$\chi _{{}_{R}}$ be the
characteristic function of
$\Omega _R$, i.{\,}e.
\begin{equation}
\label{Eq:chiRDef}
\chi _{{}_{R}}(x,\xi )
=\chi _{\Omega _R}(x,\xi )
=
\begin{cases}
1, & |x| \ge R ,\ \xi \in \rr d,
\\[1ex]
0, & |x|<R ,\ \xi \in \rr d,
\end{cases}
\end{equation}

\par

\par

\begin{thm}
\label{Thm:IndepWind}
Let $q,r\in (0,\infty ]$ be such that
$r \geq \min (1,q)$,
$\omega ,v\in
\mascP _E(\rr {2d})$ be such that
$\omega$ is $v$-moderate,
$f\in M^{\sharp ,q}_{(\omega )}(\rr d)$, and
$\phi _1,\phi _2\in M^r_{(v)}(\rr d) \setminus 0$.
Then
\begin{equation}
\label{Eq:IndepWind}
\lim _{R\to \infty}
\nm {\chi _{{}_R}V_{\phi _1}f}
{L^{\infty ,q}_{(\omega )}}=0
\quad \Leftrightarrow \quad
\lim _{R\to \infty}
\nm {\chi _{{}_R}V_{\phi _2}f}
{\sfW ^r(\omega ,\ell ^{\infty ,q})}=0.
\end{equation}
\end{thm}

\par

\begin{cor}
\label{Cor:IndepWind}
Let $q\in (0,\infty ]$, $r=\min (1,q)$,
$\omega ,v\in
\mascP _E(\rr {2d})$ be such that
$\omega$ is $v$-moderate, and let
$f\in M^{\sharp ,q}_{(\omega )}(\rr d)$.
Then \eqref{Eq:ModNormInfinityCond} holds
true for any $\phi \in M^r_{(v)}(\rr d)$.
\end{cor}

\par

As in the proof of \cite[Theorem 3.1]{ToPfTe},
the case when $\phi _1$ and $\phi _2$ are both the (standard) Gaussian
\begin{equation}
\label{Eq:StandardGaussian}
\phi _0(x)=\pi ^{-\frac d4}e^{-\frac 12|x|^2},
\qquad x\in \rr d,
\end{equation}
plays a fundamental rule. This is given in
the following proposition.

\par

\begin{prop}
\label{Prop:GaussWind}
Let $q$, $r$, $\omega$, $v$, $\chi _{{}_R}$
be as in Theorem \ref{Thm:IndepWind},
and let $\phi _0$ be as in
\eqref{Eq:IndepWind}.
Then \eqref{Eq:IndepWind} holds for
$\phi _1=\phi _2=\phi _0$.
\end{prop}

\par

Proposition \ref{Prop:GaussWind} follows
by similar arguments as in the proof
of \cite[Proposition 3.2]{ToPfTe}. 
An essential step is to prove that
for every $R_1>0$, there is an
$R_2>R_1$ such that
\begin{equation}\label{Eq:RevIneq1}
\nm {\chi _{{}_{R_2}}V_{\phi _0}f}
{\sfW ^{\infty}
(\omega , \ell ^{\infty ,q})}
\lesssim
\nm {\chi _{{}_{R_1}}V_{\phi _0}f}
{\sfW ^{r} (\omega , \ell ^{\infty ,q})},
\quad
r\in (0,\infty ].
\end{equation}
\par

In order to be self-contained we have included a proof
of Proposition \ref{Prop:GaussWind}
in the Subsection \ref{subsec2.2}.

\par

For the proof of Theorem \ref{Thm:IndepWind} we also
need some estimates given in the following proposition.

\begin{prop} \label{Prop:NormEstimates}
Let $q,r\in (0,\infty ]$ be such that
$r \geq \min (1,q)$,
$\omega ,v\in
\mascP _E(\rr {2d})$ be such that
$\omega$ is $v$-moderate,
$f\in M^{\infty ,q}_{(\omega )}(\rr d)$, 
$\phi \in M^r_{(v)}(\rr d) \setminus 0$
and $\phi _0$ be given by \eqref{Eq:StandardGaussian}.
Then for any $R_0>0$ there is an $R>R_0$ such that
\begin{align}
\nm {\chi _{{}_{R}}V_{\phi }f}
{\sfW ^r (\omega , \ell ^{\infty ,q} )}
&\le
C \big (\nm {\chi _{{}_{R_0}}V_{\phi _0}\phi}
{\sfW ^\infty (v,\ell ^{r_0})}
\nm {f}
{M^{\infty ,q}_{(\omega )}}
\notag
\\[1ex]
&+
\nm {\phi}
{M^{r_0}_{(v)}}
\nm {\chi _{{}_{R_0}}V_{\phi _0}f}
{\sfW ^r (\omega , \ell ^{\infty ,q} )}\big )
\label{Eq:EquivNorms2A}
\end{align}
and
\begin{align}
\nm {\chi _{{}_{R}} &V_{\phi _0}f}
{\sfW ^r (\omega , \ell ^{\infty ,q} )}
\notag
\\[1ex]
&\le
C\Big (
\nm {\chi _{{}_{R_0}}V_{\phi _0}\phi}
{\sfW ^\infty (v,\ell ^{r_0})}
\nm {f}
{M^{\infty ,q}_{(\omega )}}
+
\nm {\phi}
{M^{r_0}_{(v)}}
\nm {\chi _{{}_{R_0}}V_\phi f}
{\sfW ^r(\omega ,\ell ^{\infty ,q})}
\Big )
\notag
\\[1ex]
&+
\frac 12
\nm {\chi _{{}_{R_0}}V_{\phi _0}f}
{\sfW ^\infty (\omega ,\ell ^{\infty ,q})}
\label{Eq:EquivNorms2B}
\end{align}
for $r=r_0$ and $r=\infty$.
\end{prop}

\par

The proof of Proposition \ref{Prop:NormEstimates} is 
given in Subsection \ref{subsec2.2}.

\par

\begin{proof}[Proof of Theorem 
\ref{Thm:IndepWind}]
We shall use several ideas as in the proof of
\cite[Theorem 3.1]{ToPfTe}.
Let 
\begin{equation}
\label{Eq:F0FPhiDef}
F_0 = |V_{\phi _0} f \cdot \omega |,
\quad F = |V_\phi f \cdot \omega |,
\quad
\Phi _1 = |V_{\phi _0} \phi \cdot v|,
\quad
\Phi _2=|V_\phi \phi _0\cdot v|,
\end{equation}
where 
$\phi \in M^r _{(v)}(\rr d)\setminus 0$
and $\phi _0$ is given by
\eqref{Eq:StandardGaussian},
and let $f\in \Sigma _1'(\rr d)$ be fixed. 

Using \eqref{Eq:EmbWienAmLeb}, the result
follows if we prove the implications
(1)--(4) in
\begin{equation}
\label{Eq:ProofDiagram}
\begin{matrix}
\underset{R\to \infty}\lim
\nm {\chi _{{}_{R}}V_{\phi _0}f} {\sfW ^{\infty}(\omega ,\ell ^{\infty ,q} )} =0&
\!\! \overset{(1)}{\phantom i\Leftrightarrow \phantom i}\!\! &
\underset{R\to \infty}\lim
\nm {\chi _{{}_{R}}V_{\phi _0}f}{\sfW ^{r_0}(\omega ,\ell ^{\infty ,q} )}=0
\\[1ex]
{\text{\scriptsize{$(2)$}}}
\Updownarrow & &
\Updownarrow
{\text{\scriptsize{$(3)$}}}
\\[1ex]
\underset{R\to \infty}\lim
\nm {\chi _{{}_{R}}V_{\phi}f} {\sfW ^{\infty}(\omega ,\ell ^{\infty ,q} )} =0&
\!\! \underset{(4)}{\phantom i\Rightarrow \phantom i}\!\! &
\underset{R\to \infty}\lim
\nm {\chi _{{}_{R}}V_{\phi}f}{\sfW ^{r_0}(\omega ,\ell ^{\infty ,q} )}=0.
\end{matrix}
\end{equation}
Evidently, (4)
hold true, and by
Proposition \ref{Prop:GaussWind} it also
follows that (1) is true. We need to prove
(2) and (3).
Let
\begin{alignat*}{2}
Q&=[0,1]^{2d},&\quad Q (j)&=j+Q,
\quad
\\[1ex]
Q_1&=
[-1,1]^{2d},&\quad
Q_1 (j )&=j+Q_1,
\\[1ex]
\alpha _{n,s}(j)& = \nm {\Phi _n}{L^s(Q 
(j))}, & &
\quad n = 1,2,
\\[1ex]
\beta _{0,s}(j )
&=
\nm {F_0}{L^s(Q (j ))}, &
\quad
\beta _{s}(j) &= \nm F{L^s(Q (j))},
\\[1ex]
\tilde \beta _{0,s}(j ) &=
\nm {F_0}{L^s(Q_1 (j))},&
\quad \text{and}\quad
\tilde \beta _{s}(j)
&=
\nm F{L^s(Q_1 (j ))},
\quad
j \in \zz {2d}, \ 
s \in (0,\infty].
\end{alignat*}

By $|V_\phi \phi _0 (X)|
=
|V_{\phi _0} \phi (-X)|$ and the fact that $v$ is even, 
we have
\begin{equation}\label{Eq:TildeAlfaIdent}
\alpha _{1,s}(j)
=
\alpha _{2,s}(-j).
\end{equation}

Since $\phi \in M^{r_0}_{(v)}(\rr d)$,
it follows that $V_{\phi _0}\phi
\in \sfW ^\infty (v,\ell ^{r_0})$,
which implies
$$
\lim _{R\to \infty}
\nm {\chi _{{}_{R}}V_{\phi _0}\phi}
{\sfW ^\infty (v,\ell ^{r_0})}
=0.
$$
By Proposition \ref{Prop:NormEstimates}, i.e. \eqref{Eq:EquivNorms2A} and 
\eqref{Eq:EquivNorms2B}, it follows that
$$
\nm {\chi _{{}_{R}} V_{\phi _0}f}
{\sfW ^r (\omega , \ell ^{\infty ,q} )}
\lesssim
\nm f{M^{\infty ,q}_{(\omega )}}
\quad \text{and}\quad
\nm {\chi _{{}_{R}} V_{\phi}f}
{\sfW ^r (\omega , \ell ^{\infty ,q} )}
\lesssim
\nm f{M^{\infty ,q}_{(\omega )}}
$$
are decreasing with $R$,
and bounded from below by $0$. Hence
the limits
$$
A(r)=\lim _{R\to \infty}
\nm {\chi _{{}_{R}} V_{\phi _0}f}
{\sfW ^r (\omega , \ell ^{\infty ,q} )}
\quad \text{and}\quad
B(r)=\lim _{R\to \infty}
\nm {\chi _{{}_{R}} V_{\phi}f}
{\sfW ^r (\omega , \ell ^{\infty ,q} )}
$$
exist.  If we show that
$$
A(r)=0
\quad \Leftrightarrow \quad
B(r)=0,
\qquad r\in [r_0,\infty ],
$$
then the relations (2) and (3) in 
\eqref{Eq:ProofDiagram} will follow, and
thereby the result.

By (1) in \eqref{Eq:ProofDiagram}
we have
$$
A(r_1)=0
\quad \Leftrightarrow \quad
A(r_2)=0,\qquad r_1,r_2\in [r_0,\infty ].
$$

\par

Suppose $A(r)=0$ for some $r\in [r_0,\infty ]$,
and thereby for every $r\in [r_0,\infty ]$. Then
\eqref{Eq:EquivNorms2A} gives
$$
0\le B(r)
\lesssim
\left (
\left (
\lim _{R\to \infty}
\nm {\chi _{{}_{R}}V_{\phi _0}\phi}
{\sfW ^\infty (v,\ell ^{r_0})}
\right )\nm f{M^{\infty ,q}_{(\omega )}}
+
\nm \phi{M^{r_0}_{(v)}}A(r)
\right )
=0
$$
This shows that $B(r)=0$ when $A(r)=0$.

\par

Suppose instead that $B(r)=0$.
Then \eqref{Eq:EquivNorms2B} gives
\begin{align*}
0&\le A(r)
\\[1ex]
&\le
C
\left (
\lim _{R\to \infty}
\nm {\chi _{{}_{R}}V_{\phi _0}\phi}
{\sfW ^\infty (v,\ell ^{r_0})}
\right )\nm f{M^{\infty ,q}_{(\omega )}}
+
C\nm \phi{M^{r_0}_{(v)}}B(r)
+
\frac 12 A(r)
\\[1ex]
&=
\frac 12A(r).
\end{align*}
This implies that $A(r)=0$,
and the result follows.
\end{proof}

\par

\subsection{Proofs of Propositions
\ref{Prop:GaussWind}
and \ref{Prop:NormEstimates}}
\label{subsec2.2}

\par

\begin{proof}[Proof of Proposition 
\ref{Prop:GaussWind}]
We have $\nm {\chi _{{}_R} V_{\phi _0}f}
{\sfW ^r(\omega ,\ell ^{\infty ,q})}$
is increasing with respect to $r$ and
$$
\nm {\chi _{{}_R} V_{\phi _0}f}
{\sfW ^r(\omega ,\ell ^{\infty,q})}
\le
\nm {\chi _{{}_R} V_{\phi _0}f}
{L^{\infty ,q}_{(\omega)}}
\le
\nm {\chi _{{}_R} V_{\phi _0}f}
{\sfW ^\infty (\omega ,\ell ^{\infty ,q})},
\quad
r=\min (1,q).
$$
It therefore suffices to prove
that for every $R_1>0$, there is an
$R_2>R_1$ such that \eqref{Eq:RevIneq1}
holds.

\par

Let 
$$
R_1>0,\quad
F_0=|V_{\phi _0} f\cdot \omega |,
\quad
Q=[0,1]^{2d}
\quad \text{and}\quad
Q_r=[-r,r]^{2d},
\quad r>0.
$$
Also let $R_{12}=2^{2d}R_1=4^dR_1$,
$R_2=8^{2d}R_{12}=64^dR_{12}$
and
$$
\Omega _R=\sets {(j,\iota )\in \zz {2d}}
{|j|\ge R}.
$$
Then
$$
B_{R_1} (0)+[-2,2]^{2d}\subseteq
B_{R_2} (0).
$$
By Lemma \ref{Lemma:EstSTFT}
and the fact that $\omega$ is
moderate it follows that there is an
$X_j\in j+Q$ with
$j\in \Omega _{R_2}$ such that
\begin{align*}
\nm {F_0}{L^\infty (j+Q)}
&\asymp
\nm {V_{\phi _0}f}{L^\infty (j+Q)}
\omega (j)
\\[1ex]
&=
|{V_{\phi _0}f}(X_j)|\omega (j)
\lesssim
\nm {V_{\phi _0}f}{L^r (B_1(X_j))}
\omega (j)
\\[1ex]
&\le
\nm {V_{\phi _0}f}{L^r (j+Q_2)}
\omega (j)
\asymp
\nm {F_0}{L^r (j+Q_2)}.
\end{align*}
From this estimate and Remark
\ref{Rem:EquivWienerNorms} we get
\begin{align*}
\nm {\chi _{{}_{R_2}}V_{\phi _0}f}
{\sfW ^\infty (\omega ,\ell ^{\infty ,q})}
&=
\nm { \{ \nm {F_0}
{L^\infty (j+Q)}\}
_{j\in \Omega _{R_2}}}{\ell ^{\infty ,q}}
\\[1ex]
&\lesssim
\nm { \{ \nm {F_0}
{L^r (j+Q_2)}\}
_{j\in \Omega _{R_2}}}{\ell ^{\infty ,q}}
\\[1ex]
&\lesssim
\nm { \{ \nm {F_0}
{L^r (j+Q)}\}
_{j\in \Omega _{R_{12}} }}
{\ell ^{\infty ,q}}
\lesssim
\nm {\chi _{{}_{R_1}}V _{\phi _0} f}
{\sfW ^r(\omega ,\ell ^{\infty ,q})},
\end{align*}
and \eqref{Eq:RevIneq1}, and thereby the
result follows.
\end{proof}

\par


\par

\begin{proof}
[Proof of Proposition \ref{Prop:NormEstimates}]
We first prove \eqref{Eq:EquivNorms2A}.

\par

By e.{\,}g. \cite[Theorem 3.1]{ToPfTe} one has
\begin{alignat}{2}
\nm {\alpha _{n,s}}
{\ell ^{r_0}}
&=
\nm {\Phi _n}
{\sfW ^s (\ell ^{r_0}) }
\asymp
\nm \phi{M^{r_0}_{(v)}}, & \quad n&=1,2,
\label{Eq:NormWindow}
\\[1ex]
\nm {\tilde{\beta} _{0,s}}{\ell ^{\infty ,q}}
&\asymp
\nm {\beta _{0,s}}{\ell ^{\infty ,q}}
=
\nm {F_0}{\sfW ^s (\ell ^{\infty ,q} )}
\asymp \nm f{M^{\infty,q}_{(\omega )}}, &&
\notag
\intertext{and}
\nm {\tilde{\beta} _{s}}{\ell ^{\infty ,q}}
&\asymp
\nm {\beta _{s}}{\ell ^{\infty ,q}}
=
\nm F{\sfW ^s (\ell ^{\infty ,q})},
\asymp \nm f{M^{\infty,q}_{(\omega )}}, &
\qquad
s &\in (0,\infty] .
\notag
\end{alignat}
It follows from
\eqref{Eq:ChangeWindow} that
$F \lesssim   F _0 \ast \Phi_2$.
A combination of these
relations, \eqref{Eq:DiscConvEst} and 
\eqref{Eq:NormWindow} gives that for any $R_0>0$, 
there are suitable $R_1>R_0+2^d$
$R_2>2R_1$ and $R_3>R_2+2^d$ such that 
\begin{align}
\nm {\chi _{{}_{R_3}}F}
{\sfW ^{\infty} (\ell ^{\infty ,q} )}^{r_0}
&\le 
\nm {\beta _{\infty}}
{\ell ^{\infty ,q} (\maclI _{R_2})}^{r_0}
\notag
\\[1ex]
&\lesssim
\nm { \alpha_{2,\infty}*\tilde{\beta}_{0,\infty} }
{\ell ^{\infty ,q} (\maclI _{R_2})}^{r_0} 
\le
J_{1,R_1}+J_{2,R_1},
\intertext{where}
J_{1,R_1}
&=
\nm { (\chi _{{}_{R_1}}\alpha_{2,\infty})
*
\tilde{\beta}_{0,\infty} }
{\ell ^{\infty ,q} (\maclI _{R_2})}^{r_0}
\intertext{and}
J_{2,R_1}
&=
\nm { ((1-\chi _{{}_{R_1}})\alpha_{2,\infty})
*
\tilde{\beta}_{0,\infty} }
{\ell ^{\infty ,q} (\maclI _{R_2})}^{r_0}.
\label{Eq:J2R1Def}
\end{align}
Here $\maclI _R$ is the same as in
\eqref{Eq:SetOmegaR}.

\par

For $J_{1,R_1}$ we have
\begin{align}
J_{1,R_1}
&\lesssim
\nm { \chi _{{}_{R_1}}\alpha_{2,\infty} }
{\ell ^{r_0}}^{r_0}
\nm { \tilde{\beta}_{0,\infty}  }
{\ell ^{\infty ,q}}^{r_0}
\notag
\\[1ex]
&\lesssim
\nm {\chi _{{}_{R_0}}\Phi_2}
{\sfW ^{\infty}
(\ell ^{r_0})}^{r_0} 
\nm {f} {M^{\infty ,q}_{(\omega )}}^{r_0}.
\label{Eq:EstJ1R1}
\end{align}

\par

To estimate $J_{2,R_1}$, we observe that
\begin{align*}
\gamma (j,\iota )
&\equiv
\sum _{(k,\kappa )\in \zz {2d}}
(1-\chi _{R_1}(j-k,\iota -\kappa ))
\alpha_{2,\infty}(j-k,\iota -\kappa )
\tilde{\beta}_{0,r}(k,\kappa )
\\[1ex]
&=
\sum _{k\in j+\zz {d}\cap B_{R_1}}
\left (
\sum _{\kappa \in \zz {d}}
\alpha_{2,\infty}(j-k,\iota -\kappa )
\tilde{\beta}_{0,r}(k,\kappa )
\right ),
\qquad (j,\iota ) \in \maclI _{R_2}.
\end{align*}
In the previous sums we have
$k\in j+\zz {d}\cap B_{R_1}$ and
$(j,\iota )\in \maclI _{R_2}$,
giving that
$$
|k-j|<R_1
\quad \text{and}\quad
|j|\ge R_2,
$$
which implies that
\begin{equation}
\label{Eq:BackwardTriangle1}
|k|\ge |j|-|k-j|>R_2-R_1>R_1,
\end{equation}
due to the assumptions. Hence
$$
\gamma (j,\iota )
\le
\sum _{k\in \maclI _{R_1}}
\left (
\sum _{\kappa \in \zz {d}}
\alpha_{2,\infty}(j-k,\iota -\kappa )
\tilde{\beta}_{0,r}(k,\kappa )
\right )
=
\alpha_{2,\infty}
*
(\chi _{R_1}\tilde{\beta}_{0,r}).
$$
A combination of the latter estimate
with \eqref{Eq:NormWindow} and
\eqref{Eq:J2R1Def} gives
\begin{equation}
\label{Eq:J2R1Est}
\begin{aligned}
J_{2,R_1}
&\le
\nm {\alpha_{2,\infty} }
{\ell ^{r_0}}^{r_0}
\nm {\chi _{{}_{R_1}}\tilde{\beta}_{0,r}}
{\ell ^{\infty ,q}}^{r_0}
\\[1ex]
&\lesssim
\nm \phi{M^{r_0}_{(v)}}
\nm {\chi _{{}_{R_0}}F_0}
{\sfW ^\infty (\ell ^{\infty ,q})}.
\end{aligned}
\end{equation}
The estimate \eqref{Eq:EquivNorms2A} in the case
$r=\infty$ now follows
by combining \eqref{Eq:EstJ1R1} with
\eqref{Eq:J2R1Est}. For general $r$, we may
now use \eqref{Eq:RevIneq1} to get that for
any $R_0>0$ there are $R_2,R_1>0$ such that
$R_2>R_1>R_0$ and
\begin{align}
\nm {\chi _{{}_{R_2}}&F} {\sfW ^{r} (\ell ^{\infty ,q} )}
\le
\nm {\chi _{{}_{R_2}}F}
{\sfW ^{\infty} (\ell ^{\infty ,q} )}
\notag
\\[1ex]
&\lesssim
\nm {\chi _{{}_{R_1}}\! V_{\phi _0}\phi}
{\sfW ^\infty (v,\ell ^{r_0})}
\nm { F _0} {\sfW ^{\infty}
(\ell ^{\infty ,q} )}
+\nm {V_{\phi _0}\phi}
{\sfW ^\infty (v,\ell ^{r_0})}
\nm {\chi _{{}_{R_1}}F_0} {\sfW ^{\infty}
(\ell ^{\infty ,q} )}
\notag
\\[1ex]
&=
\nm {\chi _{{}_{R_1}}\! V_{\phi _0}\phi}
{\sfW ^\infty (v,\ell ^{r_0})}
\nm {f} {M^{\infty ,q}_{(\omega )} }
+\nm {\phi}
{M^{r_0}_{(v)}}
\nm {\chi _{{}_{R_1}}F_0} {\sfW ^{\infty}
(\ell ^{\infty ,q} )}
\notag
\\[1ex]
&=
\nm {\chi _{{}_{R_0}}\! V_{\phi _0}\phi}
{\sfW ^\infty (v,\ell ^{r_0})}
\nm {f} {M^{\infty ,q}_{(\omega )} }
+\nm {\phi}
{M^{r_0}_{(v)}}
\nm {\chi _{{}_{R_0}}F_0} {\sfW ^r
(\ell ^{\infty ,q} )}.
\label{eq:estimates-for-(2.10)}
\end{align}
and \eqref{Eq:EquivNorms2A} follows.

\par

It remains to prove \eqref{Eq:EquivNorms2B}.
First we consider the case when
$r\in [r_0,1]$.
The inequality $F_0 \lesssim  F 
\ast \Phi_1$ (cf. \eqref{Eq:ChangeWindow}) gives
\begin{align*}
\beta_{0,r} (j)^r
&\lesssim
\int_{j+Q}
\left ( 
\int _{\rr {2d}} F(X-Y) \, 
\Phi _1 (Y) \, dY
\right )^r
\, dX
\nonumber
\\[1ex]
&=
\int _{j+Q}
\left ( 
\sum_{k\in \zz {2d}}
\int_{k+Q} F(X-Y) \, 
\Phi_1 (Y) \, dY
\right ) ^r
\, dX
\nonumber
\\[1ex]
&\lesssim
\sum_{k\in \zz {2d}} 
\alpha_{1,\infty} (k)^r
\gamma _r(j-k)^r ,
\end{align*}
where
$$
(\gamma _r(j))^{r}
=
\left (
\int _{j+Q} \left( 
\int _{Q} F(X-Y)  
\, dY
\right )^r \, dX
\right ).
$$
Let
$$
R_1>R_0+4^d,
\quad
R_2>2R_1+4^d,
\quad
R_3>2R_2
\quad \text{and}\quad
R_4>R_3+4^d,
$$
and suppose that $|j|\ge R_3$.
Since $r_0\le r$, we obtain
$$
\beta_{0,r} (j)^{r_0}
\lesssim
\left (
(\alpha _{1,\infty}^r*\gamma _r^r)(j)
\right )^{\frac {r_0}r}
\le
(\alpha _{1,\infty}^{r_0}
*\gamma _r^{r_0})(j)
=
\mu _{1,R_2}(j)+\mu _{2,R_2}(j),
$$
where
\begin{align*}
\mu _{1,R_2}(j)
&=
((\chi _{{}_{R_2}}\alpha _{1,\infty})^{r_0}
*\gamma _r^{r_0})(j)
\intertext{and}
\mu _{2,R_2}(j)
&=
(((1-\chi _{{}_{R_2}})\alpha _{1,\infty})^{r_0}
*\gamma _r^{r_0})(j).
\end{align*}
By \eqref{Eq:BackwardTriangle1} we obtain
$$
\mu _{2,R_2}(j)
\le
(\alpha _{1,\infty}^{r_0}
*(\chi _{{}_{R_2}}\gamma _r)^{r_0})(j).
$$

\par

Let $q_0=\frac q{r_0}\ge 1$.
By applying the $\ell ^{\infty ,q}$ norm
and using
\eqref{Eq:DiscConvEst}, we obtain
\begin{align}
\nm {\chi _{{}_{R_4}}
F_0}{\sfW ^r (\ell ^{\infty ,q})}^{r_0}
&\lesssim
\nm {\chi _{{}_{R_3}}
\beta _{0,r}}{\ell ^{\infty ,q}}^{r_0}
\notag
\\[1ex]
&=
\nm {\chi _{{}_{R_3}}\beta _{0,r}^{r_0}}
{\ell ^{\infty ,q_0}}
\lesssim
\nm {\chi _{{}_{R_3}}(\alpha _{1,\infty}^{r_0}
*\gamma _r^{r_0})}
{\ell ^{\infty ,q_0}}
\notag
\\[1ex]
&\le
\nm {
((\chi _{{}_{R_2}}\alpha _{1,\infty})^{r_0}
*\gamma _r^{r_0})}
{\ell ^{\infty ,q_0}}
+
\nm {\alpha _{1,\infty}^{r_0}
*(\chi _{{}_{R_2}}\gamma _r)^{r_0}}
{\ell ^{\infty ,q_0}}
\notag
\\[1ex]
&\lesssim
J_{3,R_2}+J_{4,R_2},
\label{Eq:beta0rEst}
\end{align}
where
\begin{align}
J_{3,R_2}
&=
\nm {\chi _{{}_{R_2}}\alpha _{1,\infty}^{r_0}}
{\ell ^1}\nm {\gamma _r^{r_0}}
{\ell ^{\infty ,q_0}}
\notag
\\[1ex]
&=
\nm {\chi _{{}_{R_2}}\alpha _{1,\infty}}
{\ell ^{r_0}}^{r_0}\nm {\gamma _r}
{\ell ^{\infty ,q}}^{r_0}
\lesssim
\nm {\chi _{{}_{R_1}}V_{\phi _0}\phi \cdot v}
{L^{r_0}}^{r_0}
\nm f{M^{\infty ,q}_{(\omega )}}^{r_0}
\label{Eq:J3R1}
\intertext{and}
J_{4,R_2}
&=
\nm {\alpha _{1,\infty}^{r_0}}{\ell ^1}
\nm {\chi _{{}_{R_2}}\gamma _r^{r_0}}
{\ell ^{\infty ,q_0}}
\notag
\\[1ex]
&=
\nm {\alpha _{1,\infty}}{\ell ^{r_0}}^{r_0}
\nm {\chi _{{}_{R_2}}\gamma _r}
{\ell ^{\infty ,q}}^{r_0}
\lesssim
\nm {\phi}{M^{r_0}_{(v)}}^{r_0}
\nm {\chi _{{}_{R_2}}\gamma _r}
{\ell ^{\infty ,q}}^{r_0}
\label{Eq:J4R1}
\end{align}
We need to estimate
$\nm {\chi _{{}_{R_2}}\gamma _r}
{\ell ^{\infty ,q}}^{r_0}
$ in \eqref{Eq:J4R1}.
Let $\theta >0$ be arbitrary.
By 
$
F=F^r \, F^{1-r}
$
we have
\begin{equation}
\label{Eq:gammarComputations}
\begin{aligned}
\gamma _r (j)^{r_0}
&=
\left (
\int_{j+Q}
\left( 
\int_{Q} F(X-Y)  \, dY
\right )^r
\, dX
\right )^{\frac {r_0}r}
\\[1ex]
&\le 
\tilde{\beta}_{\infty}(j) ^{r_0(1-r)}
\left (
\int_{j+Q}
\left( 
\int_{Q} F(X-Y)^r
\, dY 
\right)^r
\, dX 
\right )^{\frac {r_0}r}
\\[1ex]
&\qquad 
\leq
\tilde{\beta}_{r}(j) ^{r_0r}
\tilde{\beta}_{\infty}(j) ^{r_0(1-r)}
\\[1ex]
&\qquad 
\leq
r\theta ^{\frac 1r}\tilde{\beta} _{r}(j)^{r_0}
+
(1-r)\theta ^{-\frac 1{1-r}}
\tilde{\beta}_{\infty}(j)^{r_0},
\end{aligned}
\end{equation}
where the last inequality follows
from the arithmetic-geometric
inequality.

\par

By using \eqref{Eq:gammarComputations} in
\eqref{Eq:beta0rEst} gives
\begin{equation}
\label{Eq:ChiR1gamma}
\begin{aligned}
\nm {\chi _{{}_{R_2}}\gamma _r^{r_0}}
{\ell ^{\infty ,q_0}}
&\lesssim
r\theta ^{\frac 1r}
\nm {\chi _{{}_{R_2}}\tilde{\beta} _{r}^{r_0}}
{\ell ^{\infty ,q_0}}
+
(1-r)\theta ^{-\frac 1{1-r}}
\nm {\chi _{{}_{R_2}}\tilde{\beta}  _{\infty}^{r_0}}
{\ell ^{\infty ,q_0}}
\\[1ex]
&=
r\theta ^{\frac 1r}
\nm {\chi _{{}_{R_2}}\tilde{\beta} _{r}}
{\ell ^{\infty ,q}}^{r_0}
+
(1-r)\theta ^{-\frac 1{1-r}}
\nm {\chi _{{}_{R_2}}\tilde{\beta}  _{\infty}}
{\ell ^{\infty ,q}}^{r_0}.
\end{aligned}
\end{equation}

\par

For the expression
$\nm {\chi _{{}_{R_2}}\tilde{\beta} _{r}}
{\ell ^{\infty ,q}}^{r_0}$ in
\eqref{Eq:ChiR1gamma} we have
\begin{equation}
\label{Eq:ChiR1beta}
\nm {\chi _{{}_{R_2}}\tilde{\beta} _{r}}
{\ell ^{\infty ,q}}
\lesssim
\nm {\chi _{{}_{R_1}}F}
{\sfW ^r(\ell ^{\infty ,q})}.
\end{equation}
We need to find suitable estimate of
$\nm {\chi _{{}_{R_2}}\tilde{\beta} _{\infty}}
{\ell ^{\infty ,q}}^{r_0}$ in
\eqref{Eq:ChiR1gamma}. If $R_{12}=R_2-4^d$,
then $R_{12}>2R_1$, and
$$
\nm {\chi _{{}_{R_2}}\tilde{\beta} _{\infty}}
{\ell ^{\infty ,q}}^{r_0}
\le
\nm {\chi _{{}_{R_{12}}}\beta _{\infty}}
{\ell ^{\infty ,q}}^{r_0}.
$$
A combination of \eqref{Eq:NormWindow}
and the estimates
\eqref{Eq:J2R1Def}--\eqref{Eq:J2R1Est}
gives
\begin{multline}
\label{Eq:CircleEst}
\nm {\chi _{{}_{R_{12}}}\beta _{\infty}}
{\ell ^{\infty ,q}}^{r_0}
\\[1ex]
\lesssim
\nm {\chi _{{}_{R_0}}\Phi _2}
{\sfW ^{\infty}
(\ell ^{r_0})}^{r_0} 
\nm {f} {M^{\infty ,q}_{(\omega )}}^{r_0}
+
\nm \phi{M^{r_0}_{(v)}}^{r_0}
\nm {\chi _{{}_{R_0}}F_0}
{\sfW ^\infty (\ell ^{\infty ,q})}^{r_0}.
\end{multline}

\par

Let $\theta >0$ be arbitrary. By
\eqref{Eq:beta0rEst}--\eqref{Eq:CircleEst}
it follows that
\begin{align}
\nm {\chi _{{}_{R_4}}
&F_0}{\sfW ^r (\ell ^{\infty ,q})}
\lesssim
(J_{3,R_2}+J_{4,R_2})^{\frac 1{r_0}}
\notag
\\[1ex]
&\lesssim
\left (
\nm {\chi _{{}_{R_4}}\Phi _2}{L^{r_0}}
\nm {f} {M^{\infty ,q}_{(\omega )}}
+
\nm \phi{M^{r_0}_{(v)}}
\nm {\chi _{{}_{R_2}}\gamma _r}
{\ell ^{\infty ,q}}
\right )
\notag
\\[1ex]
&\lesssim
\nm {\chi _{{}_{R_4}}\Phi _2}{L^{r_0}}
\nm {f} {M^{\infty ,q}_{(\omega )}}
+
\theta ^{\frac 1r}
\nm \phi{M^{r_0}_{(v)}}
\nm {\chi _{{}_{R_1}}F}
{\sfW ^r (\ell ^{\infty ,q})}
\notag
\\[1ex]
&+
\theta ^{-\frac 1{1-r}}
\nm \phi{M^{r_0}_{(v)}}
\nm {\chi _{{}_{R_0}}
F_0}{\sfW ^\infty (\ell ^{\infty ,q})},
\label{Eq:TurnArgument}
\end{align}
where the hidden  constants do not 
depend on $f$, $\phi$ and $\theta$.
The estimate \eqref{Eq:EquivNorms2B}
now follows by choosing $\theta$ large
enough, and the result follows in
the case $r\in [r_0,1)$.

\par

If $r=1$, then \eqref{Eq:EquivNorms2B}
is obtained by choosing $\theta =1$
in the computations which led to
\eqref{Eq:TurnArgument}, and the result
follows in this case as well.

\par

\par

For $r\in (1,\infty]$,
\cite[Theorem 3.1]{ToPfTe}, 
\eqref{Eq:EquivNorms2B}
and the fact that the Wiener amalgam spaces 
$\sfW ^r (\ell ^{\infty ,q})$
decrease with $r$ gives
$$
\nm {\chi _{{}_{R}}F_0}
{\sfW ^r (\ell ^{\infty ,q})}
\asymp
\nm {\chi _{{}_{R}}F_0}
{\sfW ^1 (\ell ^{\infty ,q})}
\lesssim
\nm {\chi _{{}_{R}}F}
{\sfW ^1 (\ell ^{\infty ,q})}
\le
\nm {\chi _{{}_{R}}F}
{\sfW ^r (\ell ^{\infty ,q})},
$$
giving the result.
\end{proof}

\par

\section{Basic properties of
$M^{\sharp ,q}_{(\omega )}$}\label{sec3}

\par

In this section we deduce different basic
properties for $M^{\sharp ,q}_{(\omega )}$.
First we show that $M^{\sharp ,q}_{(\omega )}$
is a quasi-Banach space of order $\min (1,q)$.
Then we show that $M^{p,q}_{(\omega )}$
is contained in $M^{\sharp ,q}_{(\omega )}$
when $p<\infty$. Thereafter we proceed with
various embedding
properties for such spaces.
In particular, we derive certain
identifications of compactly
supported elements
in $M^{\sharp ,q}_{(\omega )}$. Finally
we find convolution properties for
$M^{\sharp ,q}_{(\omega )}$.


\par

\subsection{Topological and basic embedding
properties for $M^{\sharp ,q}_{(\omega )}$}

\begin{prop}
\label{Prop:ModComplete}
Let $q\in (0,\infty ]$ and
$\omega \in \mascP _E(\rr {2d})$.
Then $M^{\sharp ,q}_{(\omega )}(\rr d)$
is a quasi-Banach space of order $r = \min (1,q)$.
\end{prop}

\par

\begin{rem}
Evidently, if $q \geq 1$ 
in Proposition
\ref{Prop:ModComplete} it follows that
$M^{\sharp ,q}_{(\omega )}(\rr d)$ is a Banach
space.
\end{rem}

\par

\begin{proof}
We only prove the result for $q<\infty$.
The case when $q=\infty$ follows by similar
arguments and is left for the reader.
Let  
$\{f_j\} _{j=1} ^{\infty}$ 
be a Cauchy sequence in 
$
M^{\sharp, q}_{(\omega)}(\rr d) 
\subseteq
M^{\infty, q}_{(\omega)}(\rr d) 
$.
Since $M^{\infty, q}_{(\omega)}(\rr d)$ is
complete 
there is an 
$f \in M^{\infty, q}_{(\omega)}(\rr d)$ 
such that 
\begin{equation}
\NM{f_j-f}{ M^{\infty, q}_{(\omega)} }
\to 0
\quad \text{as}\quad j\to \infty .
\end{equation}
We need to show that 
$f \in M^{\sharp, q}_{(\omega)}(\rr d)$.

\par

Let $r=\min (1,q)$,
$\ep >0$, $v\in \mascP _E(\rr {2d})$
be submultiplicative such that $\omega$ is $v$ moderate,
$
\phi
\in M^q_{(v)}(\rr d)
\setm
$
and choose $j \geq 1$ such that
\begin{align}\label{eq:ConvCS}
    \NM{ V_{\phi}(f_j-f) }{ L^{\infty, q} _{(\omega)} }
    <
    \frac \ep 2.
\end{align}
Since $f_j \in M^{\sharp, q}_{(\omega)} (\rr d)$ 
we can choose $R_0>0$ such that
\begin{align}
J_{1, R}
 &=
 \left(
 \int _{\rr d}
 \sup _{|x| \geq R} |
 V_{\phi} f_j(x,\xi )\cdot
 \omega
 (x,\xi)|^q
 \, d\xi
 \right)^{\frac rq}
 < 
 \frac \ep 2,
\label{Eq:J2RDef2}
\intertext{when $R \geq R_0$. 
The estimate \eqref{eq:ConvCS} implies}
J_{2,R}
&=
\left(
 \int _{\rr d}
 \sup _{|x| \geq R} |
 V_{\phi} (f-f_j)(x,\xi )\cdot
 \omega (x,\xi)|^q
 \, d\xi
 \right)^{\frac rq}
    <
    \frac \ep 2,
\label{Eq:J1RDef2}
\end{align}
when $R \geq R_0$.
For $R > R_0$ we get by \eqref{Eq:J2RDef2} 
and \eqref{Eq:J1RDef2}
\begin{align*}
\left(
 \int _{\rr d}
 \sup _{|x| \geq R} |
 V_{\phi} f(x,\xi )\cdot \omega
 (x,\xi)|^q
 \, d\xi
 \right)^{\frac rq}
 \leq 
J_{1,R} + J_{2,R}
< \ep. 
\end{align*}
This shows that 
\begin{align*}
    \left(
 \int _{\rr d}
 \sup _{|x| \geq R} 
 |V_{\phi} f(x,\xi )\cdot
 \omega (x,\xi)|^q
 \, d\xi
 \right)^{\frac rq}
 \to
 0
 \quad \text{as }
 R \rightarrow \infty
\end{align*}
which implies 
$f \in M^{\sharp,q}_{(\omega)}(\rr d)$,
and the result follows. 
\end{proof}

\par

Next we show that $M^{p,q}_{(\omega )}$
is contained in $M^{\sharp ,q}_{(\omega )}$
when $p<\infty$.

\par

\begin{prop}
\label{Prop:ModEmb}
Let $p,q,q_1,q_2\in (0,\infty ]$
and
$\omega \in \mascP _E(\rr {2d})$.
Then the following is true:
\begin{enumerate}
\item if $p<\infty$, then
$M^{p,q}_{(\omega )}(\rr d)
\hookrightarrow M^{\sharp ,q}_{(\omega )}(\rr d)$;

\vrum

\item if $q_1\le q_2$, then
$M^{\sharp ,q _1}_{(\omega )}(\rr d)
\hookrightarrow M^{\sharp ,q_2}_{(\omega )}(\rr d)$.
\end{enumerate}
\end{prop}

\par

\begin{proof}
First we prove (1).
We only prove the result for $q<\infty$. 
The case $q=\infty$ follows by similar
arguments and is left for the reader. 
Since the modulation spaces increase with the
Lebesgue parameter, it is enough to show the
embeddings for   $p\geq 1$.

\par

Let $\Omega _R$ be as in \eqref{Eq:SetOmegaR},
$f \in M^{p,q}_{(\omega)}(\rr d)$ and let
$ r = \min (1,q)$. Also let $v$, $\phi$, $\phi _0$,
$F$ and $F_0$ be the same as in the proof of Theorem
\ref{Thm:IndepWind}, and set 
$$
\Phi =|V_\phi \phi _0\cdot v|.
$$
Then 
\begin{align*}
    \Phi \in 
    L^{r}(\rr d)\bigcap L^{\infty}(\rr d),
\end{align*}
and by \cite[Section 12.1]{Gro2} it follows that
$$
|V_\phi f| \le |V_{\phi _0}f|*|V_\phi \phi _0|,
$$
which implies
\begin{equation}
\label{Eq:STFTConv}
0\le F\lesssim F_0*\Phi . 
\end{equation}
For $R_1>0$ and $R>2R_1$, let 
\begin{align}\label{Eq:JRDef1}
J_R
&=
\left(
\int_{\rr d}
\left (\sup_{|x|\geq R} F(x,\xi )
\right )^q\, d\xi
\right )^{\frac{r}{q}}.
\intertext{Then \eqref{Eq:STFTConv}
implies that $J_R \lesssim J_{1,R,R_1}+J_{2,R,R_1}$,
where}
J_{1,R,R_1}
&=
\left ( \int_{\rr d}
\left( \sup_{|x|\geq R} 
\iint _{\Omega _{R_1}} 
F_0(x-y, \xi-\eta) 
\Phi(y, \eta) \, dy d\eta \right)^q
\, d\xi\right)^{\frac{r}{q}}
\label{Eq:J1RDef1}
\intertext{and}
J_{2,R,R_1}
&=
\left ( \int_{\rr d}
\left( \sup_{|x|\geq R}
\iint _{\complement \Omega _{R_1}}
F_0(x-y, \xi-\eta)
\Phi (y, \eta) \, dy d\eta \right)^q
\, d\xi \right)^{\frac{r}{q}}.
\label{Eq:J2RDef1}
\end{align}

\par

We intend to prove that $J_R\to 0$ as $R\to \infty$,
which follows if we prove 
\begin{equation}
\label{Eq:J12RR1ep}
J_{1,R,R_1},J_{2,R,R_1} \to 0,
\quad \text{as}\quad
R_1\to \infty \ \text{and}\ R>2R_1.
\end{equation}
Then $f\in M^{\sharp,q}_{(\omega)}(\rr d)$,
and the result follows.


\par

Let $\chi _{{}_R}$
be as in \eqref{Eq:chiRDef}. Then
a slight reformulation of $J_{1,R,R_1}$
gives
$$
J_{1,R,R_1} = \left ( \int_{\rr d}
\sup_{|x|\geq R} \left (
F_0*(\chi _{{}_{R_1}}\Phi ) (x,\xi )
\right  )^q 
\, d\xi\right)^{\frac{r}{q}}
\le
\left (
\nm {F_0*(\chi _{{}_{R_1}}\Phi )}{L^{\infty ,q}}
\right )^{\frac rq}.
$$
Hence \cite[Corollary 5.2]{ToPfTe} gives
\begin{align}
J_{1,R,R_1}
&\le
\left (
\nm {F_0*(\chi _{{}_{R_1}}\Phi )}{L^{\infty ,q}}
\right )^{\frac rq}
\le
\left (
\nm {F_0*(\chi _{{}_{R_1}}\Phi )}
{\sfW ^\infty (\ell ^{\infty ,q})}
\right )^{\frac rq}
\notag
\\[1ex]
&\le
\left (
\nm {F_0}{\sfW ^\infty (\ell ^{\infty ,q})}
\nm {\chi _{{}_{R_1}}\Phi}
{\sfW ^1 (\ell ^{1,r})}
\right )^{\frac rq}
\notag
\\[1ex]
&\le
\left (
\nm {F_0}{\sfW ^\infty (\ell ^{\infty ,q})}
\nm {\chi _{{}_{R_1}}\Phi}
{\sfW ^\infty  (\ell ^{r})}
\right )^{\frac rq}.
\label{Eq:J1RR1Est}
\end{align}

\par

Since $\phi \in M^r_{(v)}(\rr d)$, we have
$$
V_\phi \phi _0 \cdot v \in L^r(\rr {2d})
$$
which is equivalent to
$$
V_\phi \phi _0 \cdot v \in \sfW ^\infty (\ell ^r(\zz {2d})),
$$
in view of \cite[Theorem 3.1]{ToPfTe}. Using the
fact that $r<\infty$,
we get
$$
\lim _{R_1\to \infty }
\nm {\chi _{{}_{R_1}}\Phi}
{\sfW ^\infty  (\ell ^{r})}
=
\lim _{R_1\to \infty }
\nm {\chi _{{}_{R_1}}V_\phi \phi _0 \cdot v}
{\sfW ^\infty  (\ell ^{r})}
=0.
$$
A combination of this limit and \eqref{Eq:J1RR1Est}
shows that
$$
J_{1,R,R_1} \to 0,
\quad \text{as}\quad
R_1\to \infty \ \text{and}\ R>2R_1.
$$

\par

Since $|x-y|\ge R_1$ when $|x|\ge R$ and $|y|\le R_1$,
due to the assumption $R\ge 2 R_1>0$, we get
\begin{align*}
J_{2,R,R_1}
&=
\left ( \int_{\rr d}
\left( \sup_{|x|\geq R}
\iint _{\complement \Omega _{R_1}}
F_0(x-y, \xi-\eta)
\Phi (y, \eta) \, dy d\eta \right)^q
\, d\xi \right)^{\frac{r}{q}}
\\[1ex]
&=
\left ( \int_{\rr d}
\left( \sup_{|x|\geq R}
\iint _{\complement \Omega _{R_1}}
\chi _{{}_{R_1}}(x-y,\xi -\eta )F_0(x-y, \xi-\eta)
\Phi (y, \eta) \, dy d\eta \right)^q
\, d\xi \right)^{\frac{r}{q}}
\\[1ex]
&\le
\left ( \int_{\rr d}
\sup_{|x|\geq R}
\left ( (\chi _{{}_{R_1}}F_0)*\Phi (x,\xi ) \right )^q
\, d\xi \right )^{\frac{r}{q}}
\le
\left (
\nm {(\chi _{{}_{R_1}}F_0)*\Phi}{L^{\infty ,q}}
\right )^{\frac rq}.
\end{align*}
Hence \cite[Corollary 5.2]{ToPfTe} gives
\begin{align}
J_{2,R,R_1}
&\le
\left (
\nm {(\chi _{{}_{R_1}} F_0)*\Phi}{L^{\infty ,q}}
\right )^{\frac rq}
\le
\left (
\nm {(\chi _{{}_{R_1}} F_0)*\Phi )}
{\sfW ^\infty (\ell ^{\infty ,q})}
\right )^{\frac rq}
\notag
\\[1ex]
&\le
\left (
\nm {\chi _{{}_{R_1}}F_0}{\sfW ^\infty (\ell ^{\infty ,q})}
\nm {\Phi}
{\sfW ^1 (\ell ^{1,r})}
\right )^{\frac rq}
\notag
\\[1ex]
&\le
\left (
\nm {\chi _{{}_{R_1}}F_0}{\sfW ^\infty (\ell ^{\infty ,q})}
\nm {\Phi}
{\sfW ^\infty  (\ell ^{r})}
\right )^{\frac rq}.
\label{Eq:J2RR1Est}
\end{align}

\par

Since $f\in M^{\sharp ,q}_{(\omega )}(\rr d)$,
Theorem \ref{Thm:IndepWind} gives
$$
\lim _{R_1\to \infty}
\nm {\chi _{{}_{R_1}}F_0}{\sfW ^\infty (\ell ^{\infty ,q})}
=0.
$$
By combining the latter limit with \eqref{Eq:J2RR1Est},
one obtains 
$$
J_{2,R,R_1} \to 0,
\quad \text{as}\quad
R_1\to \infty \ \text{and}\ R>2R_1.
$$
This gives
\eqref{Eq:J12RR1ep}, and the assertion (1) follows.

\par

It remains to prove (2). Suppose that
$q_1\le q_2$ and
$f\in M^{\sharp ,q_1}_{(\omega )}(\rr d)$.
Then $V_\phi f\in
\sfW ^\infty (\omega ,\ell ^{\infty,q_1})$
Since $\sfW ^\infty (\omega ,\ell ^{p_1,q_1})$
is continuously embedded in 
$\sfW ^\infty (\omega ,\ell ^{p_2,q_2})$,
when $p_1,p_2,q_1,q_2\in (0,\infty ]$ satisfy
$p_1\le p_2$ and $q_1\le q_2$, 
\eqref{Eq:IndepWind} gives
$$
0\le \nm {\chi _{{}_R}V_\phi f}
{\sfW ^\infty (\omega ,\ell ^{p_2,q_2})}
\lesssim
\nm {\chi _{{}_R}V_\phi f}
{\sfW ^\infty (\omega ,\ell ^{p_1,q_1})}
\to 0,
$$
as $R\to \infty$. This shows that
$M^{\sharp ,q_1}_{(\omega )}(\rr d)$
is continuously embedded in
$M^{\sharp ,q_2}_{(\omega )}(\rr d)$,
and the result follows.
\end{proof}

\par

The next result, complementary to
the previous one, shows that
$L^\sharp _{(\omega _0 )}(\rr d)$
is contained in $M^{\sharp ,\infty}
_{(\omega )}(\rr d)$.

\par

\begin{prop}
\label{Prop:LebModEmb}
Let $\vartheta \in \mascP _E(\rr d)$ and
$\omega (x,\xi )= \omega _0 (x)$. Then
\begin{alignat}{2}
M^{\infty ,1}
_{(\omega )}(\rr d)
&\hookrightarrow
L^\infty _{(\omega _0 )}(\rr d) &
&\hookrightarrow
M^{\infty ,\infty}
_{(\omega )}(\rr d)
\label{Eq:LebModEmbOld}
\intertext{and}
M^{\sharp ,1}
_{(\omega )}(\rr d)
&\hookrightarrow
L^\sharp _{(\omega _0 )}(\rr d) &
&\hookrightarrow M^{\sharp ,\infty}
_{(\omega )}(\rr d).
\label{Eq:LebModEmb}
\end{alignat}
\end{prop}

\par

\begin{proof}
The inclusions in \eqref{Eq:LebModEmbOld}
follows from analysis given in already
\cite{Fei1}. In order to be self-contained
we here give the arguments.

\par

Let $v_0\in \mascP _E(\rr d)$ be such that
$\omega _0$ is $v_0$-moderate, and let
and $\phi \in \Sigma _1(\rr d)$ such that
$\phi (0)=(2\pi )^{-\frac d2}$.
For any $f\in L^\infty _{(\omega _0)}(\rr d)$
\begin{equation}
f(x)=\int _{\rr d}V_\phi f(x,\xi )
e^{i\scal x\xi}\, d\xi .
\end{equation}
This gives
\begin{align*}
\nm f{L^\infty _{(\omega _0)}}
&\le
\sup _{x\in \rr d}
\int _{\rr d}
|V_\phi f(x,\xi )\omega (x,\xi )|\, d\xi 
\\[1ex]
&\le
\int _{\rr d}
\sup _{x\in \rr d}
|V_\phi f(x,\xi )\omega (x,\xi )|\, d\xi
\asymp
\nm f{M^{\infty ,1}_{(\omega )}}.
\end{align*}
This gives the first inclusion in
\eqref{Eq:LebModEmbOld}.

\par

In the same way, suppose that
$f\in M^{\sharp ,1}_{(\omega )}(\rr d)$.
Then
\begin{align*}
0\le
\sup _{x\ge R}|f(x)|
&\le
\sup _{x\ge R}
\int _{\rr d}
|V_\phi f(x,\xi )\omega (x,\xi )|\, d\xi 
\\[1ex]
&\le
\int _{\rr d}
\sup _{|x|\ge R}
|V_\phi f(x,\xi )\omega (x,\xi )|\, d\xi 
\to 0,
\end{align*}
as $R\to \infty$. This gives the first
inclusion in \eqref{Eq:LebModEmb}.

\par

Next we prove the second inclusion
in \eqref{Eq:LebModEmbOld}. We have
\begin{align*}
\nm f{M^{\infty ,\infty}_{(\omega )}}
&\asymp 
\sup _{x,\xi \rr d}|(2\pi )^{\frac d2}
V_\phi f(x,\xi )\omega _0(x)|
\\[1ex]
&=
\sup _{x,\xi \rr d}
\left |
\int _{\rr d}f(y)
\overline{\phi (y-x)}
e^{-i\scal y\xi}\omega _0(x)\, dy
\right |
\\[1ex]
&\lesssim
\sup _{x,\xi \rr d}
\int _{\rr d}|f(y)\omega _0(y)|\, 
|\phi (y-x)v_0(y-x)|\, dy
\\[1ex]
&\le
\nm f{L^\infty _{(\omega _0)}}
\int _{\rr d}\, 
|\phi (y-x)v_0(y-x)|\, dy
=
\nm f{L^\infty _{(\omega _0)}}
\nm \phi {L^1_{(v_0)}},
\end{align*}
and the second inclusion
in \eqref{Eq:LebModEmbOld} follows.

\par

It remains to prove the second inclusion in
\eqref{Eq:LebModEmb}.
Let $\ep, c>0$ and 
$f\in L^\sharp _{(\omega _0)}(\rr d)$. 
Then,
\begin{align}
\label{eq:EstEsssup}
\underset{|y|\geq R} \esssup \, 
|f(y) \omega _0(y)|
\to 0
\quad \text{as}\quad
R\to \infty.
\end{align}
Since 
$
L^\sharp _{(\omega _0)}(\rr d)
\hookrightarrow 
L^\infty _{(\omega _0)}(\rr d)
\hookrightarrow 
M^{\infty} _{(\omega )}(\rr d)
$, the result follows if we show that
\begin{align*}
 J_{R_1}
 \equiv
 \sup_{\xi\in \rr d} 
 \left(
 \sup_{|x|\geq R_1} 
 F_0(x, \xi) 
 \right)
 \to 0
 \quad \text{as}\quad R_1 \to \infty ,
\end{align*}
where
$f \in  M^{\sharp ,\infty} _{(\omega )}(\rr d)$,
$\phi \in \Sigma_{1} (\rr d)\setm$,
$F_0  = |V_\phi f \cdot \omega | $,
and $R_1>0$. 
Since $v_0$ is submultiplicative we have
\begin{align*}
J_{R_1}
&\lesssim
 \sup_{|x|\geq R_1}
 \left (
 \int _{\rr d} |f(y)\omega _0(y) 
 \overline{\phi(y-x)}
v_0(y-x)|\, dy
 \right )
 \\[1ex]
 &\lesssim 
 J_{1,R,R_1}+J_{2,R,R_1},
\end{align*}
where 
\begin{align*}
   J_{1,R,R_1}
   &=
   \sup_{|x|\geq R_1}
   \left ( 
   \int_{|y|\leq R}|f(y)\omega _0(y)
   \overline{\phi(y-x)}v_0(y-x)| \, dy
   \right )
\intertext{and} 
   J_{2,R,R_1}
   &=
   \sup_{|x|\geq R_1}
   \left( 
   \int_{|y|\geq R}|f(y)\omega _0(y) 
   \overline{\phi(y-x)}v_0(y-x)|\, dy
   \right).
\end{align*}

\par

Suppose $R_1>2R$. 
First we estimate $J_{1,R,R_1}$. 
Then $|y|\leq R$ and $|x|\geq R_1$ 
implies $|x-y|\geq R_1-R>R$.
This gives 
\begin{equation*} 
    J_{1,R,R_1}
    \leq
    \nm f {L^\infty_{(\omega _0)}}  
    \nm \phi {L^1_{(v)}(\complement B_{R})}.
\end{equation*}
Since
$
\nm \phi {L^1_{(v)}(\complement B_{R})}
\to 0
$
as $R \to \infty$
we get
\begin{equation}
\label{eq:EsJ1R}
J_{1,R,R_1} \to 0
\quad \text{as}\quad
R\to \infty ,\ R_1>2R. 
\end{equation}
Next we estimate $J_{2,R,R_1}$. 
By \eqref{eq:EstEsssup} we get
\begin{align}
\label{eq:EsJ2R}
   J_{2,R,R_1}
   \leq 
   \underset{|y|\geq R} \esssup
   \big (|f(y)\omega _0(y)|\big )
   \nm \phi {L^1 _{(v_0)}}
   \to 0
 \quad \text{as}\quad
 R \to \infty .
\end{align}
Combining \eqref{eq:EsJ1R} and 
\eqref{eq:EsJ2R} provides
\begin{align*}
0 
\leq 
J_{R_1}
    \leq 
    J_{1,R,R_1}+J_{2,R,R_1}
    \to 0 ,
 \quad \text{as}\quad R \to \infty ,\ R_1>2R,
\end{align*}
which finishes the proof.
\end{proof}

\par


\par

The next result give, among others, indications
on how $M^{0,q}_{(\omega )}(\rr d)$ and
$M^{\flat ,q}_{(\omega )}(\rr d)$ behaves 
under multiplications with elements in
$\Sigma _1(\rr d)$. Here let
\begin{align}
M^{0,q}_{(\omega )}(\rr d)
&\equiv
\bigcap _{\rho >0} M^{p,q}_{(\omega _\rho )}(\rr d),
\quad
\omega _\rho (x,\xi )=\omega (x,\xi )e^{\rho |x|},
\label{Eq:ApprSet0}
\intertext{and}
M^{\flat ,q}_{(\omega )}(\rr d)
&\equiv
\bigcap _{p>0} M^{p,q}_{(\omega )}(\rr d),
\label{Eq:ApprSet}
\end{align}
where the topologies are defined by
project limit topologies of
$M^{p,q}_{(\omega _\rho )}(\rr d)$,
$\rho >0$, and $M^{p,q}_{(\omega )}(\rr d)$, $p>0$,
respectively.
We observe that $M^{0,q}_{(\omega )}(\rr d)$
in \eqref{Eq:ApprSet0} is independent of $p$,
because
$$
M^{p_2,q}_{(\omega _{\rho +\ep})}(\rr d)
\hookrightarrow
M^{p_1,q}_{(\omega _\rho )}(\rr d)
\hookrightarrow
M^{p_2,q}_{(\omega _\rho )}(\rr d),
\quad \text{when}\quad
p_1\le p_2,\ \ep >0.
$$
From these inclusions it also follows that
$$
M^{0,q}_{(\omega )}(\rr d)
\hookrightarrow
M^{\flat ,q}_{(\omega )}(\rr d).
$$
Furthermore, for some $\rho _0>0$ we have
$$
\omega _0(\xi )e^{-\rho _0|x|}
\lesssim
\omega (x,\xi )
\lesssim
\omega _0(\xi )e^{\rho _0|x|},
\quad
\omega _0(\xi ) = \omega (0,\xi ).
$$
This leads to that we may replace
$\omega$ with $\omega _0$ in
\eqref{Eq:ApprSet0}.

\par

\begin{prop}
\label{Proposition:inclusions}
Let $q\in (0,\infty ]$,
$\omega \in \mascP _E(\rr {2d})$
and let $M^{\flat ,q}_{(\omega )}(\rr d)$
be as in \eqref{Eq:ApprSet}. Then
\begin{align}
\Sigma _1(\rr d)
&\subseteq
\sets {f\cdot g}
{f\in M^{\infty ,q}_{(\omega )}(\rr d),
\ g\in \Sigma _1(\rr d)}
\subseteq
M^{0,q}_{(\omega )}(\rr d)
\notag
\\[1ex]
&\subseteq
M^{\flat ,q}_{(\omega )}(\rr d)
=
\bigcap _{p\in (0,\infty ]}
M^{p,q}_{(\omega )}(\rr d)
\hookrightarrow
\bigcup _{p\in (0,\infty ]}
M^{p,q}_{(\omega )}(\rr d)
\notag
\\[1ex]
&\hookrightarrow
M^{\sharp ,q}_{(\omega )}(\rr d)
\hookrightarrow
M^{\infty ,q}_{(\omega )}(\rr d)
\hookrightarrow
\Sigma _1'(\rr d).
\label{Eq:RelSmallerLargerExpModSp}
\end{align}
\end{prop}

\par

%

\begin{proof}
The first inclusion
\eqref{Eq:RelSmallerLargerExpModSp}
follows from the facts that
$\Sigma _1(\rr d)
\subseteq
M^{\infty ,q}_{(\omega )}(\rr d)$,
and that $\Sigma _1(\rr d)$ is a
factorization algebra (see e.{\,}g.
\cite[Theorem 2.2]{ToKhNiNo}
and its proof). The second inclusion
follows from
multiplication properties for modulation spaces
explained in \cite{Toft26}.
The second last inclusion in
\eqref{Eq:RelSmallerLargerExpModSp}
follows from Proposition
\ref{Prop:ModEmb}. The other
inclusions in
\eqref{Eq:RelSmallerLargerExpModSp}
follows from the definitions.
This gives the result.
\end{proof}

\par

Next we will apply Proposition \ref{Prop:ModEmb}
to deduce properties for compactly supported
elements in modulation spaces.

\par

\begin{prop}
\label{Prop:CompSuppModSp}
Let $p\in (0,\infty ]$ and $s>1$,
$\omega _0\in \mascP _s(\rr {d})$ and
$\omega (x,\xi )=\omega _0(\xi )$,
$x,\xi \in \rr d$.
Then
\begin{alignat}{2}
M^{0,q}_{(\omega )}(\rr d)
\cap \maclE _s'(\rr d)
&=
M^{\flat ,q}_{(\omega )}(\rr d)
\cap \maclE _s'(\rr d) & &
\notag
\\[1ex]
&=
M^{p,q}_{(\omega )}(\rr d)
\cap \maclE _s'(\rr d)
\notag
\\[1ex]
&=
M^{\sharp,q}_{(\omega )}(\rr d)
\cap \maclE _s'(\rr d), &
\quad q&\in (0,\infty ],
\label{Eq:CompSuppModSp}
\\[1ex]
\mascF \sfW ^1(\ell ^q_{(\omega _0)}(\zz d) )
\cap \maclE _s'(\rr d)
&\subseteq 
M^{\flat ,q}_{(\omega )}(\rr d)
\cap \maclE _s'(\rr d) & &
\notag
\\[1ex]
&\subseteq
\mascF L^q_{(\omega _0)}(\rr d)
\cap \maclE _s'(\rr d),&
\quad q&\in (0,1 ]
\label{Eq:CompSuppFLqSpSmallq}
\intertext{and}
M^{\flat ,q}_{(\omega )}(\rr d)
\cap \maclE _s'(\rr d)
&=
\mascF L^q_{(\omega _0)}(\rr d)
\cap \maclE _s'(\rr d),& \quad q&\in [1,\infty ].
\label{Eq:CompSuppFLqSpBigq}
\end{alignat}

\par

If $\omega \in \mascP _{0,s}(\rr {2d})$,
then the same holds true with $\maclE '_{0,s}$
in place of $\maclE '_{s}$ at each occurrence.
If instead $\omega \in \mascP (\rr {2d})$,
then the same holds true with $\mascE '$
in place of $\maclE '_{s}$ at each occurrence.
\end{prop}

\par

The case when $\omega \in \mascP (\rr {2d})$
and $q\ge 1$ in Proposition \ref{Prop:CompSuppModSp}
follows from from e.{\,}g.
\cite[Corollary 6.2]{PTT} or
\cite[Theorem 2.1]{RSTT}.
The other cases follow by similar arguments. In
order to be self-contained, we here present a
proof.

\par

\begin{proof}
By Proposition \ref{Prop:ModEmb} it follows
that
\begin{alignat*}{2}
M^{\flat ,q}_{(\omega )}(\rr d)
\cap \maclE _s'(\rr d)
&\subseteq
M^{p,q}_{(\omega )}(\rr d)
\cap \maclE _s'(\rr d) &&
\\
&\subseteq
M^{\sharp,q}_{(\omega )}(\rr d)
\cap \maclE _s'(\rr d),& &
\\
&\subseteq
M^{\infty ,q}_{(\omega )}(\rr d)
\cap \maclE _s'(\rr d),&
\quad
0&<p<\infty ,
\intertext{and}
M^{p_1,q}_{(\omega )}(\rr d)
\cap \maclE _s'(\rr d)
&\subseteq
M^{p_2,q}_{(\omega )}(\rr d)
\cap \maclE _s'(\rr d),&
\quad
0&<p_1\le p_2\le \infty .
\intertext{Hence the result follows if we prove}
M^{p_2,q}_{(\omega )}(\rr d)
\cap \maclE _s'(\rr d)
&\subseteq
M^{p_1,q}_{(\omega )}(\rr d)
\cap \maclE _s'(\rr d),&
\quad
0&<p_1\le p_2\le \infty .
\end{alignat*}

\par

Therefore suppose that $p_1\le p_2$ and
$f\in M^{p_2,q}_{(\omega )}(\rr d)
\cap \maclE _s'(\rr d)$. Then there is
$\phi \in \maclD _s(\rr d)$
such that $\phi =1$ on $\supp f$,
giving that $f=f\cdot \phi$. Then
$\phi \in M^{p_0,q_0}_{(\omega _0)}$
for every $p_0,q_0\in (0,\infty ]$
and $\omega _0\in \mascP _s(\rr {2d})$.
By choosing $p_0,q_0 \in (0,\infty ]$
such that
$$
\frac 1{p_0}=\frac 1{p_1}+\frac 1{p_2}\ge 0
\quad \text{and}\quad
\frac 1{q_0} \ge 1,\frac 1q,\frac 1{p_1},
$$
it follows that (3.13) in \cite{Toft26} is fulfilled when
$p_0,p_1,p_2,q_0,q_1,q_2$
in \cite[(3.13)]{Toft26}
is substituted by $p_1,p_2,p_0,q,q,q_0$.
Furthermore, by choosing $\rho _0>0$
such that $\omega$ is $e^{\rho _0|\cdo |}$ moderate,
we obtain
$$
\omega _\rho (x,\xi _1+\xi _2)
\lesssim
\omega (x,\xi )\vartheta _\rho (x,\xi _2),
\quad
\vartheta _\rho (x,\xi ) = e^{\rho (|x|+|\xi |)}
$$
when $\rho >\rho _0$. By \cite[Theorem 3.2]{Toft26}
we get
$$
\nm f{M^{p,q}_{(\omega _\rho )}}
=
\nm {f\cdot \phi}{M^{p,q}_{(\omega _r)}}
\lesssim
\nm f{M^{p,q}_{(\omega )}}
\nm \phi{M^{\infty ,r}_{(\vartheta _\rho )}}
\asymp
\nm f{M^{p,q}_{(\omega )}}.
$$
This shows that $f\in M^{0,q}_{(\omega )}(\rr d)
\cap \maclE _s'(\rr d)$ when
$f\in M^{p,q}_{(\omega )}(\rr d)
\cap \maclE _s'(\rr d)$,
and \eqref{Eq:CompSuppModSp} follows.

\par

Next we prove \eqref{Eq:CompSuppFLqSpSmallq}
and \eqref{Eq:CompSuppFLqSpBigq}. Suppose
$f\in \maclE _s'(\rr d)$ and choose
$\phi$ as above. We have
\begin{align*}
\nm f{\mascF L^q_{(\omega _0)}}
&=
\nm {\widehat f\cdot \omega _0}{L^q}
=
\nm {\mascF (f\cdot \overline \phi )\cdot \omega _0}{L^q}
\\[1ex]
&=
\nm {V_\phi f(0,\cdo )\cdot \omega (0,\cdo )}{L^q}
\\[1ex]
&\le
\NM {\sup _{x\in \rr d}
\left | V_\phi f(x,\cdo )\cdot \omega (x,\cdo )\right |}{L^q}
=
\nm f{M^{\infty ,q}_{(\omega )}},
\end{align*}
which shows 
\begin{equation}
\label{Eq:CompSuppFLqSpAllqSecIncl}
M^{\infty ,q}_{(\omega )}(\rr d)
\cap \maclE _s'(\rr d)
\subseteq
\mascF L^q_{(\omega _0)}(\rr d)
\cap \maclE _s'(\rr d),
\quad q\in (0,\infty ]
\end{equation}
In particular, 
the second inclusion in
\eqref{Eq:CompSuppFLqSpSmallq} follows.

\par

Next we prove
\begin{equation}
\label{Eq:CompSuppFLqSpAllqFirstIncl}
\mascF \sfW ^1(\ell ^q_{(\omega _0)}(\zz d) )
\cap \maclE _s'(\rr d)
\subseteq 
M^{\infty ,q}_{(\omega )}(\rr d)
\cap \maclE _s'(\rr d),
\quad
q\in (0,\infty ],
\end{equation}
which, together with the previous estimates, will give 
\eqref{Eq:CompSuppFLqSpSmallq}.

\par

Let $r=\min (1,q)$. Since
\begin{align*}
|V_\phi f(x,\xi )\omega (x,\xi )|
&=
(2\pi )^{-\frac d2}|\widehat f
*
(\mascF \phi (\cdo -x))(\xi )\omega _0(\xi )|
\\[1ex]
&\lesssim
(|\widehat f \omega _0|
*
|\mascF (\phi (\cdo -x))v_0|)(\xi )
=
(|\widehat f \omega _0|
*
|\widehat \phi v_0|)(\xi ),
\end{align*}
we get
\begin{align*}
\nm f{M^{\infty ,q}_{(\omega )}}
&=
\NM {\sup _{x\in \rr d}| V_\phi f(x,\cdo )\cdot \omega _0|}
{L^q}
\\[1ex]
&\lesssim
\nm {|\widehat f \omega _0|
*
|\widehat \phi v_0|}{L^q}
\le
\nm {|\widehat f \omega _0|
*
|\widehat \phi v_0|}{\sfW ^\infty (\ell ^q)}
\\[1ex]
&\le
\nm {\widehat f \omega _0}{\sfW ^1 (\ell ^q)}
\nm {\widehat \phi v_0}{\sfW ^\infty  (\ell ^r)},
\end{align*}
which gives \eqref{Eq:CompSuppFLqSpAllqFirstIncl},
and \eqref{Eq:CompSuppFLqSpSmallq} follows.

\par

Next suppose that $q\ge 1$. Then 
\eqref{Eq:CompSuppFLqSpAllqFirstIncl} gives
\begin{align}
\mascF L^q_{(\omega _0)}(\rr d)\cap \maclE _s'(\rr d)
&\subseteq
\mascF \sfW ^1(\ell ^q_{(\omega _0)}(\zz d) )
\cap \maclE _s'(\rr d)
\notag
\\[1ex]
&\subseteq 
M^{\infty ,q}_{(\omega )}(\rr d)
\cap \maclE _s'(\rr d).
\label{Eq:CompSuppFLqSpAllqSecIncl2}
\end{align}
The relation \eqref{Eq:CompSuppFLqSpBigq}
now follows by combining
\eqref{Eq:CompSuppModSp},
\eqref{Eq:CompSuppFLqSpAllqSecIncl}
and
\eqref{Eq:CompSuppFLqSpAllqSecIncl2}.

\par

The other cases follow
by similar arguments and are left for the reader.
\end{proof}

\par

\par

\par

\subsection{Convolution properties for
$M^{\sharp ,q}_{(\omega )}$ spaces}

The following convolution result
is important in some investigations later on.

\par

{{\color{blue}{Mention in Section 4 or here
that we cannot use density arguments for proving
next result because if $q$ or $q_j$ equals
$\infty$, the density results in Section 4
are not applicable.}}}
\\
{{\color{red}{The text after the proposition should
explain this./JT}}}

\par

\begin{prop}
\label{Prop:ModConv}
Let $q_k\in (0,\infty ]$ and
$\omega _k\in \mascP _E(\rr {2d})$,
$k=0,1,2$, be such that
$$
\frac 1{q_1}+\frac 1{q_2}=\frac 1{q_0}
\quad \text{and}\quad
\omega _0(x_1+x_2,\xi )
\lesssim
\omega _1(x_1,\xi )\omega _2(x_2,\xi ),
\quad x_1,x_2,\xi \in \rr d.
$$
Then the map $(f_1,f_2)\mapsto f_1*f_2$
from $\Sigma _1(\rr d)\times \Sigma _1(\rr d)$
to $\Sigma _1(\rr d)$ extends uniquely
to a continuous map from
$M^{\sharp ,q_1}_{(\omega _1)}(\rr d)
\times M^{1,q_2}_{(\omega _2)}(\rr d)$
to $M^{\sharp ,q_0}_{(\omega _0)}(\rr d)$.
\end{prop}

\par

By \cite[Theorem 3.4]{ToPfTe} it follows
that the convolution map
in Proposition \ref{Prop:ModConv} is uniquely
extendable to a continuous map from
$M^{\infty ,q_1}_{(\omega _1)}(\rr d)
\times M^{1,q_2}_{(\omega _2)}(\rr d)$
to $M^{\infty ,q_0}_{(\omega _0)}(\rr d)$, which
shows that $f_1*f_2$ is uniquely defined as
an element in $M^{\infty ,q_0}_{(\omega _0)}(\rr d)$
when $f_1\in M^{\sharp ,q_1}_{(\omega _1)}(\rr d)$
and $f_2\in M^{1,q_2}_{(\omega _2)}(\rr d)$. We need
to guarantee that $f_1*f_2$ belongs to the
subspace $M^{\sharp ,q_0}_{(\omega _0)}(\rr d)$
of $M^{\infty ,q_0}_{(\omega _0)}(\rr d)$.

\par

Since we allow $q_j$ to attain $\infty$, we may
not use density arguments for the proof for such
choices of $q_j$, because
$\Sigma _1 (\rr d)$ is not dense in
$M^{p,\infty} _{(\omega _j)}(\rr d)$ or
$M^{\sharp ,\infty} _{(\omega _j)}(\rr d)$.
If $q_j<\infty$ for $j=1,2$, then
$\Sigma _1(\rr d)$ is dense in
$M^{p,q_j}_{(\omega _j)}(\rr d)$ but not in
$M^{\infty ,q_j}_{(\omega _j)}(\rr d)$. For
such $q_j$, we will show in Section \ref{sec4}
that $\Sigma _1 (\rr d)$
is dense in $M^{\sharp ,q_j}_{(\omega _j)}(\rr d)$.
On the other hand, at this stage, the latter property
is unproven, and can therefore not be used.

\par

\begin{proof}[Proof of Proposition \ref{Prop:ModConv}]
We only prove the result for 
$q_k < \infty$, $k=1,2$. 
The case when at least one of 
$q_k=\infty$, $k=1,2$, 
follows by similar arguments and is left for
the reader.

\par

Due to \cite{Toft26}, 
the map  $(f_1,f_2)\mapsto f_1*f_2$
from $\Sigma _1(\rr d)\times \Sigma _1(\rr d)$
to $\Sigma _1(\rr d)$ extends uniquely
to a continuous map from
$M^{\infty ,q_1}_{(\omega _1)}(\rr d)
\times M^{1,q_2}_{(\omega _2)}(\rr d)$
to $M^{\infty ,q_0}_{(\omega _0)}(\rr d)$
(see also \cite[Theorem 3.7]{TeoToft22}).
Hence it remains to show 
$$
f_1*f_2  \in  M^{\sharp ,q_0}_{(\omega _0)}(\rr d)
\quad \text{when} \quad
f_1 \in M^{\sharp ,q_1}_{(\omega _1)}(\rr d),\ 
f_2 \in M^{1,q_2}_{(\omega _2)}(\rr d).
$$

\par

Let $\phi \in \Sigma _1 (\rr d)$, 
$
f_1 \in M^{\sharp ,q_1}_{(\omega _1)}(\rr d)
$ 
and 
$
f_2 \in M^{1,q_2}_{(\omega _2)}(\rr d)
$.
 Additionally let
\begin{align}
F_j=|V_{\phi} f_j \cdot \omega_j|, \quad j=1,2.
\end{align}
For any $R, R_1>0$ we define
\begin{align}
J_{R_1}
&=
\left(
\int_{\rr d}
\sup_{|x|\geq R_1} 
\left(
\int_{\rr d} 
F_1(y,\xi )
F_2(x-y, \xi)
\, dy
\right )^{q_0}\, d\xi
\right )^{\frac{1}{q_0}},
\label{Eq:JRDef5}
\\[1ex]
J_{1,R,R_1}
&=
\left ( \int_{\rr d}
\left( \sup_{|x|\geq R_1} 
\int _{|y| \leq R} 
F_1(y,\xi )
F_2(x-y, \xi)
\, dy \right)^{q_0}
\, d\xi\right)^{\frac{1}{q_0}},
\label{Eq:J1RDef5}
\intertext{and}
J_{2,R,R_1}
&=
\left ( \int _{\rr d}
\left( \sup _{|x|\geq R_1}
\int _{|y| \geq R} 
F_1(y,\xi )
F_2(x-y, \xi)
\, dy \right)^{q_0}
\, d\xi \right)^{\frac{1}{q_0}}.
\label{Eq:J2RDef5}
\end{align}
%
%
We need to show 
\begin{align}\label{Eq:EsConvq0}
I_{R_1}=\left(\int _{\rr d}
\sup _{|x| \geq R_1}
\left | 
 V_{\phi \ast \phi} (f_1 \ast f_2) (x,\xi)
 \, \omega _0 (x,\xi)
\right|^{q_0}
\, d \xi
\right)^{\frac{1}{q_0}}
\to 0
\quad \text{as}\quad
R \to \infty,
\end{align}
which gives 
$f_1 \ast f_2  \in 
M^{\sharp, q_0}_{(\omega _0)} (\rr d)$,
and thereby completing the proof. By using 
$$
\omega _0(x_1+x_2,\xi )
\lesssim
\omega _1(x_1,\xi )\omega _2(x_2,\xi ),
\quad x_1,x_2,\xi \in \rr d,
$$
and
\begin{align*}
 \left|
 V_{\phi \ast \phi} (f_1 \ast f_2) (x,\xi)
 \right|
 \leq 
 \int_{\rr d} 
 \left | V_{\phi} f_1 (y, \xi) \right |
 \left | V_{\phi} f_2 (x-y, \xi)
 \right |\,
 dy, 
 \qquad x,\xi \in \rr d,
\end{align*}
we obtain $I_{R_1} \leq J_{R_1}$. Hence 
\eqref{Eq:EsConvq0} is true when 
\begin{align}
0 \leq J_{R_1} \lesssim J_{1,R,R_1}+J_{2,R,R_2}
\to 0
\quad \text{as}\quad
R \to \infty.
\label{Eq:J1Est}
\end{align}

\par

First, we estimate $J_{1,R,R_1}$ as
\begin{align*}
J_{1,R,R_1} 
&\leq
\left ( \int_{\rr d}
\left( 
\NM{F_1(\cdo,\xi)}{L ^\infty}
\sup_{|x|\geq R_1}  \int _{|y| \leq R} 
F_2(x-y, \xi)
\, dy \right)^{q_0}
\, d\xi\right)^{\frac{1}{q_0}}
\\[1ex]
&\le
\left ( \int_{\rr d}
\left(
\NM{F_1(\cdo,\xi)}{L ^\infty}
\sup_{|x|\geq R_1} 
\int _{|x-y| \leq R} 
F_2(y, \xi)
\, dy \right)^{q_0}
\, d\xi\right)^{\frac{1}{q_0}}.
\end{align*}

\par

Suppose that $R$ is large and $R_1>2 R$. 
Then $|x|\geq R_1$ and $|x-y| \le R$ implies 
$|y|\ge R_1-R>R$. 
By H{\"o}lder's inequality and the fact that
$f_2 \in M^{1,q_2}_{(\omega _2)} (\rr d)$, we get 
\begin{align}
J_{1,R,R_1}
&\le
\left ( \int _{\rr d}
\left(
\NM{F_1(\cdo,\xi)}{L ^\infty}
\int _{|y| > R} 
F_2(y, \xi)
\, dy \right)^{q_0}
\, d\xi\right)^{\frac{1}{q_0}}
\notag
\\[1ex]
&\le
\NM{F_1}{ L ^{\infty, q_1} }
\left ( \int_{\rr d}
\left(
\int _{|y| > R} 
F_2(y, \xi)
\, dy \right)^{q_2}
\, d\xi\right)^{\frac{1}{q_2}}
\to 0 
\label{Eq:J1RR1EstB}
\end{align}
as $R \to \infty$.

\par 

Next we estimate $J_{2,R,R_1}$.
Since
$f_1 \in  
M^{\sharp ,q_1}_{(\omega _1)}(\rr d)$
one has
\begin{align*}
\left ( \int_{\rr d}
\sup_{|y|\geq R}
\left(
F_1(y,\xi)
\right)^{q_1}
\, d\xi \right)^{\frac{1}{q_1}}
\to 0 
\quad \text{as}\quad 
R \to \infty.
\end{align*}
This gives
\begin{align}
J_{2,R,R_1} 
&\leq 
\left (
\int_{\rr d}
 \sup_{|x|\geq R_1}
 \left(
 \sup_{|y|\geq R}
 \left( F_1(y,\xi) \right)
 \int_{|y|\geq R}
 F_2(x-y,\xi)
 \, dy
 \right)^{q_0}
 \, d \xi
\right )^{\frac{1}{q_0}}
\notag
\\[1ex]
&\leq 
\left(
\int_{\rr d}
 \left(
 \sup_{|y|\geq R}
 \left( 
 F_1(y,\xi)
 \right)
\NM{F_2(\cdo,\xi)}{L^1}
 \right)^{q_0}
 \, d \xi
\right)^{\frac{1}{q_0}}
\notag
\\[1ex]
&\leq 
\left (
\int_{\rr d}
 \sup_{|y|\geq R}
 \left( 
 F_1(y,\xi)
 \right)^{q_1} 
 \, d \xi
\right)^{\frac{1}{q_1}}
\NM{F_2}{ L^{1, q_2} }
\to 0
\label{eq:EsF1}
\end{align}
as $R \to \infty $.
By \eqref{Eq:J1Est}, \eqref{Eq:J1RR1EstB} and 
\eqref{eq:EsF1}, we get
\eqref{Eq:EsConvq0}, and the result follows.
\end{proof}

\par

\section{Density properties for
$M^{\sharp ,q}_{(\omega )}$}
\label{sec4}

\par

In this section we prove the main result 
which will be used in the second part of the paper 
when considering Toeplitz operators on
$M^{\sharp ,q}_{(\omega )}(\rr d)$.
It can be formulated as follows.

\par

\begin{thm}
\label{Thm:ModDensity}
Let $q\in (0,\infty )$ and
$\omega \in \mascP _E(\rr {2d})$.
Then the following is true:
\begin{enumerate}
\item $M^{\flat ,q}_{(\omega )}(\rr d)$
and $\Sigma _1(\rr d)$
are dense in
$M^{\sharp ,q}_{(\omega )}(\rr d)$;

\vrum

\item if in addition
$\omega \in \mascP (\rr {2d})$, then
$M^{\flat ,q}_{(\omega )}(\rr d)
\cap \mascE '(\rr d)$ and
$C_0^\infty (\rr d)$ are dense in
$M^{\sharp ,q}_{(\omega )}(\rr d)$.
\end{enumerate}
\end{thm}

\par

\begin{rem}
As a consequence of Theorem
\ref{Thm:ModDensity}, one has that
$M^{\sharp ,q}_{(\omega )}(\rr d)$
agrees with
$M^{0 ,q}_{\omega}(\rr d)$ in \cite{FeGaPr},
when $1\le q<\infty$.
\end{rem}

\par

In order to prove Theorem \ref{Thm:ModDensity} we 
use approximation properties of convolutions and products in $M^{\sharp ,q}_{(\omega )}(\rr d)$
which are given in following propositions.

\par

\begin{prop}
\label{Prop:ModDensityConv}
Let $q\in (0,\infty )$ and
$\omega \in \mascP _E(\rr {2d})$,
$\phi \in
\Sigma _1 (\rr d)$ be such that 
\begin{align}\label{Eq:DefPhiConv}
\int _{\rr d} \phi (x) \, dx = 1,
\end{align}
$f\in M^{\sharp ,q}_{(\omega )}(\rr d)$, and
$\phi _{\ep } = \ep  ^{-d}
\phi (\ep ^{-1}\cdo )$, $\ep >0$.
Then
\begin{equation}
\label{Eq:DensityModSharpCaseConv}
\NM {f \ast \phi _{\ep _2} -f}
{M^{\infty ,q}_{(\omega )}}
\to 0 
\quad \text{as} \quad
\ep _2\to 0^+.
\end{equation}
\end{prop}

\par

\begin{prop}
\label{Prop:ModDensityProd}
Let $q\in (0,\infty )$ and
$\omega \in \mascP _E(\rr {2d})$,
$\phi \in \Sigma _1 (\rr d)$ be such that 
\begin{align}
\label{Eq:DefPhiProd}
\phi (0) = 1 ,
\end{align}
and let $f\in M^{\sharp ,q}_{(\omega )}(\rr d)$.
Then
\begin{equation}
\label{Eq:DensityModSharpCaseProd}
\NM {f \cdo \phi (\ep _1\cdo) -f}
{M^{\infty ,q}_{(\omega )}}
\to 0 
\quad \text{as} \quad
\ep _1\to 0^+.
\end{equation}
\end{prop}

The proofs are rather long, and  we 
present them in the forthcoming Subsection \ref{subsec3.1}.

\par

\par

\begin{proof}[Proof of Theorem
\ref{Thm:ModDensity}]
For any $\omega ,v\in \mascP _E(\rr {2d})$,
there are smooth equivalent weights in
$\mascP _E(\rr {2d})$. Hence we may assume that
$\omega$ and $v$ are smooth, and that $v(0)\ge 1$.

\par

We shall approximate distributions
with test functions in similar ways
as in \cite[Chapter IV]{Ho1}. 
Let $\phi _1,\phi _2\in
\Sigma _1 (\rr d)$ be such that 
\begin{align}
\tag*{(\ref{Eq:DefPhiProd})$'$}
\label{Eq:DefPhiProdA}
\phi _1(0) = 1 
\quad \text{ and } \quad
\int _{\rr d} \phi _2(x) \, dx = 1.
\end{align}
For any $f\in \Sigma _1'(\rr d)$
and $\ep _1,\ep _2 \ge 0$, let
\begin{equation}
\begin{aligned}
f _{\ep _1,\ep _2} 
&=
\begin{cases}
\left(
f \cdo \phi _1(\ep _1\cdo)
\right )
\ast \phi _{2,\ep _2}, & \ep _1,\ep _2>0,
\\[1ex]
f \ast \phi _{2,\ep _2},
& \ep _1=0,\ \ep _2>0,
\\[1ex]
f \cdo \phi _1(\ep _1\cdo),
& \ep _1>0,\ \ep _2=0,
\\[1ex]
f, & \ep _1,\ep _2=0,
\end{cases}
\\[1ex]
\phi _{2,\ep }(x) &= \ep  ^{-d}
\phi _2 (\ep ^{-1}x),
\qquad x\in \rr d.
\end{aligned}
\end{equation}
We observe that $f_{\ep _1,\ep _2}\in 
\Sigma _1(\rr d)$ when $\ep _1,\ep _2>0$,
and that
$$
\lim _{\ep _1,\ep _2\to 0^+}
f_{\ep _1,\ep _2}
=
\lim _{\ep _1\to 0^+}
f_{\ep _1,0}=
\lim _{\ep _2\to 0^+}
f_{0,\ep _2} = f,
$$
with convergence in $\Sigma _1'(\rr d)$,
see e.{\,}g. \cite{Toft22}.


Suppose that in addition $f \in M^{\sharp ,q}_{(\omega )}(\rr d)$.
By Propositions \ref{Prop:ModDensityConv} and
\ref{Prop:ModDensityProd}, and Cantor's diagonalization
principle, we get
\begin{alignat}{3}
\NM {f _{\ep _1,\ep _2} -f}
{M^{\infty ,q}_{(\omega )}}
&\to 0 & 
\quad &\text{as} &\quad
\ep _1,\ep _2&\to 0^+,
\label{Eq:DensityModSharp}
\end{alignat}
and the asserted density in (1) follows.
%
%
%
%
%
%
%
%
%
%

\medspace

Next suppose that $\omega \in \mascP (\rr {2d})$. Then
$$
C_0^\infty (\rr d) \hookrightarrow
M^{\flat ,q}_{(\omega )}(\rr d)
\cap \mascE '(\rr d).
$$
Hence (2) will follow if we prove that $C_0^\infty (\rr d)$
is dense in $M^{\sharp ,q}_{(\omega )}(\rr d)$. Since
\begin{equation}
\label{Eq:SchwSpModIncl}
\mascS (\rr d) \hookrightarrow M^{\sharp ,q}_{(\omega )}(\rr d),
\end{equation}
and $C_0^\infty (\rr d)$ is dense in $\mascS (\rr d)$, the assertion
follows if we show that the embedding \eqref{Eq:SchwSpModIncl}
is dense. This is however a consequence of (1), \eqref{Eq:ModerateClassesComp},
and the fact that $\Sigma _1(\rr d)$ is dense in $\mascS (\rr d)$.
\end{proof}

\par

A combination of Proposition
\ref{Prop:ModComplete} and
Theorem \ref{Thm:ModDensity} now gives the
following. The details are left for the reader.

\par

\begin{cor}
\label{Cor:MsharpCompl}
Let $q\in (0,\infty )$ and
$\omega \in \mascP _E(\rr {2d})$. Then
$M^{\sharp ,q}_{(\omega )}(\rr d)$
is the completion of $\Sigma _1(\rr d)$
under the norm
$\nm \cdo{M^{\infty ,q}_{(\omega )}}$.
\end{cor}

\par

\subsection{Proofs of Propositions \ref{Prop:ModDensityConv} and \ref{Prop:ModDensityProd}}
\label{subsec3.1}

\par

We need some preparations before the proofs.
For any $y\in \rr d$ and $r\ge 0$,
we let $Q(y)$ and $Q_r(y)$ be the cubes
\begin{align*}
Q(y)
&=
y+[0,1]^d
=
\sets {x\in \rr d}
{y_j\le x_j\le y_j+1,\ j=1,\dots ,d}
\intertext{and}
Q_r(y)
&=
y+[-r,r]^d
=
\sets {x\in \rr d}
{y_j-r\le x_j\le y_j+r,\ j=1,\dots ,d}
\end{align*}
We also let
$\lfloor t \rfloor$ be the integer
part of $t\in \mathbf R$, i.{\,}e.
$\lfloor t \rfloor$ is the largest
integer $n$ which fulfills $n\le t$,
and set
$$
\lfloor x \rfloor
\equiv
(\lfloor x_1 \rfloor ,\dots ,
\lfloor x_d \rfloor) \in \zz d
\quad \text{when}\quad
x=(x_1,\dots ,x_d)\in \rr d.
$$

\par

\begin{lemma}
\label{Lemma:SupEstimates}
Let $q\in (0,1]$, $\ep \in [0,1]$,
$F\in \maclM (\rr {2d})$,
$\psi \in \maclM (\rr d)$, $Q_d=[0,1]^d$,
$Q_{d,r}=[-r,r]^d$, and let
$\Omega _1,\Omega _2\subseteq \zz d$. Then
\begin{multline}
\sum _{\iota \in \Omega _1}
\sup _{j\in \zz d}
\left (
\sum _{\kappa \in \Omega _2}
\underset
{(x,\xi ,\eta )\in Q(j,\iota ,\kappa)}
\essup \, 
\, \big  (
|F(x,\xi -\ep \eta )\psi (\eta )|
\big  )
\right )^q
\\[1ex]
\le
\sum _{\kappa \in \Omega _2}
\left (
\sum _{\iota \in \Omega _1}
\nm F
{L^\infty (\rr d \times Q_2
(\iota -\lfloor \ep \kappa \rfloor ))}^q
\right )
\nm \psi {L^\infty (Q(\kappa ))}^q.
\end{multline}
\end{lemma}

\par

\begin{proof}
Since $q\le 1$, it follows that
for any $\iota \in \Omega _1$, one
obtains
\begin{align*}
\sup _{j\in \zz d}
&
\left (
\sum _{\kappa \in \Omega _2}
\underset
{(x,\xi ,\eta )\in Q(j,\iota ,\kappa)}
\essup \, 
\big  (
|F(x,\xi -\ep \eta )\psi (\eta )|
\big  )
\right )^q
\\[1ex]
&\le
\sum _{\kappa \in \Omega _2}
\sup _{j\in \zz d}\,
\underset {x \in Q(j)}
\essup \, 
\underset
{(\xi ,\eta ) \in Q(\iota ,\kappa )}
\essup \, 
\, 
\big  (
|F(x,\xi -\ep \eta )\psi (\eta )|
\big  )^q
\\[1ex]
&\le
\sum _{\kappa \in \Omega _2}
\underset {\xi \in Q_2(\iota -\lfloor \ep \kappa \rfloor )}
\essup \, 
\, \big  (
\nm {F(\cdo ,\xi )}
{L^\infty (\rr d)}\nm \psi {L^\infty (Q(\kappa ))}|
\big  )^q.
\end{align*}
In the last step we have used the fact
that
$$
\xi -\ep \eta \in
Q_2(\iota -\lfloor \ep \kappa \rfloor )
\quad \text{when}\quad
\xi \in Q(\iota ),\ \eta \in Q(\kappa ).
$$
The result now follows by taking the
sum over all $\iota \in \Omega _1$.
\end{proof}

\par

We also have the following.

\par

\begin{lemma}
\label{Lemma:TailEstimates}
Suppose $q\in (0,\infty )$, 
$K\subseteq \rr d$ is a bounded convex set,
$F\in \sfW ^\infty (\ell ^{p,q}(\zz {2d}))$
and $\psi \in \sfW ^\infty (\ell ^{q}(\zz {d}))$. Then
$$
\lim _{R\to \infty}
\left (
\sum _{|\kappa |\ge R}
\nm F{L^\infty (\rr d\times (\kappa +K))}^q
\right )
=
\lim _{R\to \infty}
\left (
\sum _{|\kappa |\ge R}
\nm \psi {L^\infty (\kappa +K)}^q
\right )
=0.
$$
\end{lemma}

\par

\begin{proof}
The result follows from the facts that $q<\infty$, and that
$$
\nm F{\sfW ^\infty (\ell ^{\infty ,q})}^q
\asymp
\sum _{\kappa \in \zz d}
\nm F{L^\infty (\rr d\times (\kappa +K))}^q
\quad \text{and}\quad
\nm \psi {\sfW ^\infty (\ell ^{q})}^q
\asymp
\sum _{\kappa \in \zz d}
\nm \psi {L^\infty (\kappa +K)}^q.
$$
The details are left for the reader.
\end{proof}

\par

Next we proceed with the proofs of Propositions
\ref{Prop:ModDensityConv}  and \ref{Prop:ModDensityProd}.

\par

\begin{proof}
[Proof of Proposition \ref{Prop:ModDensityConv}]
Let
\begin{equation*}
f _{0,\ep _2} 
=
\begin{cases}
f \ast \phi _{\ep _2},
& \ep _2\in (0,1],
\\[1ex]
f, & \ep _2=0,
\end{cases}
\end{equation*}
We use
$T_{\phi _0}f$ in 
\eqref{def:TTransform}
instead of $V_{\phi _0}f$
in the definition of
$M^{\infty,q}_{(\omega)}(\rr d)$.
Here we choose the standard
Gaussian
$ 
\phi _0 (x)=\pi ^{-\frac d4}
e^{-\frac 12|x|^2}$,
$x\in \rr d
$,
as window function, which is possible
in view of Theorem \ref{Thm:IndepWind}.
An advantage by using
$T_{\phi _0}$ is that it commutes with
convolutions in the sense
\begin{align*}
    T _{\phi_0} (f _{0,\ep _2})(x,\xi) 
    &=
    T _{\phi_0} \left(  f \ast \phi _{\ep _2}
    \right )(x,\xi)
    =
    \big ((T _{\phi_0}f)(\cdo ,\xi )*
    \phi _{\ep _2} \big )(x)
    \nonumber 
    \\[1ex]
    &=\ep _2^{-d} \int _{\rr d} (T _{\phi_0} f)(x- y,\xi) 
    \phi (\ep _2^{-1}y) \, dy, 
    \quad x,\xi \in \rr d.
\end{align*}
By taking $\ep _2^{-1}y$ as a
new variable of integration
in combination with the assumption 
\eqref{Eq:DefPhiProd},
we have
\begin{align}
   T _{\phi_0} (f _{0,\ep _2} -f)(x,\xi) 
   &=  \int _{\rr d} 
   \big ( (T _{\phi_0} f)(x-\ep _2y,\xi) 
   - (T _{\phi_0} f) (x,\xi)
   \big )
    \phi (y) \,dy
    \notag
    \\[1ex]
    &= J_{1, \ep _2, R}(x,\xi) 
        + J_{2, \ep _2, R} (x,\xi),
    \label{eq:CalcTPhi} 
    \intertext{when $R>0$, where}
     J_{1, \ep_2, R}(x,\xi)
     &= \int _{|y| \leq R} 
   \big ( (T _{\phi_0} f)(x-\ep _2y,\xi) 
   - (T _{\phi_0} f) (x,\xi)
   \big )
    \phi (y) \, dy
    \notag
     \intertext{and}
      J_{2, \ep_2, R}(x,\xi)
     &= \int _{|y| \geq R} 
   \big (
   (T _{\phi_0} f)(x-\ep _2y,\xi) 
   - (T _{\phi_0} f) (x,\xi)
   \big )
    \phi (y) \, dy
        \notag 
\end{align}
for all $x,\xi \in \rr d$.

\par

Let $r=\min (1,q)$ and
$$
\psi (y) = \sup _{\ep _2 \in [0,1]}|\phi (y)(1+v(\ep _2y,0)) |.
$$
Then $\psi \in L^p(\rr d)$ for every $p\in (0,\infty ]$, because $\phi \in \Sigma _1(\rr d)$.
Moreover, $\nm {f _{0,\ep _2} -f}
{M^{\infty ,q}_{(\omega )}}^r$  obeys the estimate
\begin{align*}
    \nm {f _{0,\ep _2} -f}
    {M^{\infty ,q}_{(\omega )}}^r
    \leq 
    M_{1,\ep _2} + M_{2,\ep _2}
    + M_{3,\ep _2} + M_{4,\ep _2},
\end{align*}
where
\begin{equation}
\label{Eq:MepExpressions}
\begin{alignedat}{2} 
    M_{1,\ep _2} 
    &=
    \NM {\chi_{B_{R_1} \times B_{R_2}}
    J_{1, \ep _2, R}} 
    {L^{\infty ,q}_{(\omega )}}^r,
    &\quad 
    M_{2,\ep _2} &=
    \NM {\chi_{B_{R_1} \times \complement B_{R_2}} 
        J_{1, \ep _2, R}} 
        {L^{\infty ,q}_{(\omega )}}^r,
        \quad
    \\[1ex]
    M_{3,\ep _2} 
    &=
    \NM {\chi_{\complement B_{R_1} \times \rr d}\,  J_{1, \ep _2, R}} 
    {L^{\infty ,q}_{(\omega )}}^r, 
    &\quad
    M_{4,\ep _2} 
    &=
     \NM {J_{2, \ep _2, R}}
     {L^{\infty ,q}_{(\omega )}}^r,
\\[1ex]
R_2&>R_1>2R>0. &&
\end{alignedat}
\end{equation}
The limit \eqref{Eq:DensityModSharpCaseConv}
follows if, for every $\ep >0$, we prove that for some
choices of $R$, $R_1$ and $R_2$, one has
\begin{align}\label{Eq:Case1}
M_{n,\ep _2} < \ep ,
\qquad
n = 1,2,3,4,
\end{align}
for every $\ep _2>0$ which is small enough.

\par

Therefore, let $\ep >0$ be arbitrary.
First, we prove 
\eqref{Eq:Case1} for $n=4$. 
Let 
\begin{equation}
\label{Eq:F0GRDef}
\begin{aligned}
F_0(x,\xi )&= |T _{\phi_0} f(x,\xi )\cdot
\omega (x,\xi )|,
\quad G_R(\xi )
=
\sup _{x\in \complement B_R} \big (F_0(x,\xi )\big )
\\[1ex]
\text{and}\quad
G(\xi )&=G_0(\xi ),
\qquad \xi \in \rr d,\ R\ge 0.
\end{aligned}
\end{equation}
and get due to \eqref{eq:CalcTPhi} and
the $v$-moderateness of $\omega$
\begin{align}
|J_{2, \ep _2, R}(x,\xi) &\cdot \omega (x,\xi)|
\notag
\\[1ex]
&\leq 
\int _{|y| \geq R} 
    \big (F _0(x-\ep _2y, \xi) v(\ep _2y,0) 
    +  F_0(x,\xi)
    \big ) |\phi (y)|\, dy 
\notag
\\[1ex]
&\leq 
\int _{|y| \geq R} 
    \big (F _0(x-\ep _2y, \xi)
    +  F_0(x,\xi)
    \big ) \psi (y)\, dy ,
\label{Eq:EstimateJ2EpR}
\end{align}
for all $x,\xi \in \rr d$.
In view of \eqref{Eq:F0GRDef},
the previous estimate gives
\begin{align*}
|J_{2, \ep _2, R}(x,\xi) \cdot
\omega (x,\xi)|
\leq
G(\xi) 
\nm { \psi} {L^1(\complement B_R)} 
\end{align*}
Since $\psi \in L^1 (\rr d)$, we get
\begin{align*}
    \nm { \psi }{L^1(\complement B_R)}
    \to 0
    \quad \text{as} \quad R \to \infty ,
\end{align*}
which in turn gives
\begin{align*}
    M_{4, \ep_2} 
    &=\NM {J_{2, \ep _2, R}} {L^{\infty ,q}_{(\omega )}}^r
    \leq
    \NM {G}{L^q}^r \nm { \psi }{L^1(\complement B_R)}^r
    =   \NM {f}{M^{\infty, q}_{(\omega)} }^r
    \nm { \psi }{L^1(\complement B_R)}^r
    \leq \ep ,
\end{align*}
provided $R$ is chosen large enough.
This gives \eqref{Eq:Case1} for $n=4$.

\par

Next, we prove \eqref{Eq:Case1} for $n=3$.
By \eqref{Eq:MepExpressions} and \eqref{Eq:F0GRDef}
we have
\begin{align}
|J_{1, \ep _2, R}(x,\xi) &\cdot \omega (x,\xi)|
\le
\int _{|y| \leq R}  \big ( F_0(x-\ep _2 y, \xi)+F_0(x, \xi) \big )
\psi (y) \, dy .
\label{eq:EstimateJ1EpRA}
\end{align}
Since $0<\ep _2<1$ we get
$$
|x-\ep _2y|\ge |x|-|y|\ge R_1-R\ge R
\quad \text{when}\quad |x|\ge R_1,\ |y|\le R.
$$
By \eqref{eq:EstimateJ1EpRA}
we obtain
\begin{align*}
|J_{1, \ep _2, R}(x,\xi) \cdot \omega (x,\xi)|
&\le
\big (G_{R_1-R}(\xi ) + G_{R_1}(\xi ) \big )\nm {\psi}{L^1}
\\[1ex]
&\le
2G_{R}(\xi ) \big )\nm {\psi}{L^1}
\asymp
G_{R}(\xi )
\end{align*}
In the last inequality we have used
the last line in \eqref{Eq:MepExpressions}.
Hence, for some constant $C>0$ we have
\begin{align*}
M_{3,\ep _2}
&\le
C\nm {G_R}{L^q}^r = C \NM {\sup _{|x|\ge R}
\big ( F_0(x,\cdo ) \big )}{L^q}^r<\ep 
\end{align*}
provided $R$ is chosen large enough. This shows
\eqref{Eq:Case1} for $n=3$.

\par

Next we prove \eqref{Eq:Case1} for $n=2$.
Since $R_2>R$ and that $\omega$ is $v$
moderate, we get
\begin{align*}
M_{2,\ep _2}
&\le
\left (
\int _{|\xi |\ge R}\NM
{\int _{\rr d} (F_0(\cdo -\ep _2y,\xi )
+F_0(\cdo ,\xi ))
\psi (y)\, dy}{L^\infty}
^q d\xi 
\right ) ^{\frac rq}
\\[1ex]
&\le
2 \left (
\int _{|\xi |\ge R}G_0(\xi )^q\, d\xi
\right )^{\frac rq}
\nm \psi {L^1}^r,
\end{align*}
which tends to $0$ as $R$ turns to infinity.
This gives \eqref{Eq:Case1} for $n=2$.

\par

It remains to prove  
\eqref{Eq:Case1} for $n=1$. 
Without loss of generality, we may assume that 
$\omega$ is continuous, cf. Subsection 
\ref{subsec1.1}. 
Let
\begin{align*}
H &: 
\rr d \times \rr d \times \rr d \times [0,1]
\to \mathbf R
\intertext{be defined by}
    H(x,y,\xi, \ep_2)
    &=
    |T _{\phi_0} f (x-\ep _2y, \xi) - T _{\phi_0} f(x,\xi)| 
    |\phi (y) \omega(x,\xi)|,
    \\[1ex]
    x,y,\xi
    &\in \rr d,\ \ep _2\in [0,1],
\end{align*}
By
\eqref{Eq:STFTTempDistSchwartz},
\eqref{def:TTransform}
and the fact that 
$
\phi \in \Sigma _1 (\rr d) 
$
is continuous,
it follows that
is continuous.
Hence $H$, restricted to the 
compact set
$$
K=
B_{R_2} \times B_{R} \times B_{R_1} \times [0,1],
$$
is uniformly continuous.
Consequently we have for $\ep _0>0$  
\begin{align*}
| H(x,y,\xi, \ep _2)| 
= | H(x,y,\xi, \ep_2) -  H(x,y,\xi, 0)|
< \ep _0,
\quad
(x,y,\xi ,\ep _2)\in K,
\end{align*}
provided $\ep _2$ is small enough. From this 
estimate it follows
\begin{align*}
M_{1,\ep _2} 
= 
\NM {\chi_{B_{R_1} \times B_{R_2}}
J_{1, \ep _2, R}} 
{L^{\infty ,q}_{(\omega )}}^r
\leq 
C \ep _0^r 
< \ep 
\end{align*}
provided $\ep _2$ and therefore also $\ep _0$
is small enough. 
Hence \eqref{Eq:Case1},
and thereby \eqref{Eq:DensityModSharpCaseConv},
hold. 
\end{proof}

\par

\begin{proof}
[Proof of Proposition \ref{Prop:ModDensityProd}]
Let
\begin{equation*}
f _{\ep _1,0} 
=
\begin{cases}
f \cdot \phi (\ep _1\cdo ),
& \ep _1\in (0,1],
\\[1ex]
f, & \ep _1=0,
\end{cases}
\end{equation*}
and
\begin{multline*}
H(x,\xi ,\eta, \ep _1)
=
|(V_{\phi _0}f(x,\xi-\ep _1\eta)
-
V_{\phi _0}f(x,\xi))\omega (x,\xi )|
|\widehat \phi (\eta)| ,
\\[1ex]
x,\xi ,\eta \in \rr d,\ \ep _1\in [0,1].
\end{multline*}
Also let
\begin{align*}
\psi (\eta )
&=
\sup _{\ep _1\in [0,1]}|\widehat \phi (\eta)(1+v(0,\ep _1 \eta ))|,
\\[1ex]
H_1(j,\iota ,\kappa , \ep _1)
&=
\underset {(x,\xi ,\eta )\in Q(j,\iota ,\kappa )}
\esssup 
H(x,\xi ,\eta , \ep _1)
\intertext{and}
H_2(j,\iota ,\kappa , \ep _1)
&=
\underset {(x,\xi ,\eta )\in Q(j,\iota ,\kappa )}
\esssup 
|(V_{\phi _0}f(x,\xi-\ep _1\eta)
\omega (x,\xi -\ep _1\eta )
\psi (\eta )|.
\end{align*}
Since $\omega$ is $v$-moderate, we have
\begin{equation}
\label{Eq:H1H2Compare}
H_1(j,\iota ,\kappa , \ep _1)
\le
H_2(j,\iota ,\kappa , \ep _1)
+ 
H_2(j,\iota ,\kappa , 0).
\end{equation}
We notice that $H$ is continuous, leading to continuity of
$$
[0,1]\ni \ep _1 \mapsto H_1(j,\iota ,\kappa ,\ep _1)
$$
for every $j,\iota ,\kappa \in \zz d$. In particular,
\begin{alignat}{2}
\lim _{\ep _1\to 0^+}H(x,\xi ,\eta , \ep _1)
&=
H(x,\xi ,\eta , 0)=0,&
\quad
x,\xi ,\eta &\in \rr d
\notag
\intertext{and}
\lim _{\ep _1\to 0^+}H_1(j,\iota ,\kappa , \ep _1)
&=
H_1(j,\iota ,\kappa , 0)=0,&
\quad
j,\iota ,\kappa &\in \zz d.
\label{Eq:H1Limit}
\end{alignat}

\par

Let $\ep >0$ be fixed, $R,R_1,R_2>0$
be such that $R_2>R_1>2R>0$, and let $r=\min (1,q)$. Then
\eqref{Eq:H1H2Compare} and Minkowski's inequality give
\begin{align}
\nm { &f_{\ep _1,0} - f}
{M^{\infty ,q}_{(\omega )}}^r
\asymp
\nm {V_{\phi_0} (f_{\ep _1,0} - f)\cdot \omega}
{\sfW ^ \infty (\ell^{\infty,q})} ^r
\notag
\\[1ex]
&=
\left (
\sum _{\iota\in \zz d}
\sup_{j\in \zz d} 
\nm {V_{\phi_0} (f_{\ep _1,0} - f)\cdot \omega}
{L^\infty (Q(j,\iota ))} ^q
\right ) ^{\frac rq}
\notag
\\[1ex]
& \leq
\left (
\sum _{\iota\in \zz d}
\sup_{j\in \zz d} 
\underset {(x,\xi )\in Q(j,\iota )} \essup  
\left(
\int_{\rr d}
H(x,\xi ,\eta, \ep _1) \, d\eta
\right)^q
\right ) ^{\frac rq}
\notag
\\[1ex]
& \leq 
\left (\sum _{\iota\in \zz d}
\sup_{j\in \zz d} 
\underset {(x,\xi )\in Q(j,\iota )} \essup 
\left(\sum _{\kappa\in \zz d}
\int_{Q(\kappa )}
H(x,\xi ,\eta, \ep _1)\, d\eta
\right)^q
\right ) ^{\frac rq}
\notag
\\[1ex]
& \leq 
\sum _{\kappa \in \zz d}
\left (
\sum _{\iota \in \zz d}
\sup_{j\in \zz d} 
\underset {(x,\xi )\in Q(j,\iota )} \essup 
\left(
\int_{Q(\kappa )}
H(x,\xi ,\eta, \ep _1)
 \, d\eta
\right)^q
\right ) ^{\frac rq}
\notag
\\[1ex]
& \lesssim
J_{1,R,\ep _1} + J_{2,R,\ep _1}
+
J_{3,R,\ep _1} + J_{4,R,\ep _1}
+
J_{1,R,0} +J_{2,R,0} + J_{3,R,0},
\label{Eq:MultEstCutoff}
\end{align}
where
\begin{align*}
J_{1,R,\ep _1}
&=
\sum _{|\kappa |\ge R}
\left (
\sum _{\iota \in \zz d}
\sup_{j\in \zz d}
H_2(j,\iota ,\kappa , \ep _1)^q
\right ) ^{\frac rq},
\\[1ex]
J_{2,R,\ep _1}
&=
\sum _{|\kappa | \leq R}
\left (
\sum _{|\iota |\ge R_1}
\sup_{j\in \zz d}
H_2(j,\iota ,\kappa , \ep _1)^q
\right ) ^{\frac rq},
\\[1ex]
J_{3,R,\ep _1}
&=
\sum _{|\kappa | \leq R}
\left (
\sum _{\iota \in \zz d}
\sup_{|j| \ge R_2}
H_2(j,\iota ,\kappa , \ep _1)^q
\right ) ^{\frac rq}
\intertext{and}
J_{4,R,\ep _1}
&=
\sum _{|\kappa | \leq R}
\left (
\sum _{|\iota |\le R_1}
\sup _{|j|\le R_2}
H_1(j,\iota ,\kappa , \ep _1)^q
\right ) ^{\frac rq}.
\end{align*}
The limit \eqref{Eq:DensityModSharpCaseProd}
essentially follows if we prove
\begin{alignat}{2}
\lim _{R\to \infty} \left (\sup _{\ep _1\in [0,1]}
J_{n,R,\ep _1} \right )
&= 0,&
\qquad & n= 1,2,3
\label{Eq:FirstJnLimits}
\intertext{and}
\lim _{\ep _1\to 0^+} J_{4,R,\ep _1}
&=0,&
\qquad &\text{for $R,R_1,R_2$ fixed}.
\label{Eq:J4Limits}
\end{alignat}
We shall reach these properties by suitable
estimates, using Lemma \ref{Lemma:SupEstimates}
when $q\le 1$, and Minkowski's inequality otherwise.

\par

First we prove \eqref{Eq:FirstJnLimits}
for $n=1$.
By Lemma \ref{Lemma:SupEstimates} and Minkowski's inequality
we obtain
\begin{align*}
0
\le
J_{1,R,\ep _1}
&=
\sum _{|\kappa |\ge R}
\left (
\sum _{\iota \in \zz d}
\sup_{j\in \zz d}
H_2(j,\iota ,\kappa , \ep _1)^q
\right ) ^{\frac rq},
\\[1ex]
&\le
\sum _{|\kappa | \geq R}
\left (
\sum _{\iota \in \zz d}
\nm {V_{\phi _0}f \cdot \omega }
{L^\infty (\rr d \times Q_2
(\iota -\lfloor \ep _1\kappa \rfloor ))}^q
\right )^{\frac rq}
\nm {\psi}{L^\infty (Q(\kappa ))}^r
\\[1ex]
&=
\left (
\sum _{\iota \in \zz d}
\nm {V_{\phi _0}f \cdot \omega }
{L^\infty (\rr d \times Q_2
(\iota ))}^q
\right )^{\frac rq}
\left (
\sum _{|\kappa | \geq R}
\nm {\psi} {L^\infty (Q(\kappa ))}^r
\right )
\end{align*}
The limit \eqref{Eq:FirstJnLimits} for $n=1$ is now a consequence of Lemma
\ref{Lemma:TailEstimates} and the fact that $\phi \in \Sigma _1(\rr d)$,
which implies that
$\psi \in \sfW ^\infty (\ell ^r(\zz d))$.

\par

Next we prove \eqref{Eq:FirstJnLimits}
for $n=2$. Since $R_1>2R$ and $0\le \ep _1\le 1$, it follows
that $|\iota - \lfloor \ep _1\kappa \rfloor |\ge R$ when $|\iota |\ge R_1$
and $|\kappa |\le R$.
A combination of these estimates and Lemma \ref{Lemma:SupEstimates} gives
\begin{align*}
0
&\le
J_{2,R,\ep _1}
=
\sum _{|\kappa | \leq R}
\left (
\sum _{|\iota |\ge R_1}
\sup_{j\in \zz d}
H_2(j,\iota ,\kappa , \ep _1)^q
\right ) ^{\frac rq},
\\[1ex]
&\le
\sum _{|\kappa | \leq R}
\left (
\sum _{|\iota |\ge R_1}
\nm {V_{\phi _0}f \cdot \omega }
{L^\infty (\rr d \times Q_2
(\iota -\lfloor \ep \kappa \rfloor ))}^q
\right )^{\frac rq}
\nm {\psi} {L^\infty (Q(\kappa ))}^r
\\[1ex]
&\le
\sum _{\kappa \in \zz d}
\left (
\sum _{|\iota |\ge R}
\nm {V_{\phi _0}f \cdot \omega }
{L^\infty (\rr d \times Q_2
(\iota ))}^q
\right )^{\frac rq}
\nm {\psi} {L^\infty (Q(\kappa ))}^r
\\[1ex]
&=
\left (
\sum _{|\iota | \ge R}
\nm {V_{\phi _0}f \cdot \omega }
{L^\infty (\rr d \times Q_2
(\iota ))}^q
\right )^{\frac rq}
\nm {\psi} {\sfW ^\infty (\ell ^r)}^r
\end{align*}
The limit \eqref{Eq:FirstJnLimits} for $n=2$ is now a consequence of Lemma
\ref{Lemma:TailEstimates} and the facts that
$V_{\phi _0}f \cdot \omega \in \sfW ^\infty( \ell ^{\infty ,q}(\zz d))$
and $\psi \in \sfW ^\infty (\ell ^r(\zz d))$.
%

\par

Next we prove \eqref{Eq:FirstJnLimits}
for $n=3$. Since $R_2>R_1+2^d$ and $R_1>2R$ and $0\le \ep _1\le 1$, it follows
that 
$$
\sets {\xi -\ep _1\eta}{\xi \in Q(\iota),\ \eta \in Q(\kappa )}
\subseteq
Q_2(\iota - \lfloor \ep _1\kappa \rfloor )).
$$
%
Hence, for such choices of 
\begin{align*}
H_2(j,\iota ,\kappa ,\ep _1)
&\le
\sup _{x\in Q(j)}\sup _{(\xi ,\eta )\in Q(\iota ,\kappa )}
|F_0(x,\xi -\ep _1\eta )| \nm {\psi} {L^\infty (Q(\kappa ))}
\\[1ex]
&\le
\sup _{x\in Q(j)} 
\nm {F_0(x,\cdo  )}{L^\infty (Q_2(\iota - \lfloor \ep _1\kappa \rfloor ))}
 \nm {\psi} {L^\infty (Q(\kappa ))}
\\[1ex]
&\le
\nm {F_0(x,\cdo  )}{L^\infty (Q_2(j,\iota - \lfloor \ep _1\kappa \rfloor ))}
\nm {\psi} {L^\infty (Q(\kappa ))}
\end{align*}

\par

This gives
\begin{align*}
0
&\le
J_{3,R,\ep _1}
\\[1ex]
&\le
\sum _{\kappa \in \zz d}
\left (
\sum _{\iota \in \zz d}
\sup _{|j|\ge R_2}
\nm {F_0(x,\cdo  )}{L^\infty (Q_2(j,\iota - \lfloor \ep _1\kappa \rfloor ))}^q
\nm {\psi} {L^\infty (Q(\kappa ))}^q
\right )^{\frac rq}
\\[1ex]
&=
\sum _{\kappa \in \zz d}
\left (
\sum _{\iota \in \zz d}
\sup _{|j|\ge R_2}
\nm {F_0(x,\cdo  )}{L^\infty (Q_2(j,\iota))}^q
\right )^{\frac rq}
\nm {\psi} {L^\infty (Q(\kappa ))}^r
\\[1ex]
&=
\left (
\sum _{\iota \in \zz d}
\sup _{|j|\ge R_2}
\nm {F_0(x,\cdo  )}{L^\infty (Q_2(j,\iota ))}^q
\right )^{\frac rq}
\nm {\psi} {\sfW ^\infty (\ell ^r)}^r.
\end{align*}
%
For the first factor on the last line
we observe
\begin{align*}
\sum _{\iota \in \zz d}
\sup _{|j|\ge R_2}
\nm {F_0(x,\cdo  )}{L^\infty (Q_2(j,\iota ))}^q
&\le
\sum _{\iota \in \zz d}
\nm {\chi _{{}_R} F_0}{L^\infty (\rr d\times Q_2(\iota ))}^q
\\[1ex]
&=
\nm {\chi _{{}_R} F_0}{\sfW ^\infty (\omega ,\ell ^{\infty ,q})}^q.
\end{align*}
In the last inequality we have used the facts that $R_2>R_1+2^d$ and
$R_1>2R$. Here $\chi _{{}_R}$ is the same as in the proof of
Theorem \ref{Thm:IndepWind}.

\par

A combination of these estimates gives
$$
0\le J_{3,R,\ep _1}
\lesssim
\nm {\chi _{{}_R} F_0}{\sfW ^\infty (\omega ,\ell ^{\infty ,q})}^r.
$$
An application of Theorem \ref{Thm:IndepWind} and
using that $f\in M^{\sharp ,q}_{(\omega )}(\rr d)$ give
\eqref{Eq:FirstJnLimits} for $n=3$.

\par

It remains to prove \eqref{Eq:J4Limits}. Since the sums in
\eqref{Eq:J4Limits} are finite, and the supremum over
$|j|\le R_2$ is taken over a finite set, \eqref{Eq:H1Limit} gives
\begin{align*}
0 &\le
\lim _{\ep _1\to 0^+} J_{4,R,\ep _1}
\le
\lim _{\ep _1\to 0^+}
\sum _{|\kappa | \leq R}
\left (
\sum _{|\iota |\le R_1}
\sup _{|j|\le R_2}
H_1(j,\iota ,\kappa , \ep _1)^q
\right ) ^{\frac rq}.
\\[1ex]
&=
\sum _{|\kappa | \leq R}
\left (
\sum _{|\iota |\le R_1}
\sup _{|j|\le R_2}
\lim _{\ep _1\to 0^+}
H_1(j,\iota ,\kappa , \ep _1)^q
\right ) ^{\frac rq}
=0
\end{align*}
and \eqref{Eq:J4Limits} follows. 

\par

A combination of \eqref{Eq:MultEstCutoff},
\eqref{Eq:FirstJnLimits} and \eqref{Eq:J4Limits}
now gives
$$
\lim _{\ep _1\to 0^+} \nm { f_{\ep _1,0} - f}
{M^{\infty ,q}_{(\omega )}} =0,
$$
giving the result.
\end{proof}

\par




\par

\section{Compactness for
pseudo-differential on modulation spaces}
\label{sec5}

\par

In this section we make a review of some facts about
pseudo-differential operators and recall continuity
properties for such operators on broad classes
of modulation spaces, given in \cite[Section 6]{ToPfTe}. 
Thereafter we deduce compactness properties
for such pseudo-differential operators with symbols
in $M^{\sharp ,q}_{(\omega )}$ spaces, when acting
on these modulation spaces.

\medspace

First we recall some facts about distribution
kernels of linear continuous operators.
Suppose
$K\in \Sigma _1'(\rr {2d})$. Then $T_K$,
uniquely defined by the formula
\begin{equation}
\label{Eq:OpKernel}
(T_Kf,g)_{L^2(\rr d)}
=
(K,g\otimes \overline f)_{L^2(\rr {2d})},
\qquad f,g\in \Sigma _1(\rr d),
\end{equation}
is a linear continuous operator from
$\Sigma _1(\rr d)$ into $\Sigma _1'(\rr d)$.
On the other hand, by the kernel theorem of
Schwartz, it follows that for any
linear and continuous operator $T$ from
$\Sigma _1(\rr d)$ into $\Sigma _1'(\rr d)$,
there is a kernel $K$ such that
\eqref{Eq:OpKernel} holds for $T=T_K$ (see 
\cite{Tre}).
This shows that there is a one to one
correspondence between $\Sigma _1'(\rr {2d})$
and linear and continuous operators from
$\Sigma _1(\rr d)$ to $\Sigma _1'(\rr d)$.

\par

Next we recall the definition of pseudo-differential
operators. Let $\GL (d,\mathbf R)$ be the set of
$d\times d$-matrices with
entries in $\mathbf R$, $\fka \in \Sigma _1 
(\rr {2d})$, and let $A\in \GL (d,\mathbf R)$
be fixed. Then the
pseudo-differential operator $\op _A(\fka )$
is the linear and continuous operator on
$\Sigma _1 (\rr d)$, given by
\begin{equation}\label{Eq:PseudoDef}
(\op _A(\fka )f)(x)
=
(2\pi  ) ^{-d}\iint \fka (x-A(x-y),\xi )f(y)e^{i\scal {x-y}\xi }\,
dyd\xi , \quad  x\in \rr d
\end{equation}
(see e.g. Chapter XVIII in \cite{Ho1}).
For general $\fka \in \Sigma _1'(\rr {2d})$, the
pseudo-differential operator $\op _A(\fka )$ is defined as the continuous
operator from $\Sigma _1(\rr d)$ to $\Sigma _1'(\rr d)$ with
distribution kernel
\begin{equation}\label{Eq:KernelPsDO}
K_{\fka ,A}(x,y)=(2\pi )^{-d/2}(\mascF _2^{-1}\fka )(x-A(x-y),x-y), 
\quad x,y \in \rr d.
\end{equation}
Here $\mascF _2F$ is the partial Fourier transform of $F(x,y)\in
\Sigma _1'(\rr {2d})$ with respect to the $y$ variable. This
definition makes sense since the mappings
\begin{equation}\label{Eq:HomeoF2tmap}
\mascF _2\quad \text{and}\quad F(x,y)\mapsto F(x-A(x-y),x-y)
\end{equation}
are homeomorphisms on
$\Sigma _s(\rr {2d})$, $\maclS _s(\rr {2d})$,
$\mascS (\rr {2d})$, and their duals, $s\ge 1$.
In particular we have the following.

\par

\begin{lemma}
\label{Lemma:SymbKernelHom}
Let $s\ge 1$, $A\in \GL (d,\mathbf R )$,
and $K_{\fka ,A}$ be as in \eqref{Eq:KernelPsDO}, 
when $\fka \in \Sigma _1'(\rr {2d})$.
Then the map $\fka \mapsto K_{\fka ,A}$ is a homeomorphism on
$\Sigma _1'(\rr {2d})$, which restricts to homeomorphisms on
$$
\Sigma _s(\rr {2d}),
\quad
\maclS _s(\rr {2d}),
\quad
\mascS (\rr {2d}),
\quad
\mascS '(\rr {2d}),
\quad
\maclS _s'(\rr {2d})
\quad \text{and}\quad
\Sigma _s'(\rr {2d}).
$$
\end{lemma}

\par

For symbols and kernels in $\Sigma _1(\rr {2d})$
we stress that corresponding operators
easily become compact, as indicated
in the following result.

\par

\begin{prop}
\label{Prop:Sig1KerCompact}
Let $A\in \GL (\mathbf R,d)$ and
$\fka ,K\in \Sigma _1(\rr {2d})$. Then
$T_K$ and $\op _A(\fka )$ from $\Sigma _1(\rr d)$
to $\Sigma _1'(\rr d)$, extend uniquely to
compact operators
\begin{equation}
\label{Eq:SigmaKernelMaps}
T_K:\Sigma _1'(\rr d)\to \Sigma _1(\rr d)
\quad \text{and}\quad
\op _A(\fka ):\Sigma _1'(\rr d)\to \Sigma _1(\rr d).
\end{equation}
\end{prop}

\par

The uniqueness in Proposition 
\ref{Prop:Sig1KerCompact} follows from \cite{Tre}
(or \eqref{Eq:OpKernel} when $K\in \Sigma _1
(\rr {2d})$ and $f,g\in \Sigma _1'(\rr d)$).
The remaining continuity part follows from
Section 3 and its analysis in \cite{ToKhNiNo}.
In order to be self-contained, we here present
a proof of the continuity part,
based on the fact that the map
\begin{equation}
\label{Eq:WeylProdSigma}
\Sigma _1(\rr {2d})\times \Sigma _1(\rr {2d})
\ni (\fka ,\fkb )\mapsto \fka \wpr _A\fkb
\in 
\Sigma _1(\rr {2d})
\end{equation}
is a continuous surjection.
(See \cite[Theorem 2.5]{ToKhNiNo}.)

\par

\begin{proof}
Since the map $\fka \mapsto K_{\fka ,A}$ is bijective on
$\Sigma _1(\rr {2d})$, it suffices to prove the
result for $\op _A(\fka )$.

\par

By the surjectivity of 
\eqref{Eq:WeylProdSigma}, we have
\begin{align*}
\fka &= \fka _3\wpr _A\fka _2\wpr _A\fka _1,
\quad \text{for some}\quad
\fka _1,\fka _2,\fka _3\in \Sigma _1(\rr {2d}).
\intertext{Then}
\op _A(\fka _1) &: \Sigma _1'(\rr d) \to L^2(\rr d),
\\[1ex]
\op _A(\fka _2) &: L^2(\rr d) \to L^2(\rr d)
\intertext{and}
\op _A(\fka _3) &: L^2(\rr d) \to \Sigma _1(\rr d)
\end{align*}
are continuous. Since
$\fka _2\in \Sigma _1(\rr {2d})
\subseteq L^2(\rr {2d})$, it follows that
$\op _A(\fka _2)$ is a Hilbert-Schmidt operator
and thereby compact. Hence the map $\op _A(\fka)$
in \eqref{Eq:SigmaKernelMaps} factorizes through
the mappings
$$
\op _A(\fka )
=
\op _A(\fka _3)\circ \op _A(\fka _2)
\circ \op _A(\fka _1),
$$
where the first and third operator are continuous,
and the middle operator is compact. This implies that
$\op _A(\fka )$ is compact, giving the result.
\end{proof}

\par

A combination of Propositions 
\ref{Prop:ContAndCompactModSp} (1)
and \ref{Prop:Sig1KerCompact} gives the following.
The details are left for the reader.

\par

\begin{cor}
\label{Cor:Sig1KerCompact}
Let $A\in \GL (\mathbf R,d)$,
$\fka ,K\in \Sigma _1(\rr {2d})$,
$\omega _j\in \mascP _E(\rr {2d})$,
and $\mascB$ be a normal QBF space on $\rr {2d}$.
Then
\begin{equation*}
T_K:M(\omega _1,\mascB )
\to M(\omega _2,\mascB )
\quad \text{and}\quad
\op _A(\fka ):M(\omega _1,\mascB )
\to M(\omega _2,\mascB )
\end{equation*}
are compact.
\end{cor}

\par

The standard (Kohn-Nirenberg) representation, $\fka (x,D)=\op (\fka )$, and
the Weyl quantization $\op ^w(\fka )$ of $\fka$ are obtained by choosing
$A=0$ and $A=\frac 12 I$, respectively, in \eqref{Eq:PseudoDef}
and \eqref{Eq:KernelPsDO}, where $I$ is the identity matrix.

\par

\begin{rem}\label{Rem:BijKernelsOps}
Let $s\ge 1$. By Fourier's inversion formula, \eqref{Eq:KernelPsDO} and the kernel theorem
\cite[Theorem 2.2]{LozPerTask} for operators from
Gelfand-Shilov spaces to their duals,
it follows that the map $\fka \mapsto \op _A(\fka )$ 
is bijective from $\Sigma _s'(\rr {2d})$
to the set of all linear and continuous operators 
from $\Sigma _s(\rr d)$
to $\Sigma _s'(\rr {2d})$. The same holds true
with $\maclS _s$ or $\mascS$ in place of $\Sigma _s$
at each occurrence.
\end{rem}

\par

By Remark \ref{Rem:BijKernelsOps}, it follows that for every $\fka _1\in \Sigma _1'(\rr {2d})$
and $A_1,A_2\in \GL (d,\mathbf R)$, there is a unique $\fka _2\in \Sigma _1'(\rr {2d})$ such that
$\op _{A_1}(\fka _1) = \op _{A_2} (\fka _2)$.
Indeed, the relation between $\fka _1$
and $\fka _2$ is given by
\begin{equation}
\label{Eq:CalculiTransform}
\op _{A_1}(\fka _1) = \op _{A_2}(\fka _2)
\quad \Leftrightarrow \quad
e^{i\scal {A_2D_\xi}{D_x}}\fka _2(x,\xi )=e^{i\scal {A_1D_\xi}{D_x}}\fka _1(x,\xi )
\end{equation}
(see e.g. \cite[Section 18.5]{Ho1}
or \cite[Proposition 1.1]{Toft20}).
Here we note that the operator
$e^{i\scal {AD_\xi}{D_x}}$ is homeomorphic
on $\Sigma _s(\rr {2d})$, $\maclS _s(\rr {2d})$,
$\mascS (\rr {2d})$, and their duals
(cf. \cite{CaTo,Tr}).

\par

We also recall that $\op _A(\fka )$ is a rank-one operator, i.e.
\begin{equation}\label{Eq:RankOneSymb}
\op _A(\fka )f=(2\pi )^{-\frac d2}
(f,f_2)f_1, \qquad f\in \Sigma _1(\rr d),
\end{equation}
for some $f_1,f_2\in \Sigma _1'(\rr d)$,
if and only if $\fka$ is equal to the \emph{$A$-Wigner distribution}
\begin{equation}\label{Eq:WignerAdef}
W_{f_1,f_2}^{A}(x,\xi ) \equiv \mascF (f_1(x+A\cdo
)\overline{f_2(x-(I-A)\cdo )} )(\xi ),
\quad x,\xi \in \rr d, 
\end{equation}
of $f_1$ and $f_2$. If in addition $f_1,f_2\in L^2(\rr d)$, then $W_{f_1,f_2}^{A}$
takes the form
\begin{equation}\label{wignertdef2}
W_{f_1,f_2}^{A}(x,\xi ) =
(2\pi )^{-\frac d2}
\int _{\rr d} f_1(x+Ay)
\overline{f_2(x-(I-A)y)}e^{-i\scal y\xi} \, dy,
\quad x,\xi \in \rr d.
\end{equation}
(Cf. e.g. \cite{Toft20}.) Since
the Weyl case is peculiar interesting,
we also set $W_{f_1,f_2}=W_{f_1,f_2}^{A}$
when $A=\frac 12 I$. A straight-forward combination of
\eqref{Eq:CalculiTransform} and the fact that $\fka$ in \eqref{Eq:RankOneSymb}
equals $W_{f_1,f_2}^{A}$ gives
$$
e^{i\scal {A_2D_\xi}{D_x}}W_{f_1,f_2}^{A_2}
=
e^{i\scal {A_1D_\xi}{D_x}}W_{f_1,f_2}^{A_1}.
$$

\par


\par

The first part of the next result follows from
\cite[Theorem 6.4]{ToPfTe}.
Here the involved weight functions should be related
as
\begin{multline}
\label{Eq:WeightPseudoRel}
\frac {\omega _2(x,\xi )v_0(x-y,\xi -\eta )}
{\omega _1(y,\eta)}
\lesssim
\omega _0(x+A(y-x),\eta +A^*(\xi -\eta ),\xi -\eta ,y-x),
\\[1ex]
x,y,\xi ,\eta \in \rr d.
\end{multline}
For any operator $T$ from the quasi-Banach space
$\mascB _1$ to the quasi-Banach space $\mascB _2$,
we let $\nm T{\mascB _1\to \mascB _2}$ be the
quasi-norm of $T$. That is,
\begin{equation}
\label{Eq:OpNorm}
\nm T{\mascB _1\to \mascB _2}
\equiv
\sup \nm {Tf}{\mascB _2},
\end{equation}
where the supremum is taken over all
$f\in \mascB _1$ such that $\nm f{\mascB _1}\le 1$.
For convenience we set
\begin{equation}
\label{Eq:L2Norm}
\nmm \cdo
\equiv
\nm {\cdo}{L^2(\rr d)\to L^2(\rr d)}.
\end{equation}

\par

\begin{thm}
\label{Thm:PseudoCont2}
Suppose $A\in \GL (d,\mathbf R)$, $\mascB$ is a normal
QBF space on $\rr {2d}$ with respect to $r_0\in (0,1]$ and
$v_0\in \mascP _E(\rr {2d})$,
$\omega _0\in \mascP _E(\rr {4d})$,
and $\omega _1,\omega _2\in \mascP _E(\rr {2d})$, be such that
\eqref{Eq:WeightPseudoRel}$'$ holds. Also let
$\fka \in M^{\infty ,r_0}_{(\omega _0)}(\rr {2d})$.
Then the the map
\begin{align}
\op _A(\fka ) :
M(\omega _1,\mascB ) &\to M(\omega _2,\mascB )
\label{Eq:PseudoCont2A}
\intertext{is continuous, and}
\nm {\op _A(\fka )}
{M(\omega _1,\mascB ) \to M(\omega _2,\mascB )}
&\le
C\nm {\fka}{M^{\infty ,r_0}_{(\omega _0)}},
\label{Eq:PseudoCont2}
\end{align}
for some constant $C>0$
If, more restrictive,
$\fka \in M^{\sharp ,r_0}_{(\omega _0)}(\rr {2d})$,
then the map \eqref{Eq:PseudoCont2A} is compact.
\end{thm}




\par

\begin{proof}
The continuity of the map \eqref{Eq:PseudoCont2A}
follows immediately from
\cite[Theorem 6.4]{ToPfTe}.

\par

Let $\ep >0$ be arbitrary and suppose that
$\fka \in M^{\sharp ,r_0}_{(\omega _0)}(\rr {2d})$.
Then Corollary \ref{Cor:MsharpCompl} and 
\eqref{Eq:PseudoCont2} show that there is
an $\fka _\ep \in \Sigma _1(\rr {2d})$
such that
$$
\nm {\op _A(\fka )-\op _A(\fka _\ep )}
{M(\omega _1,\mascB ) \to M(\omega _2,\mascB )}
<\ep .
$$
Hence $\op _A(\fka )$ may in norm be approximated
with arbitrary precision
by operators with symbols in $\Sigma _1(\rr {2d})$.
Since $\op _A(\fka _\ep )$ from
$M(\omega _1,\mascB )$ to $M(\omega _2,\mascB )$
is compact, in view of Corollary
\ref{Cor:Sig1KerCompact}, it also follows that
$\op _A(\fka )$ is compact, and the result follows.
\end{proof}


\par

\par

\section{Compactness properties for
Toeplitz operators}
\label{sec7}

\par

In this section we first review some facts about
Toeplitz operators. Thereafter we deduce
continuity for Toeplitz operators with symbols
in suitable modulation spaces, when acting 
between other modulation spaces. As in
\cite{CorGro,Toft04,GroTof1,Toft10,ToBo08}, 
the main idea is to
express the Toeplitz operators as
pseudo-differential operators
with symbols being a convolution
between the Toeplitz symbol
and a convenient function.

\par

%

Let $\fka \in \Sigma _1 (\rr {2d})$. 
Then the Toeplitz 
operator $\tp _{\phi _1,\phi _2}(\fka)$,
with symbol $\fka$, and window functions 
$\phi _1,\phi _2\in
\Sigma _1 (\rr d)$, is the continuous
operator on $\Sigma _1(\rr d)$, defined by
the formula
\begin{equation}\label{Eq:ToepDef}
(\tp _{\phi _1,\phi _2}(\fka )f_1,f_2)
_{L^2(\rr d)}
= 
(\fka \cdot V_{\phi _1}f_1,
V_{\phi _2}f_2)_{L^2(\rr {2d})}
\end{equation}
when $f_1,f_2\in \Sigma _1 (\rr d)$. That is,
$\tp _{\phi _1,\phi _2}(\fka )$ is the operator, given by
\begin{equation}\label{Eq:ToepDefAlt}
\tp _{\phi _1,\phi _2}(\fka )f
= 
V_{\phi _2}^*(\fka \cdot V_{\phi _1}f).
\end{equation}
The definition of $\tp _{\phi _1,\phi _2}(\fka )$
extends uniquely to a continuous operator from
$\Sigma _1'(\rr d)$ to $\Sigma _1(\rr d)$.
For convenience we write
$\tp _{\phi _1}(\fka )$ 
for
$\tp _{\phi _1,\phi _1}(\fka )$.

\par

We may interpret the operator in 
\eqref{Eq:ToepDef}
as a pseudo-differential operator, using
the fact that
\begin{equation}\label{Eq:ToeplWeyl}
\begin{aligned}
\operatorname{Tp}_{\phi _1,\phi _2}(\fka)
&=
\op ^w(\fka \ast u)
\quad \text{when}\quad 
u = (2\pi)^{-\frac d2}
\check W_{\phi _2,\phi _1},
\end{aligned}
\end{equation}
which follows by straight-forward application of
Fourier's inversion formula. Here
$\check F(x,\xi )=F(-x,-\xi )$,
when $F$ is a function or (ultra-)distribution
on $\rr {2d}$. By
\eqref{Eq:ToepDef} and \eqref{Eq:ToeplWeyl},
the mappings
\begin{equation}
\label{Eq:ToeplMapSymbWind}
(\fka ,\phi _1,\phi _2)\mapsto \tp
_{\phi _1,\phi _2}(\fka )
\quad \text{and}\quad
(\fka ,\phi _1,\phi _2)
\mapsto 
\fka *\check W_{\phi _2,\phi _1}
\end{equation}
from $\Sigma _1(\rr {2d})\times
\Sigma _1(\rr d)\times \Sigma _1(\rr d)$
to $\maclL (\Sigma _1(\rr d),\Sigma _1(\rr d))$
and $\Sigma _1(\rr {2d})$, respectively,
are equivalent.
These mappings 
extend in several ways, e.g. from
$\Sigma _1'(\rr {2d})\times
\Sigma _1(\rr d)\times \Sigma _1(\rr d)$
to $\maclL (\Sigma _1(\rr d),\Sigma _1(\rr d))$
and $\Sigma _1'(\rr {2d})$, respectively.
(Cf. e.g. 
\cite{CorGro,Fol,GroTof1,Toft04,Toft20,ToBo08}.)

\par

Next we discuss issues on continuity and compactness
for other extensions of \eqref{Eq:ToeplMapSymbWind},
where the involved
window functions $\phi _j$ and the symbol $\fka$
belong to suitable Lebesgue and modulation
spaces. For a relaxed start we first consider
the case when the symbols
belong to subclasses of $L^\infty (\rr {2d})$.
Here recall \eqref{Eq:OpNorm} and \eqref{Eq:L2Norm}
for notations.

\par

\begin{prop}
\label{Prop:ToeplContLebSpSymb}
Let $\phi _1,\phi _2\in L^2(\rr d)$. Then the
following is true:
\begin{enumerate}
\item
if $\fka \in L^\infty (\rr {2d})$, then
$\tp _{\phi _1,\phi _2}(\fka )$ is continuous
on $L^2(\rr d)$. Furthermore,
\begin{equation}
\label{Eq:ToeplEstLebSymb}
\begin{aligned}
\nmm {\tp _{\phi _1,\phi _2}(\fka )}
&\le
\nm {\fka}{L^\infty}\nm {\phi _1}{L^2}
\nm {\phi _2}{L^2},
\\[1ex]
\fka &\in L^\infty (\rr {2d}),\ 
\phi _1,\phi _2\in L^2(\rr d)
\text ;
\end{aligned}
\end{equation}

\item
if $\fka \in L^\sharp (\rr {2d})$, then
$\tp _{\phi _1,\phi _2}(\fka )$ is compact
on $L^2(\rr d)$.
\end{enumerate}
\end{prop}

\par

Evidently, for the symbol classes in
Proposition \ref{Prop:ToeplContLebSpSymb}
we have
$L^\sharp (\rr {2d})
\subseteq
L^\infty (\rr {2d})$.

\par

\begin{proof}
The assertion (1) and its proof are given
in e.{\,}g. \cite{Fol,HeWong,Toft04}.

\par

The assertion (2) essentially also
follows from these facts. In fact,
suppose that $\fka \in L^\sharp (\rr {2d})$,
and let $\ep >0$ be arbitrary. Choose
$R>0$ such that
$$
|\fka (X)|
\le
\ep (1+\nm {\phi _1}{L^2})^{-1}
(1+\nm {\phi _2}{L^2})^{-1},
\qquad
|X|\ge R.
$$
If
$$
\fka _1(X)
=
\begin{cases}
\fka (X), & |X|<R,
\\[1ex]
0, & |X|\ge R,
\end{cases}
\quad \text{and}\quad
\fka _2(X)
=
\begin{cases}
0, & |X|<R,
\\[1ex]
\fka (X), & |X|\ge R.
\end{cases}
$$
Then $\fka -\fka _1=\fka _2$, which gives
$$
\tp _{\phi _1,\phi _2}(\fka )
-\tp _{\phi _1,\phi _2}(\fka _1)
=
\tp _{\phi _1,\phi _2}(\fka _2).
$$
By \eqref{Eq:ToeplEstLebSymb} we get
$$
\nmm {\tp _{\phi _1,\phi _2}(\fka )
-
\tp _{\phi _1,\phi _2}(\fka _1)}
=
\nmm {\tp _{\phi _1,\phi _2}(\fka _2)}
\le
\nm {\fka _2}{L^\infty}\nm {\phi _1}{L^2}
\nm {\phi _2}{L^2}
<\ep .
$$
Furthermore, since $\fka _1\in L^1(\rr {2d})$,
it follows that
$\tp _{\phi _1,\phi _2}(\fka _1)$ is
a trace-class operator, and thereby
a compact operator, in view of
e.{\,}g. \cite{Fol,HeWong,Toft04}.
Consequently, we may approximate
$\tp _{\phi _1,\phi _2}(\fka )$
in norm by compact operators, which
implies that also $\tp _{\phi _1,\phi _2}(\fka )$
is compact. This gives the result.
\end{proof}




\par

We shall combine \eqref{Eq:ToeplWeyl}
with Proposition \ref{Prop:ModConv}
and Theorem \ref{Thm:PseudoCont2} 
in various contexts to obtain compactness
for such operators. We observe that
basic continuity properties of such 
operators are guaranteed by
\cite[Theorem 6.3]{ToPfTe}.
Since the mappings in \eqref{Eq:ToeplMapSymbWind}
are strongly linked to \eqref{Eq:ToeplWeyl}, it is 
important to search for continuity properties of
these mappings. In the following theorem we find such 
continuity estimates in the background of suitable 
modulation spaces.
Here the involved Lebesgue exponents and
weight functions are related to each others
as
\begin{equation}
\label{Eq:LebExpToeplToPseudo}
1\le \frac 1r\le \frac 1{r_0}
=
\frac 1{q}+\frac 1{r},
\quad \text{or}\quad
\frac 12 \le \frac 1r= \frac 1{r_0}
=1-\frac 1{2q},
\end{equation}
and
\begin{equation}
\label{Eq:WeylToeplWeightCondClass}
\begin{aligned}
\omega_0(x_1-x_2, &\xi_1-\xi_2,\eta ,y)
\\[1ex]
&\lesssim 
\omega (x_1,\xi_1,\eta ,y)
\cdot \vartheta _1 (x_2 -{\textstyle{\frac y2}}
,\xi_2 + {\textstyle{\frac \eta 2}})
\cdot \vartheta _2 (x_2+{\textstyle{\frac y2}} ,
\xi_2 - {\textstyle{\frac \eta 2}}),
\\[1ex]
&x_1,x_2,y,\xi _1,\xi _2,\eta \in \rr d.
\end{aligned}
\end{equation}


\par

\begin{thm}
\label{Thm:ToeplPseudoSymbClassMod}
Let $q,r_0,r\in (0,\infty ]$,
$\omega _0,\omega
\in \mascP _E(\rr {4d})$
and $\vartheta _1,\vartheta _2
\in \mascP _E(\rr {2d})$, be such that
\eqref{Eq:LebExpToeplToPseudo}
and
\eqref{Eq:WeylToeplWeightCondClass} holds.
Then the following is true:
\begin{enumerate}
\item the first map in \eqref{Eq:ToeplMapSymbWind}
is continuous
from $M^{\infty ,q}_{(\omega )}(\rr {2d})
\times M^{r}_{(\vartheta _1)}(\rr d)
\times M^{r}_{(\vartheta _2)}(\rr d)$
to $\op ^w(M^{\infty ,r_0}_{(\omega _0)}
(\rr {2d}))$;

\vrum

\item the first map in \eqref{Eq:ToeplMapSymbWind}
is continuous 
from $M^{\sharp ,q}_{(\omega )}(\rr {2d})
\times M^{r}_{(\vartheta _1)}(\rr d)
\times M^{r}_{(\vartheta _2)}(\rr d)$
to $\op ^w(M^{\sharp ,r_0}_{(\omega _0)}
(\rr {2d}))$.
\end{enumerate}
\end{thm}

\par

Because of the equivalences between the
mappings in \eqref{Eq:ToeplMapSymbWind},
it follows that the previous theorem
is equivalent with the following lemma.

\par

\begin{lemma}
\label{Lemma:ToeplPseudoSymbClassMod}
Let $q,r_0,r\in (0,\infty ]$,
$\omega _0,\omega
\in \mascP _E(\rr {4d})$
and $\vartheta _1,\vartheta _2
\in \mascP _E(\rr {2d})$, be such that
\eqref{Eq:LebExpToeplToPseudo}
and
\eqref{Eq:WeylToeplWeightCondClass} holds.
Then the following is true:
\begin{enumerate}
\item the second map in \eqref{Eq:ToeplMapSymbWind}
is continuous
from $M^{\infty ,q}_{(\omega )}(\rr {2d})
\times M^{r}_{(\vartheta _1)}(\rr d)
\times M^{r}_{(\vartheta _2)}(\rr d)$
to $M^{\infty ,r_0}_{(\omega _0)}
(\rr {2d})$. Furthermore,
\begin{equation}
\label{Eq:ToeplPseudoSymbClassMod}
\nm {\fka *\check W_{\phi _2,\phi _1}}
{M^{\infty ,r_0}_{(\omega _0)}}
\le
C \nm {\fka}
{M^{\infty ,r_0}_{(\omega )}}
\nm {\phi _1}{M^r_{(\vartheta _1)}}
\nm {\phi _2}{M^r_{(\vartheta _2)}},
\end{equation}
for some constant $C>0$ which is independent
of $\fka \in M^{\infty ,r_0}_{(\omega )}(\rr {2d})$,
$\phi _1\in M^r_{(\vartheta _1)}(\rr d)$ and
$\phi _2\in M^r_{(\vartheta _2)}(\rr d)$.

\vrum

\item the second map in \eqref{Eq:ToeplMapSymbWind}
is continuous 
from $M^{\sharp ,q}_{(\omega )}(\rr {2d})
\times M^{r}_{(\vartheta _1)}(\rr d)
\times M^{r}_{(\vartheta _2)}(\rr d)$
to $M^{\sharp ,r_0}_{(\omega _0)}
(\rr {2d})$.
\end{enumerate}
\end{lemma}

\par

\begin{proof}
The assertion (1) follows from
\cite[Theorem 7.4]{ToPfTe} and its proof.
The details are left for the reader.

\par

Suppose that
$\fka \in M^{\sharp ,q}_{(\omega )}(\rr {2d})$,
$\phi _1\in M^{r}_{(\vartheta _1)}(\rr d)$
and
$\phi _2\in M^{r}_{(\vartheta _2)}(\rr d)$.
Then $\fkb \equiv \fka *\check W_{\phi _2,\phi _1}$
is well-defined as an element in
$M^{\infty ,q}_{(\omega _0)}(\rr {2d})$,
in view of (1). We need to show that
$\fkb \in M^{\infty ,q}_{(\omega _0)}(\rr {2d})$.
We have that $\Sigma _1(\rr {2d})$ is dense
in $M^{\sharp ,q}_{(\omega )}(\rr {2d})$ in view
of Corollary \ref{Cor:MsharpCompl}. Furthermore,
since $r<\infty$, it also follows that
$\Sigma _1(\rr d)$ is dense in
$M^{r}_{(\vartheta _j)}(\rr d)$, $j=1,2$.

\par

Now choose $\fka _k\in \Sigma _1(\rr {2d})$
and $\phi _{j,k}\in \Sigma _1(\rr d)$, $j=1,2$,
$k\in \mathbf Z_+$ such that
$$
\lim _{k\to \infty} \nm {\fka -\fka _k}
{M^{\infty ,r_0}_{(\omega )}}=
\lim _{k\to \infty}
\nm {\phi _j-\phi _{j,k}}{M^r_{(\vartheta _j)}}
=0,
\quad j=1,2.
$$
If
$$
\fkb _k= \fka _k*\check W_{\phi _{2,k},\phi _{1,k}},
\quad k\in \mathbf Z_+,
$$
then it follows from
\eqref{Eq:ToeplPseudoSymbClassMod}, and the limits
above, that
$$
\lim _{k\to \infty} \nm {\fkb -\fkb _k}
{M^{\infty ,r_0}_{(\omega _0)}}=0.
$$
Since $\fkb _k\in \Sigma _1(\rr {2d})$, it follows
that $\fkb$ be approximated by elements in
$\Sigma _1(\rr {2d})$ with arbitrary precision.
By applying Corollary \ref{Cor:MsharpCompl}
again it follows that
$\fkb \in M^{\infty ,r_0}_{(\omega _0)}(\rr {2d})$,
and the result follows.
\end{proof}

\par

\begin{thm}
\label{Thm:GenContResForToepl}
Let $p,q,r_0,r\in (0,\infty ]$ be such that
\eqref{Eq:LebExpToeplToPseudo} holds, and suppose
$
\mascB
$
is an invariant QBF space on $\rr {2d}$
with respect to 
$r_0$
and 
$v_0 \in \mascP_E(\rr {2d} )$,
$\omega \in \mascP _E (\rr {4d})$,
and
$\omega _1, \omega _2, \vartheta_1, \vartheta _2
\in \mascP _E (\rr {2d})$, 
are such that
\begin{multline}
\label{Eq:ToepOpCond2}
    \frac {\omega _2 (x-z, \xi - \zeta)} 
        {\omega _1 (x-y, \xi-\eta)}
    \cdot
    v_0(y-z, \eta-\zeta)
\\[1ex]
    \lesssim
    \omega (x,\xi, \eta-\zeta, z-y)
    \vartheta _1 (y,\eta )
    \vartheta _2 (z, \zeta)
\end{multline}
for $x,y, z, \xi, \eta, \zeta \in \rr d$. 
Also suppose
$
\phi _j \in 
M ^{r} _{(\vartheta _j)} (\rr d)$, $j =1,2$.
Then the following is true:
\begin{enumerate}
\item if $\fka \in M ^{\infty ,q} _{(\omega)}( \rr {2d})$,
then 
\begin{align}
\tp _{\phi _1,\phi _2}(\fka) &:
M(\omega _1,\mascB )
\to
M(\omega _2,\mascB )
\label{Eq:ToeplitzContA}
\intertext{is continuous, and}
\nm {\tp _{\phi _1,\phi _2}(\fka)f}
{M(\omega _2,\mascB )}
&\le
C\nm {\fka}{M^{\infty ,q}_{(\omega)}}
\nm f{M(\omega _1,\mascB )}
\nm {\phi _1}{M^r_{(\vartheta _1)}}
\nm {\phi _2}{M^r_{(\vartheta _2)}},
\notag
\\[1ex]
\fka \in M^{\infty ,q}_{(\omega )}(\rr {2d}),\ 
f&\in M(\omega _1,\mascB ),\ 
\phi _j\in M^r_{(\vartheta _j)}(\rr d),\ j=1,2,
\label{Eq:ToeplitzCont}
\end{align}
for some constant $C>0$, which only depends on
$\mascB$, $\omega$, and $\omega _j$, $j=1,2$;

\vrum

\item
if $\fka \in M ^{\sharp ,q} _{(\omega)}( \rr {2d})$,
then the map \eqref{Eq:ToeplitzContA} is compact.
\end{enumerate}
\end{thm}

\par

\begin{proof}
The assertion (1) follows from \cite[Theorem 7.6]{ToPfTe}.

\par

Suppose that
$\fka \in M ^{\sharp ,q} _{(\omega)}( \rr {2d})$,
and let $\fka _k$, $\fkb$, $\fkb _k$ and
$\phi _{j,k}$ as in the proof of Lemma
\ref{Lemma:ToeplPseudoSymbClassMod}. Then
\eqref{Eq:ToeplitzCont} shows that
\begin{align*}
\lim _{k\to \infty}&\nm {\tp _{\phi _1,\phi _2}(\fka )
-\tp _{\phi _{1,k},\phi _{2,k}}(\fka _k)}
{M(\omega _1,\mascB )\to M(\omega _2,\mascB )}
\\[1ex]
&=
\lim _{k\to \infty}\nm {\op ^w(\fkb )
-\op ^w(\fkb _k)}
{M(\omega _1,\mascB )\to M(\omega _2,\mascB )}
=0.
\end{align*}
Since $\fkb _k = \fka _k*\check W_{\phi _{2,k},\phi _{1,k}}
\in \Sigma _1(\rr {2d})$, it follows that
$$
\tp _{\phi _{1,k},\phi _{2,k}}(\fka _k)
=
(2\pi )^{-\frac d2}\op ^w(\fkb _k)
$$
from $M(\omega _1,\mascB )$ to $M(\omega _2,\mascB )$
is compact. Hence, due to the limits above it follows
that the map \eqref{Eq:ToeplitzContA} can be
approximated by arbitrary precision with compact operators.
This implies that $\tp _{\phi _1,\phi _2}(\fka)$
is compact, and the result follows.
\end{proof}

\par





\par

\par

\end{document}